\documentclass{article}

\usepackage[margin=1in]{geometry}
\usepackage{amsmath, mathtools, amsfonts, mathrsfs, amsthm, amssymb, cite, suffix, enumitem, bm, graphbox}
\mathtoolsset{showonlyrefs}
\usepackage[colorlinks]{hyperref}
\hypersetup{linkcolor=blue}
\usepackage{xcolor}

\usepackage{multirow}
\usepackage{caption}

\newtheorem{theorem}{Theorem}[section]
\newtheorem{lemma}[theorem]{Lemma}

\newtheorem{proposition}[theorem]{Proposition}

\theoremstyle{definition}
\newtheorem{definition}[theorem]{Definition}
\newtheorem{remark}[theorem]{Remark}

\usepackage{mleftright}
\renewcommand{\left}{\mleft}
\renewcommand{\right}{\mright}

\renewcommand{\leq}{\leqslant}
\renewcommand{\geq}{\geqslant}

\renewcommand{\emptyset}{\varnothing}

\renewcommand{\P}{\mathbb{P}}
\newcommand{\E}{\mathbb{E}}
\newcommand{\Z}{\mathbb{Z}}
\newcommand{\R}{\mathbb{R}}
\newcommand{\N}{\mathbb{N}}

\newcommand{\cN}{\mathcal{N}}
\newcommand{\cL}{\mathcal{L}}

\renewcommand{\c}{\mathsf{c}}

\newcommand{\Var}{\operatorname{Var}}

\newcommand{\Cov}{\operatorname{Cov}}
\newcommand{\TV}{\mathrm{TV}}
\renewcommand{\d}{\operatorname{\mathbf{d}}}
\newcommand{\ind}[1]{\mathbf{1}_{\{#1\}}}

\newcommand{\eps}{\varepsilon}
\newcommand{\Exp}[1]{\exp\left(#1\right)}

\WithSuffix\newcommand\ind*[1]{\mathbf{1}_{#1}}

\newcommand{\dtv}{\d_\TV}

\renewcommand{\tilde}{\widetilde}

\newcommand{\treebn}{{\mathcal{T}_{b,n}}}

\newcommand{\mubno}{{\mathcal{M}_{b,n}^0}}
\newcommand{\mubnh}{{\mathcal{M}_{b,n}^\ham}}
\newcommand{\mubnhr}{{\mathcal{M}_{b,n}^{\hamr}}}
\newcommand{\mubnhbl}{{\mathcal{M}_{b,n}^{\hambl}}}
\newcommand{\mubnhal}{{\mathcal{M}_{b,n}^{\hamal}}}

\newcommand{\walk}{\mathcal{W}}

\newcommand{\bb}{\mathfrak{b}}
\newcommand{\mm}{\mathfrak{m}}

\newcommand{\ham}{\mathscr{H}}
\newcommand{\rem}{\mathscr{R}}
\newcommand{\pot}{\mathscr{V}}

\newcommand{\hamr}{\ham_r}
\newcommand{\hambl}{\ham_{\lambda,\beta}}

\newcommand{\hamal}{\ham_\lambda^\alpha}

\newcommand{\qua}{\mathscr{A}}
\newcommand{\ext}{\mathscr{B}}

\newcommand{\trem}{\tilde{\rem}}
\newcommand{\twalk}{\tilde{\walk}}

\newcommand{\oud}{\mathcal{U}}

\newcommand{\plat}{\mathsf{Plateau}}

\title{Renormalizing small ball events for branching random walk}
\author{Vilas Winstein}

\begin{document}

\maketitle

\begin{abstract}
Consider a branching random walk (BRW) of depth $n$ with standard Gaussian
increments, conditioned on all endpoints lying in an interval of radius $r$
(which may depend on $n$), i.e.\ a \emph{small ball event}. We prove that this
conditioning results in exponential decay of correlations between the endpoints,
with correlation length $O(r)$. Our proof is based on the \emph{renormalization group},
which in this context captures the effect of the endpoint conditioning on
earlier steps in the walk. This effect can be separated into so-called \emph{relevant}
and \emph{irrelevant} parts, and in our setting the relevant part makes the conditioned
BRW behave like a branching Ornstein-Uhlenbeck walk for the first $n-O(r)$
steps, leading to the correlation decay. The irrelevant part shrinks in a
nontrivial manner, and the simplest argument to bound it uniformly only yields
correlation length $O(r^2)$. To achieve the optimal $O(r)$ bound, we establish
that the irrelevant part shrinks horizontally in space before decaying
uniformly.

The BRW is a prototypical example of a log-correlated Gaussian
field, and it forms the basis of many \emph{hierarchical field theories} in the
mathematical formulation of quantum field theory. As such, we present our
results in a general form which is applicable to other hierarchical field
theories with confining potentials, of which the small ball conditioned BRW is
an example. We apply our general result to two other models: first, the
\emph{sinh-Gordon} model which has a strictly convex potential, so at the level of
covariances the model behaves like one with a quadratic potential, i.e.\ the
model has a positive mass. Second, we consider generalizations of the
\emph{$\phi^4$ model} with any power $\alpha \geq 2$; as $\alpha$ increases from $2$ to
$\infty$, these power potentials naturally interpolate between the massive
(quadratic) potential and the infinite square well potential of the small ball
event.
\end{abstract}

\vfill

\begin{center}
    \captionsetup{type=figure}
    \includegraphics[width=\textwidth]{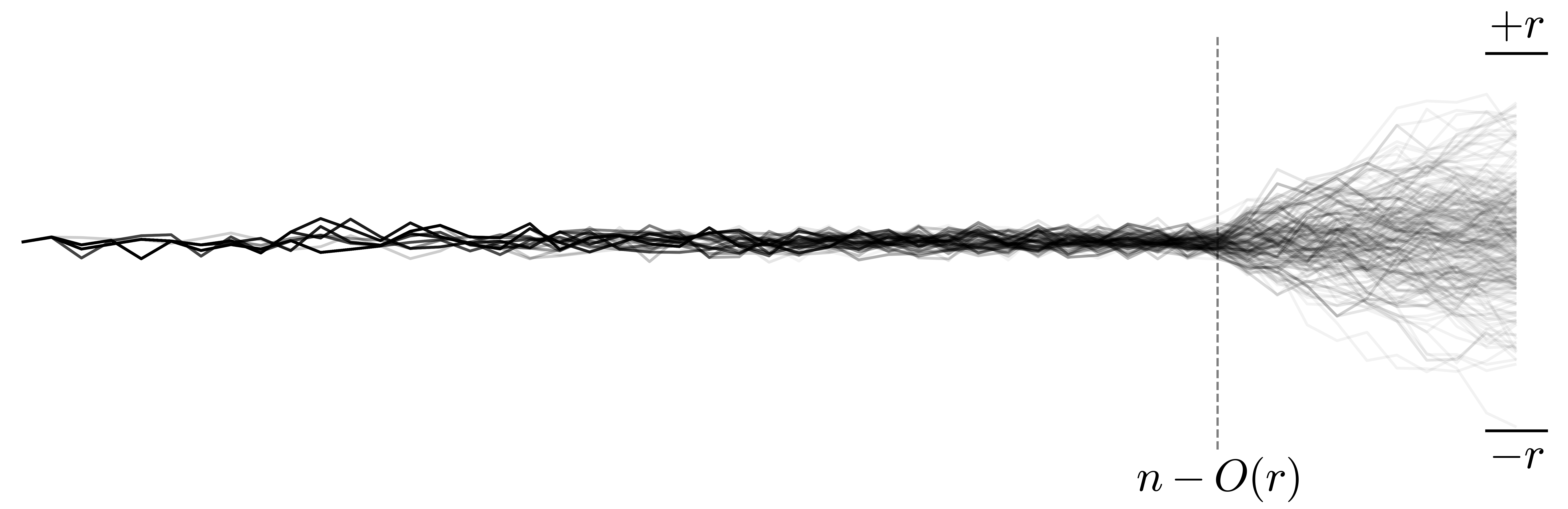}
    \caption{
    Illustrative depiction of the behavior of branching random walk conditioned
    on all endpoints lying in the interval $[-r,r]$. The initial $n-O(r)$ steps become
    a branching Ornstein--Uhlenbeck walk with a drift towards the origin.
    The drift strength (and thus fluctuation scale) of this portion of the walk
    only depends on the branching rate of the walk, not the radius $r$ of the interval.
    }
\label{fig:frontpage}
\end{center}

\clearpage

\setcounter{tocdepth}{2}
\tableofcontents

\vfill

\section{Introduction}
\label{sec:intro}

In this article we study the branching random walk.
As a prototypical example of a log-correlated field, it is
often used as a stand-in for the two-dimensional discrete Gaussian free field
upon which a host of physical phenomena are modeled.
However, to keep the exposition broadly accessible, before commenting further on
physical motivations, we lay out the elementary setup and state our main result.

\subsection{Basic setup and main result}
\label{sec:intro_setupresults}

For $b \in \{ 2, 3, 4, \dotsc \}$, the $b$-branching random walk (BRW) of depth $n$ is a
random walk indexed by $\treebn$, the complete $b$-ary tree of depth $n$.
We will identify $\treebn$ with the set of $b$-ary strings $\tau = (\tau_1,\dotsc,\tau_k)$
of length $k \leq n$, i.e.\
\begin{equation}
    \treebn = \bigcup_{k=0}^n [b]^k,
\end{equation}
where $[b] = \{1,\dotsc,b\}$.
Note that $[b]^0$ consists of a unique element, the \emph{empty string} $\emptyset$.
All strings in $[b]^k$ have \emph{length} $k$; the length of $\tau \in \treebn$
is denoted by $|\tau|$.
Additionally, for $\tau, \tau' \in \treebn$ with $|\tau'| \leq |\tau|$,
we say that $\tau'$ is an \emph{ancestor} of $\tau$,
denoted by $\tau' \preceq \tau$, if $\tau_j' = \tau_j$ for all $j \in \{1, \dotsc, |\tau'|\}$.

To construct the BRW $\left( \phi_\tau \right)_{\tau \in \treebn}$, we first define its increments
$\left( \zeta_\tau \right)_{\tau \in \treebn \setminus \{ \emptyset \}}$ to be mutually independent standard Gaussian random
variables, and then simply set
\begin{equation}
\label{eq:def_brw}
    \phi_\tau = \sum_{\emptyset \neq \tau' \preceq \tau} \zeta_{\tau'}.
\end{equation}
Note that with this definition, the root value $\phi_\emptyset$ is always zero.
We will often denote by $\Phi$ the set of endpoints of the branching random walk,
i.e.\ $\Phi = \left( \phi_\tau \right)_{\tau \in [b]^n}$.
Now we consider $\Phi$ conditioned on an $\infty$-ball of radius $r$; in other words,
we condition on the event
\begin{equation}
    \{ \| \Phi \|_\infty \leq r \} = \{ |\phi_\tau| \leq r \text{ for all } \tau \in [b]^n \}.
\end{equation}
Our main result concerns the correlation structure of $\Phi$ under this conditioning.
For this, we introduce the \emph{distance} $\d(\tau_1,\tau_2)$ between two strings
$\tau_1,\tau_2 \in [b]^n$, defined as $n - |\tau'|$, where $\tau'$ is the longest
common ancestor of $\tau_1$ and $\tau_2$.

\begin{theorem}
\label{thm:main_brw} 
There are constants $C, c > 0$ depending only on $b$ such that
for all $\tau_1, \tau_2 \in [b]^n$, we have
\begin{equation}
    \left| \Cov \left[ \, \phi_{\tau_1}, \phi_{\tau_2} \, \big| \,
    \|\Phi\|_\infty \leq r \right] \right|
    \leq C \Exp{- c \d(\tau_1,\tau_2) + C r}.
\end{equation}
\end{theorem}

Note that the covariance without conditioning is exactly $n - \d(\tau_1,\tau_2)$.
Since $\| \Phi \|_\infty$ is typically of order $n$ with a very light lower tail
\cite{A13Convergence,CH20Lower},
we are conditioning on an extremely low-probability event if $r \ll n$.
Thus, it is to be expected that the measure will change drastically,
and one of our main goals is to describe this change quantitatively.
Specifically, we will derive Theorem \ref{thm:main_brw} by precisely understanding
how the conditioning of the \emph{endpoints} changes the behavior of the \emph{increments}.

The picture to keep in mind, as indicated in Figure \ref{fig:frontpage},
is that instead of behaving like a standard BRW, the first $n-O(r)$ increments will
behave like a branching Ornstein--Uhlenbeck (OU) walk, with a strong pull back towards
the origin at each step (whose strength does not depend on $r$).
The OU walk has a stationary probability distribution and
exhibits exponential decay of correlations along its trajectory.
Thus as soon as the longest common ancestor $\tau'$ of $\tau_1$ and $\tau_2$
has length less than $n - O(r)$, the two walks leading from $\tau'$ to
$\tau_1$ and $\tau_2$ will quickly decorrelate as they take independent approximate OU steps.

\begin{remark}
\label{rmk:lowerbounds_intro}
Theorem \ref{thm:main_brw} does not indicate what the correlation structure is
for $\d(\tau_1,\tau_2) \lesssim r$, and clearly the stated upper bound
is loose in this situation.
We expect that the correlation structure in this scenario is similar to that 
of a depth-$\Omega(r)$ standard BRW.
This is indeed indicated by our methods (see e.g.\ Remark \ref{rmk:lowerbounds}),
and the precise picture should be that the first $n-O(r)$ increments
have a strong pull back to the origin, while the last $\Omega(r)$
steps are close to standard Gaussians.
However, we do not pursue this analysis in the present article.
\end{remark}

\begin{remark}
\label{rmk:truebehavior}
The heuristic described in Remark \ref{rmk:lowerbounds_intro} also explains why
the correct correlation length should be $O(r)$ rather than $O(r^2)$ (or some other function of $r$).
Indeed, the histogram of endpoints of a depth-$k$ BRW typically looks Gaussian with spread $O(\sqrt{k})$,
but includes many points with size $\Omega(k)$, deep into the tails of the Gaussian.
So if, for instance, the last $\Omega(r^\alpha)$ layers looked like a standard BRW (with $\alpha \in (1,2)$)
then while any particular endpoint would lie within $[-r,r]$ with high probability,
the \emph{extremal} endpoints would escape this interval.
Because of the extremely light lower tails of the maximum (and upper tails of the minimum),
it would be probabilistically very costly to cut them off via the conditioning.
\end{remark}

The proof of Theorem \ref{thm:main_brw} is based on the
\emph{renormalization group}, which is a broad class of techniques
for multiscale analysis in statistical mechanics and quantum field theory.
In the present context it may be thought of as a way to capture the effect of the
endpoint conditioning on the previous increments in the walk,
in principle even giving an explicit procedure for sampling from the
conditioned measure.
However, the distributions of increments are defined recursively, backwards from the endpoints,
and it is a nontrivial task to understand them; this constitutes the bulk of our analysis in
the present article.
The broad overview of this method and our specific technique for analyzing
it will be given in Section \ref{sec:oop}.

Our analysis of the renormalization group
also applies to a more general class of \emph{field theories} based on the BRW,
constructed by applying a potential to the endpoints.
Conditioning the BRW to have endpoints lying in $[-r,r]$ may be thought of as imposing
an \emph{infinite square well potential}, which is a strongly confining
potential.
This perspective will be described in more detail in Section \ref{sec:intro_fieldtheories}
below, and we will state results for a few other specific examples of confining field theories,
namely the \emph{sinh--Gordon model} and generalizations of the \emph{$\phi^4$ model},
as Theorems \ref{thm:main_sinhgordon} and \ref{thm:main_phialpha} in that section.
These results follow from a general formulation which will also yield Theorem \ref{thm:main_brw}.
This generalization will be stated as Theorem \ref{thm:main_general}, after we introduce
some requisite terminology in Section \ref{sec:oop}.

\vspace{3mm}

We now turn to describing the physical motivation of our setup,
as well as reviewing the related literature which provides a bit of context
for our results as well as our methods.
\subsection{Branching random walk and the Gaussian free field}
\label{sec:intro_gff}

As mentioned above, the BRW is often used as a toy model of the
two-dimensional discrete Gaussian free field (DGFF).
Let us quickly describe this model for completeness, although we will
not refer to it after Section \ref{sec:intro}.
For $m \in \N$, let us set $\Lambda_m = \{0,\dotsc,m-1\}^2$, and let
$\Lambda_m^+ = \{-1,\dotsc,m\}^2$.
The DGFF in $\Lambda_m$ (with Dirichlet boundary conditions)
is a Gaussian process $\Psi = \left(\psi_x\right)_{x \in \Lambda_m}$ with density
proportional to 
\begin{equation}
\label{eq:gff}
    \Exp{- \frac{1}{16}
        \sum_{\substack{x,y \in \Lambda_m^+ \\ \| x - y \| = 1}}
        \left( \psi_x - \psi_y \right)^2},
\end{equation}
where for $x \in \Lambda_m^+ \setminus \Lambda_m$ we set $\psi_x = 0$.
The covariance structure of this process is given by the Green's function
of the random walk in $\Z^2$ killed upon exiting $\Lambda_m$ \cite[Exercise 1.7]{B20Extrema}.
In particular, the variance of $\psi_x$ well within the interior of $\Lambda_m$
is of order $\log m$, and the covariance between $\psi_x$ and $\psi_y$
is of order $\log \frac{m}{\| x - y\|}$ for $x \neq y$ \cite[Theorem 1.17]{B20Extrema}.
For this reason, the DGFF is called \emph{log-correlated}.

Since the variance of $\psi_x$ increases to $\infty$ as $m \to \infty$,
there is no \emph{thermodynamic limit} of the DGFF in all of $\Z^2$,
unlike in other models in statistical physics such as the Ising model.
However, one can obtain a distribution-valued continuum limit $h$,
also called the Gaussian free field (GFF), in any reasonable 
finite domain $D \subseteq \R^2$.
This is obtained by considering the DGFF in $m D \cap \Z^2$ for $m \in \N$
(defined via a formula similar to \eqref{eq:gff}) as a distribution $h_m$ in $D$
which maps a test function $f : D \to \R$ to
\begin{equation}
    h_m(f) = \sum_{x \in m D \cap \Z^2} \psi_x f(x/m),
\end{equation}
and then taking a suitable distributional limit as $m \to \infty$.
Heuristically, the continuum GFF may be thought of as a random function $h : D \to \R$
with density proportional to
\begin{equation}
\label{eq:cgff}
    \Exp{- \frac{1}{2} \left< h, -\Delta h \right>},
\end{equation}
where $\Delta$ is the Laplacian operator in $D$.
However, as $h$ is not actually a function, this perspective requires some care to make precise
(e.g.\ via Fourier analysis; see \cite[Theorem 1.44]{BP25Gaussian}).

The continuum GFF arises as the universal scaling limit of a wide variety of
models from all domains of probability and statistical mechanics
such as tilings \cite{K00Conformal,K01Dominos,BLR20Dimers},
Dyson Brownian motion \cite{AGZ10Introduction,B10Clt},
and the critical two-dimensional Ising model via bosonization \cite{D11exact}.
It is also conjectured to be the universal limit of a variety of other
interface-type models
such as the solid-on-solid model with slanted boundary conditions \cite{BGV01Entropic}
and, what is quite related, the interface in the low-temperature 3-dimensional
Ising model under slanted Dobrushin boundary conditions \cite{LL24tilted}.
In the physics literature, the continuum GFF represents a ``non-interacting''
quantum field theory \cite{GJ12Quantum}, and various interaction terms may be imposed to
obtain other field theories, which we will explore shortly in
Section \ref{sec:intro_fieldtheories}.
For more information on the GFF and related topics,
the reader may consult the excellent references \cite{B20Extrema,BP25Gaussian}.

So the GFF, and by extension the DGFF, is of broad interest both in the mathematics
and the physics communities,
and as such, it is useful to introduce a simpler model which displays many of the same
qualitative characteristics.
The BRW is perhaps the simplest model which exhibits the same logarithmic correlation
structure as the DGFF, and from this a variety of similar phenomena can be observed
in the two models, and comparisons can also be made to derive results about one from the
other (see e.g.\ \cite[Lecture 7]{B20Extrema}).
Let us now see how to relate the DGFF and the BRW more explicitly.

\begin{figure}
    \centering
    \includegraphics[width=0.18\textwidth]{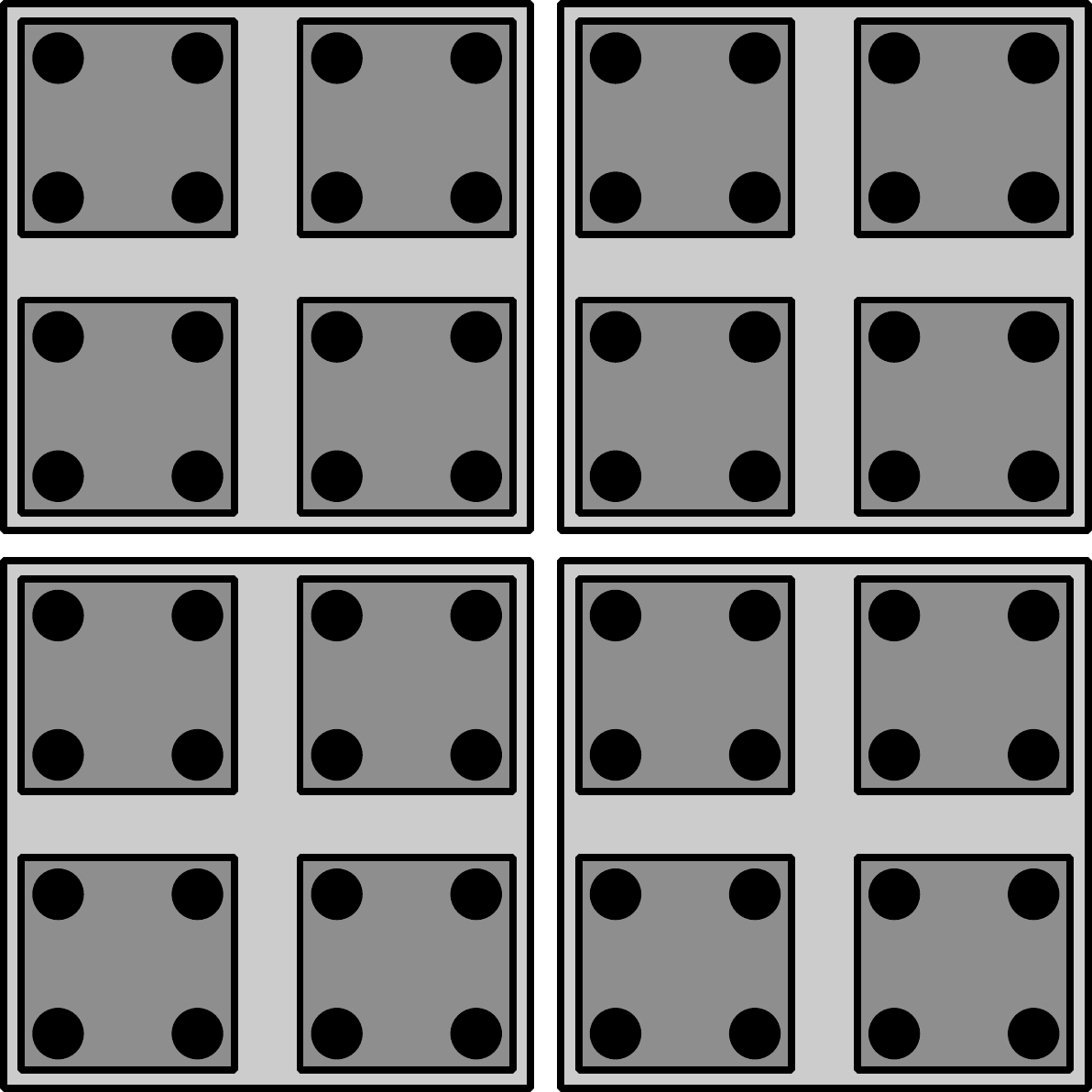}
    \hspace{0.03\textwidth}
    \includegraphics[width=0.75\textwidth]{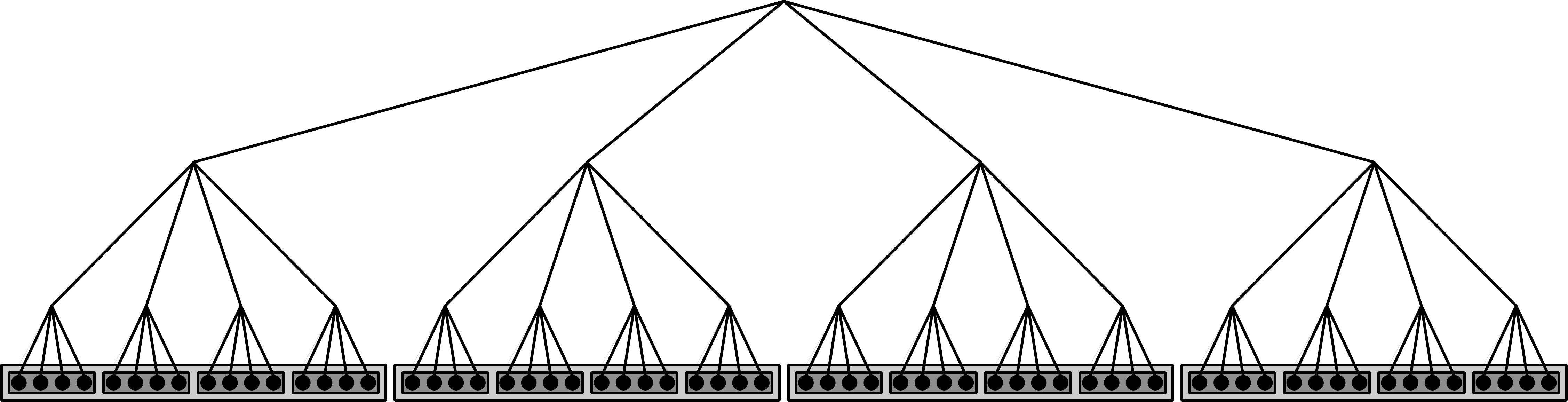}
    \caption{
    An illustration of the correspondence between $x \in \Lambda_m$ (left-hand side)
    and $\tau \in [4]^n$ (right-hand side), for $m = 8 = 2^3 = 2^n$.
    The lightly shaded larger boxes contain points which have the same first coordinate
    in $\tau$ under this identification, and the darkly shaded smaller boxes contain points
    which have the same first and second coordinate.
    }
    \label{fig:hierarchical}
\end{figure}

When $m = 2^n$, one can approximate the DGFF
$\left(\psi_x\right)_{x \in \Lambda_m}$ in $\Lambda_m$
by the endpoints of a $4$-branching random walk of depth $n$, i.e.\
$\left(\phi_\tau\right)_{\tau \in [4]^n}$, by identifying $x = (x_1, x_2) \in \Lambda_m$
with the string $\tau_x \in [4]^n$ via the binary expansions of $x_1$ and $x_2$ (see Figure \ref{fig:hierarchical}).
In particular, if we have
\begin{equation}
    x_1 = \sum_{j=0}^{n-1} 2^j x_1^j
    \qquad \text{and} \qquad
    x_2 = \sum_{j=0}^{n-1} 2^j x_2^j
\end{equation}
with $x_1^j, x_2^j \in \{0,1\}$ for $j \in \{0,\dotsc,n-1\}$,
then we identify $x$ with the $4$-ary string
\begin{equation}
    \tau_x = \left( 2 x_1^{n-1} + x_2^{n-1}, \dotsc, 2 x_1^0 + x_2^0 \right).
\end{equation}
Note that under this identification, $\d(\tau_x,\tau_y)$ is approximately
$\log_2 \| x - y \|$ for most choices of $x$ and $y$, although there will be a
large discrepancy for certain pairs of points, for instance those which
straddle one of the lines which divide $\Lambda_m$ in half.
For such points, $\| x - y \| = 1$ while $\d(\tau_x,\tau_y) = n$, but
there are relatively few points where this discrepancy is large.
Now, recalling the definition \eqref{eq:def_brw} of the BRW, for any
$x \in \Lambda_m$ we have $\Var[\phi_{\tau_x}] = n = \log_2 m$.
Further, for any $x,y \in \Lambda_m$ with $x \neq y$ we have
\begin{equation}
    \Cov\left[ \phi_{\tau_x}, \phi_{\tau_y} \right]
    = n - \d(\tau_x,\tau_y)
    \approx \log_2 m - \log_2 \| x - y \|,
\end{equation}
which is comparable to the covariance between $\psi_x$ and $\psi_y$
under the DGFF, as mentioned below \eqref{eq:gff}.

\vspace{3mm}

This motivates the comparison between the BRW and the DGFF; let us now
turn to a review of the literature related to our results with this 
perspective in mind.
\subsection{Entropic repulsion and confined random surfaces}
\label{sec:intro_confined}

There is a host of literature on random lattice surfaces
constrained to lie above a hard floor;
an example is the DGFF $\Psi = \left(\psi_x\right)_{x \in \Lambda_m}$
conditioned on the event
\begin{equation}
    \left\{ \psi_x \geq 0 \text{ for all } x \in \Lambda_m \right\}.
\end{equation}
This is essentially a one-sided version of the conditioning we consider
in the present article.
Such conditioning leads to the phenomenon of \emph{entropic repulsion},
whereby the surface is pushed upwards to allow for entropically favorable
downwards fluctuations.
This phenomenon has been proven for a variety of models including
the $(2+1)$-dimensional solid-on-solid model \cite{BMF86Random,CLMST14Dynamics}
which is a stand-in for the low-temperature interface in the $3$-dimensional Ising model,
the DGFF \cite{BDZ95Entropic,BDG01Entropic}, and more recently the BRW itself \cite{R24Branching}.

Heuristically, entropic repulsion in the BRW means that the initial few steps are pushed up
high enough so that the remaining steps do not feel the effect of the floor and may behave as usual.
This may be compared to the behavior discussed in Remarks \ref{rmk:lowerbounds_intro}
and \ref{rmk:truebehavior}: in our setting of imposing a hard floor and ceiling, the behavior
at the beginning is also altered in such a way that the remainder of the walk (the last $O(r)$ steps)
may behave freely.
However, the effect is stronger in the present setting, leading to more of the initial steps
being modified.

The version of this floor-and-ceiling problem for lattice surfaces,
e.g.\ the DGFF $\Psi$ conditioned on 
the small $\infty$-ball event
\begin{equation}
    \left\{ \| \Psi \|_\infty \leq r \right\}
    = \left\{ |\psi_x| \leq r \text{ for all } x \in \Lambda_m \right\},
\end{equation}
has received much less attention than the situation with only a hard floor.
However, a somewhat softer version of this problem, namely the \emph{infinite-volume}
local weak limit obtained by sending $m \to \infty$ with $r$ fixed,
was considered by both \cite{MS77Decay} and \cite{BMF86Random}.

First, in 1977, McBryan and Spencer \cite{MS77Decay},
motivated by connections to the $\mathrm{SO}(n)$ model, gave a short 
proof that the infinite-volume floor-and-ceiling model
exhibits exponential decay of correlations,
with a bound of $e^{O(r^2)}$ on the correlation length.
Later in 1986, the seminal work of Bricmont, El Mellouki, and Fr\"ohlich \cite{BMF86Random}
(which also introduced the notion of entropic repulsion) considered
the same model as a stand-in for the behavior of one surface in a linear series
of domain walls, i.e.\ random surfaces under a nonintersection constraint,
and via a more involved argument, they improved the bound on the
correlation length to $e^{O(r)}$.
We also mention in passing that the nonintersecting surfaces which motivated \cite{BMF86Random}
(and which were also mentioned by \cite{S05Random}) have only been recently studied directly \cite{AGP26Rigidity}.

The floor-and-ceiling problem for BRW itself (i.e.\ the problem we consider)
has not been studied before,
but our results may be predicted from the results of \cite{BMF86Random} on the DGFF.
Recall from Section \ref{sec:intro_gff} that we can compare the
DGFF in $\Lambda_m$ with $m = 2^n$
to the $4$-BRW endpoints $\Phi = (\phi_\tau)_{\tau \in [4]^n}$,
and that we have $\d(\tau_x,\tau_y) \approx \log_2 \| x - y \|$ under this comparison.
So, having a correlation length of $e^{O(r)}$ for the DGFF corresponds
to having a correlation length of $O(r)$ for the BRW under the conditioning
on $\{ \| \Phi \|_\infty \leq r \}$, which is exactly what we obtain.
As alluded to in Remark \ref{rmk:truebehavior}, this is the qualitatively
correct behavior; see \cite[Section 4]{V06Localization} for more about this
in the lattice setting.

One benefit of working with the simpler BRW as compared to the DGFF which was
previously considered is that our arguments are simpler and reveal more transparently
the underlying mechanisms behind the correlation decay.
They even provide an explicit idea of why the $e^{O(r^2)}$ bound on the correlation
length obtained by \cite{MS77Decay} (corresponding to an $O(r^2)$ bound for the BRW)
is indeed a simpler observation than the optimal bound
$e^{O(r)}$ of \cite{BMF86Random}; see Section \ref{sec:oop_shrinking} for more about this.
Furthermore, we work in a finite-volume setting and obtain bounds which hold
for all $n$ and all $r$ simultaneously (in particular, $r$ may depend on $n$).
Finally, our methods apply to a more general class of models, which we now introduce.
\subsection{Correlation decay results for hierarchical field theories}
\label{sec:intro_fieldtheories}

As mentioned in Section \ref{sec:intro_gff}, in the mathematical formulation of quantum field theory (QFT),
the continuum GFF plays the role of a ``free'' theory with no interactions.
As a first step towards realistic models such as quantum electrodynamics
or quantum chromodynamics (which involve multiple interacting fields),
a self-interaction term may be added to the GFF.
This is heuristically achieved by including a potential term
in \eqref{eq:cgff}, resulting in a ``density'' of the form
\begin{equation}
\label{eq:fieldtheory}
    \Exp{- \frac{1}{2} \left< h, -\Delta h \right> - \int_D \ham(h(x)) \,dx },
\end{equation}
for some $\ham : \R \to \R$.
A common initial choice for the potential is $\ham(\phi) = \lambda \phi^4$ for some $\lambda > 0$, 
resulting in the well-studied \emph{$\phi^4$ model}, which is the simplest
nontrivial QFT \cite{GJ12Quantum,S15P,GRS75P,BGH23Multiscale}.
Another choice, motivated by the theory of Gaussian multiplicative chaos (GMC), is
the \emph{sinh--Gordon} model which uses the potential $\ham(\phi) = \lambda \cosh \left( \sqrt{\beta} \phi \right)$
for some $\lambda, \beta > 0$ \cite{AFZ79Quantum,L08Construction,FP77Remarks,GGV262d,BV21Elliptic,HZ25Liouville,HZ24decay,H25Coupling}.
For more background on GMC and QFT, the reader may consult \cite{BP25Gaussian}
and \cite{GJ12Quantum} respectively.

Of course, as the continuum GFF $h$ is not a function in two dimensions
but rather a distribution, it is already not an easy task to define the measure
with ``density'' \eqref{eq:fieldtheory}, let alone study it.
There are multiple approaches to the construction of such a measure, including
via smooth mollification and via discretization.
For the latter approach, an analogous potential is instead applied to the DGFF with density \eqref{eq:gff},
which is well-defined.
In the present article, we will consider hierarchical versions of these discretized field theories,
where the DGFF is replaced by the BRW as described in Section \ref{sec:intro_gff}.
In other words, we will consider the following measures which are exponential tilts away from the endpoints of the standard BRW.

\begin{definition}
\label{def:mubnh}
Let $\mubno$ denote the distribution of the endpoints $\Phi = \left( \phi_\tau \right)_{\tau \in [b]^n}$
of the standard BRW as defined in \eqref{eq:def_brw}.
Then for any convex function $\ham : \R \to \R \cup \{\infty\}$, we may consider the probability measure
$\mubnh$ on $\R^{[b]^n}$, which has the following Radon--Nikodym derivative:
\begin{equation}
    \frac{\mubnh(d\Phi)}{\mubno(d\Phi)} \propto \Exp{- \sum_{\tau \in [b]^n} \ham\left( \phi_\tau \right)}.
\end{equation}
\end{definition}

Note that by the convexity assumption, the normalization constant is finite and so $\mubnh$
is indeed a well-defined probability measure.
This generalizes the setting of Section \ref{sec:intro_setupresults}, since we may write the distribution
of BRW endpoints conditioned to lie in $[-r,r]$ as $\mubnhr$, where
\begin{equation}
\label{eq:hamr}
    \hamr(\phi) = \begin{cases}
        0 &\text{if } |\phi| \leq r, \\
        \infty &\text{otherwise.}
    \end{cases}
\end{equation}
For notational convenience, we will use $\mubnh$ to denote both the probability and the expectation
under this measure.
In other words, for $F : \R^{[b]^n} \to \R$, the expression $\mubnh\left[ F(\Phi) \right]$
denotes the expectation of $F(\Phi)$ for $\Phi \sim \mubnh$, and for $A \subseteq \R^{[b]^n}$,
the expression $\mubnh\left[ A \right]$ denotes the probability that $\Phi \in A$.

Our main result, Theorem \ref{thm:main_brw}, can be generalized quite broadly to bound covariances
under the measures $\mubnh$.
The general version of the result will be stated as Theorem \ref{thm:main_general} in Section \ref{sec:oop_oucc}
below, after some necessary preparation.
For now let us state a few consequences of the general theorem in a few particular cases of interest.
Here and in the sequel, we will only consider potentials $\ham$ which are symmetric about $0$,
which implies that $\mubnh[\phi_\tau] = 0$ for all $\tau \in [b]^n$.
Thus the covariance of $\phi_{\tau_1}$ and $\phi_{\tau_2}$ under $\mubnh$ is equal to
the two-point function $\mubnh \left[ \phi_{\tau_1} \phi_{\tau_2} \right]$, and our results
are stated as bounds on this latter quantity.

First, we consider the discrete hierarchical version of the sinh--Gordon model.

\begin{theorem}
\label{thm:main_sinhgordon}
For any $\lambda, \beta > 0$, let $\hambl(\phi) = \lambda \cosh \left( \sqrt{\beta} \phi \right)$.
There are constants $C, c > 0$ depending only on $b$ such that
\begin{equation}
    \left| \mubnhbl \left[ \phi_{\tau_1} \phi_{\tau_2} \right] \right|
    \leq C \Exp{- c \d(\tau_1, \tau_2) + C \log_+ \frac{1}{\lambda \beta}}
\end{equation}
for all $\tau_1, \tau_2 \in [b]^n$, where $\log_+ x = \max\{0,\log x\}$.
\end{theorem}

\begin{remark}
\label{rmk:strict_convexity}
Theorem \ref{thm:main_sinhgordon} is actually somewhat easier than Theorem \ref{thm:main_brw},
as will be discussed in more detail in Section \ref{sec:init_sinhgordon}.
This is because $\hambl(\phi) = \lambda \cosh \left( \sqrt{\beta} \phi \right)$
is strictly convex, and its second derivative is always at least $\lambda \beta$.
For this reason, at the level of covariances, the sinh--Gordon model behaves similarly to the case of
a quadratic potential $m^2 \frac{\phi^2}{2}$ with $m^2 = \lambda \beta$.
In the physics lingo, the sinh--Gordon model has \emph{mass} at least $\sqrt{\lambda \beta}$.
\end{remark}

Applying a quadratic potential as mentioned in Remark \ref{rmk:strict_convexity} would make the entire system Gaussian and thus exactly solvable;
we will describe the behavior of such a system in Section \ref{sec:oop_quadratic}, and there
we will see that the correlation length under the quadratic potential $m^2 \frac{\phi^2}{2}$ is
of order $\log \frac{1}{m}$ for small $m$, agreeing with the obtained correlation
length in Theorem \ref{thm:main_sinhgordon}.
Our next result improves upon this by giving a bound on the correlation length for potentials
$\lambda |\phi|^\alpha$ with any power $\alpha \geq 2$; in particular, this includes the hierarchical $\phi^4$ model.
Unlike the case of the sinh--Gordon model, for $\alpha > 2$ this potential is not strictly convex,
and the following theorem makes use of the full strength of our general Theorem \ref{thm:main_general}.

\begin{theorem}
\label{thm:main_phialpha}
For any $\lambda > 0$ and $\alpha \geq 2$, let $\hamal(\phi) = \lambda |\phi|^\alpha$.
There are constants $C, c > 0$ depending only on $b$, and for each $\alpha_* \geq 2$
a constant $D_{\alpha_*} > 0$, such that for all $\alpha \in [2,\alpha_*]$ and all $\lambda > 0$
we have
\begin{equation}
    \left| \mubnhal \left[ \phi_{\tau_1} \phi_{\tau_2} \right] \right|
    \leq C \Exp{- c \d(\tau_1, \tau_2) + C \left( \lambda^{-\frac{\alpha-2}{\alpha^2}} + \frac{1}{\alpha} \log_+ \frac{1}{\lambda} + D_{\alpha_*} \right)}
\end{equation}
for all $\tau_1, \tau_2 \in [b]^n$, where $\log_+ x = \max\{0,\log x\}$.
\end{theorem}

This result can be seen as an interpolation between the massive (quadratic) potential
and the hard conditioning potential $\hamr$ defined in \eqref{eq:hamr}; indeed, when $\alpha = 2$ we recover the logarithmic correlation length
since $\frac{\alpha-2}{\alpha^2} = 0$, and if we let $\alpha \to \infty$ and $\lambda \to 0$ keeping $\lambda^{-\frac{1}{\alpha}} = r$,
then we see that
\begin{equation}
    \hamal(\phi) = \left| \frac{\phi}{\lambda^{-\frac{1}{\alpha}}} \right|^\alpha
    = \left| \frac{\phi}{r} \right|^\alpha
    \to \begin{cases}
        0 &\text{if } |\phi| < r, \\
        1 &\text{if } |\phi| = r, \\
        \infty &\text{if } |\phi| > r,
    \end{cases}
\end{equation}
which is equal to $\hamr$ except at $\phi = \pm r$.
As $\alpha \to \infty$ the power $\frac{\alpha-2}{\alpha^2}$ is approximately $\frac{1}{\alpha}$,
so we heuristically recover the correlation length $\lambda^{-\frac{1}{\alpha}} = r$, which is
consistent with Theorem \ref{thm:main_brw}.

\begin{remark}
\label{rmk:alphabehavior}
The constant $D_{\alpha_*}$ in the statement of Theorem \ref{thm:main_phialpha} is exponentially large in
$\alpha_*$, ruling out any direct comparison
along the lines of the heuristic just mentioned above.
Additionally, the precise form of the power $\frac{\alpha-2}{\alpha^2}$ is not optimal,
and in our arguments may be replaced by a variety of similar-looking expressions
at the cost of changing the constants (see Remark \ref{rmk:nonoptimal_power}).
We expect that a more careful analysis may be able to elucidate
the true optimal behavior, as well as tame the dependence of $D_{\alpha_*}$
on $\alpha_*$ to obtain a unified treatment of all $\alpha \geq 2$ and $\lambda > 0$,
but we do not pursue this presently.
\end{remark}
\subsection{Couplings via the renormalization group}
\label{sec:intro_renormalization}

As hinted at in Section \ref{sec:intro_setupresults}, our proofs will be based on couplings
with a branching Ornstein--Uhlenbeck walk, which are obtained via a renormalization group method
which will be described in Section \ref{sec:oop}.
Broadly speaking, the renormalization group is a family of techniques for performing
multiscale analysis, going back in the physics literature to the work of
Kadanoff, Wilson, and Polchinski \cite{K66Scaling,W71Renormalization,W75Renormalization,P84Renormalization}.
Mathematically rigorous formulations of the renormalization group have seen success
in a diverse array of applications across mathematical physics where multiscale analysis
is relevant, the references \cite{B87Renormalization,BY90Grad,DH00Sine,BBS15Logarithmic,H14Theory,ADC21Marginal} representing just a
small sample of the vast literature.
The reader may also consult the excellent surveys \cite{B09Lectures,BBS19Introduction,BBD24Stochastic}
for more background on these methods.

An idea which has gained traction recently is to use renormalization group techniques to construct couplings
between a model of interest (such as a field theory as described in Section \ref{sec:intro_fieldtheories})
and a baseline model (such as the GFF) for comparison.
Often such couplings are built scale-by-scale, starting at the coarsest scale and then proceeding to finer scales.
These couplings have been constructed in a variety of settings, allowing for a variety of features
of the baseline model, such as the behavior of maxima or covariances,
to be observed also in the model of interest \cite{BH22Maximum,BGH23Multiscale,H25Coupling,H25Extreme,HZ25Liouville,BH23limit,BH24Phase}.

Our formulation of the renormalization group, which will be introduced in
Section \ref{sec:oop_polchinski} below, is most directly inspired
by the work of Biskup and Huang \cite{BH23limit,BH24Phase} in this direction.
They studied hierarchical versions of the \emph{discrete Gaussian}
(DG) model, where the endpoints of a BRW are conditioned to take integer values,
and showed in \cite{BH23limit} that there is a subcritical regime of parameters
where this conditioning does not affect the BRW much until the very end, when
all particles quickly snap to integer values.
This results in a coupling between the hierarchical DG model and the standard BRW
which gives sufficiently tight control for the analysis of the maximum value
of the DG model.
In \cite{BH24Phase}, they also studied the critical regime where the conditioning
actually has a nontrivial effect on the earlier steps of the walk, leading to an interesting family of
non-Gaussian step distributions; the effect of this change is noticeable at the level of covariances.
We also mention the related work of Bauerschmidt, Park, and Rodriguez \cite{BPR24Discretea,BPR24Discreteb}
which studied the lattice version of the DG model in the subcritical regime.

In the present article, rather than coupling to a standard BRW
or a branching walk with other non-mean-reverting increments,
our baseline model will be a branching Ornstein--Uhlenbeck walk (to be described
in Section \ref{sec:oop_oucc}), leading to a more drastic change in the
covariance structure.
For this reason, although they use different methods,
the work of Hofstetter and Zeitouni \cite{HZ24decay} bears a closer resemblance
to our results than the work of Biskup and Huang.
In \cite{HZ24decay}, the authors study a version of the hierarchical sinh--Gordon model (and the related asymmetric
\emph{Liouville model}) in the \emph{infinite-volume} limit,
where the potential arises from the Gaussian multiplicative chaos of
the full infinite-volume BRW rather than the finite discrete setting we consider,
and prove, among other things, a bound on the decay of correlations.
We also mention the forthcoming work of Abdelghani, Bauerschmidt, Hofstetter and Zeitouni
\cite{ABHZForthcoming} which constructs and analyzes a hierarchical version of the
continuum sinh--Gordon model in $\R^2$ via renormalization group methods.
\subsection{Acknowledgements and AI use statement}
\label{sec:intro_acknowledgements}

I would like to thank my advisor, Shirshendu Ganguly, for a series of stimulating
discussions which led to the initiation of this project, as well as Roland Bauerschmidt
for feedback on a draft of this article.
This work was supported by the NSF GRF grant DGE 2146752.

Large language models suggested some ideas for the proofs of
Lemmas \ref{lem:unimodality}, \ref{lem:tailcomparison}, and \ref{lem:hamal_plateau_big}.
While these lemmas are important for our results, they are all relatively simple
technical points and do not constitute main ideas of this work.
Large language models also provided the Python code which generated
Figure \ref{fig:frontpage}, helped to locate some references in the literature,
and were used to proofread the article.
All writing was done by the author.
\section{Outline of the proof}
\label{sec:oop}

In this section we give a broad overview of our arguments,
stating the main propositions along the way but deferring their proofs to later sections.
We will begin in Section \ref{sec:oop_polchinski} below with our formulation of the renormalization group,
and then in Section \ref{sec:oop_quadratic} we will examine the important exactly-solvable case
of the quadratic potential.
After that, (sub)Sections \ref{sec:oop_flow}, \ref{sec:oop_shrinking}, and \ref{sec:oop_oucc}
will give overviews of Sections \ref{sec:flow}, \ref{sec:shrink}, and \ref{sec:ou} of the paper respectively,
each focusing on one main step of the proof.
The generalized version of our main result
will be stated as Theorem \ref{thm:main_general} in Section \ref{sec:oop_oucc}.
Theorems \ref{thm:main_brw}, \ref{thm:main_sinhgordon}, and \ref{thm:main_phialpha}
will all be proven in Section \ref{sec:init} of the article, as corollaries of 
Theorem \ref{thm:main_general}.

\subsubsection*{Notational conventions}

We work in the general setting, fixing a convex potential $\ham : \R \to \R \cup \{\infty\}$
throughout this section, which is applied to each endpoint of a standard BRW on $\treebn$ resulting in the 
measure $\mubnh$ of Definition \ref{def:mubnh}.
Recall from below that definition that we will use
\begin{equation}
    \mubnh \left[ F(\Phi) \right]
    \qquad \text{and} \qquad
    \mubnh \left[ A \right]
\end{equation}
to denote the expectation under $\Phi \sim \mubnh$ of $F(\Phi)$ for any $F : \R^{[b]^n} \to \R$,
and the probability that $\Phi \in A$ for any $A \subseteq \R^{[b]^n}$ respectively.
We will adopt a similar convention for measures on $\R$, such as the Gaussian measure
$\cN_{\sigma^2}^\mu$ with mean $\mu$ and variance $\sigma^2$.
So in particular
\begin{equation}
    \cN_{\sigma^2}^\mu \left[ F(\rho) \right]
    \qquad \text{and} \qquad
    \cN_{\sigma^2}^\mu \left[ A \right]
\end{equation}
will denote the expectation with respect to $\rho \sim \cN_{\sigma^2}^\mu$
of $F(\rho)$ for $F : \R \to \R$, and the probability that $\rho \in A$ for $A \subseteq \R$
respectively.
This convention will simplify the notation greatly, since we will shortly have a variety
of other measures on $\R$ to work with, changing them frequently and comparing one to another.
If not otherwise specified, the symbol $\rho$ will be used for the sample;
a notable exception is that we frequently use $\zeta$ in the special case of $\cN_1^0$.
The symbols $\E$ and $\P$ will also be used throughout to denote expectations and probabilities
under distributions which are specified in context.

\subsection{Renormalization group setup}
\label{sec:oop_polchinski}

Let us first set up some notation for the renormalization group method
we will use, and prove a basic lemma demonstrating its usefulness.
The particular formulation of the renormalization group which we adhere to
is based most directly on the work of \cite{BH23limit,BH24Phase} which studied
the BRW conditioned to take values in a discrete lattice of points in $\R$,
as mentioned in Section \ref{sec:intro_renormalization} above.
In particular, the main lemma of this subsection, Lemma \ref{lem:treemc} below,
is the same as \cite[Lemma 3.2]{BH23limit}.
Nevertheless, we give a (slightly different) proof here for the reader's convenience.
We begin with the fundamental definition.

\begin{definition}
\label{def:vk}
The \emph{renormalized potentials} $\pot_k$ for $k \in \N \cup \{0\}$
are defined recursively.
First, set $\pot_0 = \ham$ and then define $\pot_k$ for $k \geq 1$ by
\begin{equation}
\label{eq:def_vk}
    e^{- \pot_k(\phi)} \coloneqq \left( \cN_1^\phi \left[ e^{- \pot_{k-1}(\rho)} \right] \right)^b,
\end{equation}
where the expression inside the parentheses denotes an expectation with respect to $\rho \sim \cN_1^\phi$,
a Gaussian with mean $\phi$ and variance $1$, recalling our convention just above.
\end{definition}

The renormalized potentials capture the effect of the potential (which is only applied
at the endpoints of the BRW) on the previous steps of the walk.
Specifically, $\pot_k$ gives the distribution $k$ steps \emph{before} the endpoints.
Intuitively, the impact on a parent with value $\phi$ is simply the amalgamated impacts
on each of the $b$ independent children of that parent, which before tilting would each have
distribution $\cN_1^\phi$.

The following definitions and subsequent lemma make this precise, and represent
the tilted measure $\mubnh$ as the distribution of the endpoints of a branching
Markov chain with increments determined by the renormalized potentials
$\pot_{n-1}, \pot_{n-2}, \dotsc, \pot_1, \pot_0$, in that order.

\begin{definition}
\label{def:wk}
The \emph{renormalized walk distributions} $\walk_k^\phi$ for $k \in \N \cup \{0\}$
and $\phi \in \R$ are probability measures which have the following Radon--Nikodym
derivative with respect to $\cN_1^\phi$, the Gaussian distribution with mean $\phi$
and variance $1$:
\begin{equation}
    \frac{\walk_k^\phi(d\rho)}{\cN_1^\phi(d\rho)}
    \propto e^{- \pot_k(\rho)}.
\end{equation}
\end{definition}

The renormalized walk distributions allow us to view the measure $\mubnh$ as the
endpoint distribution of a tree-indexed Markov chain, much like $\mubno$
is the distribution of the endpoints of the BRW constructed in \eqref{eq:def_brw}.
To do this, let us introduce the notation of \emph{concatenation of strings}:
for $\tau \in [b]^k$ and $\sigma \in [b]$, we set
$\tau \sigma = (\tau_1, \dotsc, \tau_k, \sigma) \in [b]^{k+1}$.
In other words, the children of $\tau$ in $\treebn$ are $\{ \tau\sigma : \sigma \in [b] \}$.

Note that conditional on $\phi_\tau$, the children $\phi_{\tau\sigma}$ under the standard
BRW measure are independent, each having distribution $\cN_1^{\phi_\tau}$.
As the next definition and lemma show, if we replace $\cN_1^{\phi_\tau}$ by $\walk_k^{\phi_\tau}$
appropriately (noting that the $k$ in $\walk_k^\phi$ corresponds to the distance from the \emph{leaves} of the tree, not the root),
then the resulting distribution of endpoints is exactly the tilted measure $\mubnh$.

\begin{definition}
\label{def:treemc}
Define a tree-indexed Markov chain $\left( \phi_\tau \right)_{\tau \in \treebn}$
by first setting $\phi_\emptyset = 0$ and then, having defined $\phi_\tau$ for all $\tau \in [b]^{k-1}$, sample
$\phi_{\tau\sigma} \sim \walk_{n-k}^{\phi_\tau}$ independently for all
$\tau \in [b]^{k-1}$ and $\sigma \in [b]$.
\end{definition}

\begin{lemma}
\label{lem:treemc}
Let $\left( \phi_\tau \right)_{\tau \in \treebn}$ be as in Definition \ref{def:treemc}.
Then $\Phi = \left(\phi_\tau\right)_{\tau \in [b]^n}$ has the distribution $\mubnh$.
\end{lemma}

To prove this, we first represent the partition function
(normalization constant) of $\mubnh$ in terms of the final
renormalized potential.
The argument of the following lemma may also be viewed as a warm-up
for the argument of Lemma \ref{lem:treemc}.
By an abuse of notation, we set
\begin{equation}
    \ham(\Phi) = \sum_{\tau \in [b]^n} \ham(\phi_\tau)
\end{equation}
for any $\Phi = \left( \phi_\tau \right)_{\tau \in [b]^n} \in \R^{[b]^n}$.

\begin{lemma}
\label{lem:partitionfunction}
We have
\begin{equation}
    \mubno \left[ e^{- \ham(\Phi)} \right] = e^{- \pot_n(0)}.
\end{equation}
\end{lemma}

\begin{proof}[Proof of Lemma \ref{lem:partitionfunction}]
In this proof let us use $\left( \psi_\tau \right)_{\tau \in \treebn}$
to denote the standard BRW as in the definition \eqref{eq:def_brw}, also using
$\left( \zeta_\tau \right)_{\tau \in \treebn \setminus \{\emptyset\}}$
to denote the increments which are i.i.d.\ standard Gaussians.
Then, by definition, we have
\begin{equation}
    \mubno \left[ e^{- \ham(\Phi)}\right]
    = \E \left[ \prod_{\tau \in [b]^n} e^{- \ham(\psi_\tau)} \right]
    = \E \left[ \prod_{\tau \in [b]^n} e^{- \pot_0(\psi_\tau)} \right],
\end{equation}
and also for any $k \in \{1, \dotsc,n\}$ we have
\begin{equation}
    \E \left[ \prod_{\tau\sigma \in [b]^{n-(k-1)}} e^{-\pot_{k-1}(\psi_{\tau\sigma})} \right]
    = \E \left[
        \E \left[
            \prod_{\tau \in [b]^{n-k}} \prod_{\sigma \in [b]}
            e^{-\pot_{k-1}(\psi_\tau + \zeta_{\tau \sigma})}
        \,\middle|\,
            \left\{ \psi_\tau : \tau \in [b]^{n-k} \right\}
        \right]
    \right].
\end{equation}
The factors in the product above are conditionally independent,
so we may pull the conditional expectation inside, and then
invoke Definition \ref{def:vk} to simplify the expression:
\begin{align}
    \E \left[ \prod_{\tau\sigma \in [b]^{n-(k-1)}} e^{-\pot_{k-1}(\psi_{\tau\sigma})} \right]
    &= \E \left[
        \prod_{\tau \in [b]^{n-k}} \prod_{\sigma \in [b]}
        \E \left[
            e^{-\pot_{k-1}(\psi_\tau + \zeta_{\tau \sigma})}
        \,\middle|\,
            \left\{ \psi_\tau : \tau \in [b]^{n-k} \right\}
        \right]
    \right] \\
    &= \E \left[
        \prod_{\tau \in [b]^{n-k}} \prod_{\sigma \in [b]}
        e^{- \frac{1}{b} \pot_k(\psi_\tau)}
    \right] \\
    &= \E \left[
        \prod_{\tau \in [b]^{n-k}}
        e^{- \pot_k(\psi_\tau)}
    \right].
\end{align}
Iterating the above $n$ times finishes the proof as $\psi_\emptyset = 0$
deterministically by definition.
\end{proof}

\begin{proof}[Proof of Lemma \ref{lem:treemc}]
Let $F : \R^{[b]^n} \to \R$ be any function which is bounded and continuous.
Applying Lemma \ref{lem:partitionfunction} and Definitions \ref{def:mubnh} and \ref{def:vk} of
$\mubnh$ and $\pot_0$, we have
\begin{equation}
\label{eq:treemc_init}
    \mubnh \left[ F(\Phi) \right]
    = \frac{\mubno\left[ F(\Phi) e^{- \ham(\Phi)} \right]}
    {\mubno \left[ e^{- \ham(\Phi)} \right]}
    = e^{\pot_n(0)} \cdot \mubno \left[ F(\Phi) \prod_{\tau \in [b]^n} e^{-\pot_0(\phi_\tau)} \right].
\end{equation}
Now let us use $\left( \psi_\tau \right)_{\tau \in \treebn}$
to denote the standard BRW of \eqref{eq:def_brw}, and define a sequence of interpolations between
this and $\left( \phi_\tau \right)_{\tau \in \treebn}$, the walk of Definition \ref{def:treemc}.
Specifically, for each $m \in \{0,\dotsc,n\}$, let us define
$\left( \psi_\tau^m \right)_{\tau \in \treebn}$ recursively as follows:
first set $\psi_\emptyset^m = 0$ and then, having defined $\psi_\tau^m$
for all $\tau \in [b]^{k-1}$, subsequently sampling $\psi_{\tau\sigma}$
independently for all $\tau \in [b]^{k-1}$ and $\sigma \in [b]$ via
\begin{equation}
    \psi_{\tau\sigma}^m \sim \begin{cases}
        \cN_1^{\psi_\tau^m} & \text{if } k \leq m, \\
        \walk_{n-k}^{\psi_\tau^m} & \text{if } k > m.
    \end{cases}
\end{equation}
In other words, we use standard BRW steps for the first $m$ steps
and then apply the renormalized potentials after that.
Then, defining $\Psi^m = \left( \psi_\tau^m \right)_{\tau \in [b]^n}$,
we have $\Psi^0 = \Phi = \left( \phi_\tau \right)_{\tau \in [b]^n}$,
the endpoints of the walk of Definition \ref{def:treemc},
and $\Psi^n = \Psi = \left( \psi_\tau \right)_{\tau \in [b]^n}$, the endpoints of the standard BRW.
Let us again use
$\left( \zeta_\tau \right)_{\tau \in \treebn \setminus \{\emptyset\}}$
to denote i.i.d.\ standard Gaussians.
Now for each $m \in \{1,\dotsc,n\}$, since we may write
\begin{equation}
    \left( \psi_{\tau\sigma}^m : \tau \in [b]^{m-1}, \sigma \in [b] \right)
    =
    \left( \psi_\tau^m + \zeta_{\tau\sigma} : \tau \in [b]^{m-1}, \sigma \in [b] \right),
\end{equation}
we find that
\begin{equation}
\label{eq:treemc_step}
    \E \left[ F(\Psi^m) \prod_{\tau\sigma \in [b]^m} e^{-\pot_{n-m}(\psi_{\tau\sigma}^m)} \right]
    = \E \left[
        F(\Psi^m)
        \prod_{\tau \in [b]^{m-1}}
        \prod_{\sigma \in [b]}
        e^{-\pot_{n-m}(\psi_\tau^m + \zeta_{\tau\sigma})}
    \right].
\end{equation}
Now as in the proof of Lemma \ref{lem:partitionfunction},
the factors in the product in the right-hand side of \eqref{eq:treemc_step}
are conditionally independent after conditioning on
$\left\{ \psi_\tau^m : \tau \in [b]^{m-1} \right\}$,
but they are \emph{not} conditionally independent from $F(\Psi^m)$.
We will next show, though, that when we perform a similar manipulation
as in the proof of Lemma \ref{lem:partitionfunction}, the distribution of
$\Psi^m$ becomes that of $\Psi^{m-1}$.

For this, first note that by Definition \ref{def:wk},
the measure $\walk_{n-m}^{\psi_\tau^m}(d\rho)$ has density
proportional to $e^{- \pot_{n-m}(\rho)}$ with respect to
$\cN_1^{\psi_\tau^m}(d\rho)$,
and by Definition \ref{def:vk}, the normalization constant is
\begin{equation}
    \E \left[ e^{- \pot_{n-m}(\psi_\tau^m + \zeta_{\tau\sigma})} \right]
    = e^{-\frac{1}{b} \pot_{n-(m-1)}(\psi_\tau^m)}.
\end{equation}
So if we replace each factor of $e^{- \pot_{n-m}(\psi_\tau^m + \zeta_{\tau\sigma})}$
in \eqref{eq:treemc_step} by $e^{- \frac{1}{b} \pot_{n-(m-1)}(\psi_\tau^m)}$
and change the distribution of $\psi_{\tau\sigma}^m$
from $\cN_1^{\psi_\tau^m}$ to $\walk_{n-m}^{\psi_\tau^m}$,
then the expectation in \eqref{eq:treemc_step} will be unchanged.
This amounts to changing the overall distribution of the full walk
$\left( \psi_\tau^m \right)_{\tau \in \treebn}$ into the distribution
of $\left( \psi_\tau^{m-1} \right)_{\tau \in \treebn}$.
In other words, we have
\begin{align}
    \E \left[ F(\Psi^m) \prod_{\tau\sigma \in [b]^m} e^{-\pot_{n-m}(\psi_{\tau\sigma}^m)} \right]
    &= \E \left[
        F(\Psi^{m-1})
        \prod_{\tau \in [b]^{m-1}} \prod_{\sigma \in [b]}
        e^{-\frac{1}{b} \pot_{n-(m-1)}(\psi_\tau^{m-1})}
    \right] \\
    &= \E \left[ F(\Psi^{m-1}) \prod_{\tau \in [b]^{m-1}} e^{-\pot_{n-(m-1)}(\psi_\tau^{m-1})} \right].
\end{align}
Iterating this equality $n$ times finishes the proof by \eqref{eq:treemc_init}
and the fact that $\psi_\emptyset^0 = 0$.
\end{proof}

The definitions and lemmas in this subsection have been quite general,
not relying on any properties of the initial potential $\ham$
other than the existence of the measure $\mubnh$.
It is worthwhile to pause now and see exactly how the renormalized
potentials and walk measures behave under a specific choice of $\ham$
where everything can be computed explicitly.
\subsection{The case of a quadratic potential}
\label{sec:oop_quadratic}

As mentioned below Remark \ref{rmk:strict_convexity},
in the case of a quadratic potential $\ham(\phi) = m^2 \frac{\phi^2}{2}$, Gaussian computations
elucidate the structure of the renormalization group flow exactly.
This corresponds to adding a \emph{mass term} ($m$ is the mass), and it is known \cite[Theorem 8.34]{FV12Statistical}
that in the 2D discrete Gaussian free field discussed in Section \ref{sec:intro_gff},
adding a mass induces exponential decay of correlations, with correlation length $\Theta\left(m^{-1}\right)$
for small $m$.
Since the distance $\d(\tau_1,\tau_2)$ between leaves of $\treebn$
should be thought of as the logarithm of the corresponding distance in the 2D lattice,
we should expect to find a correlation length of order $\log \frac{1}{m}$ for small $m$.
This will be proven rigorously as the $\alpha=2$ case of Theorem \ref{thm:main_phialpha}, so
we will be somewhat informal in the present section, aiming only to get the main ideas across quickly.

Let us calculate the renormalized potentials.
By Definition \ref{def:vk}, with $\zeta \sim \cN_1^0$, the first one satisfies
\begin{align}
    e^{- \frac{1}{b} \pot_1(\phi)}
    &= \cN_1^0 \left[ \Exp{- m^2 \frac{\left( \phi + \zeta \right)^2}{2} } \right] \\
    &= \frac{1}{\sqrt{2\pi}} \int_{-\infty}^\infty \,d\zeta \Exp{- \frac{\zeta^2}{2} - m^2 \frac{\left( \phi + \zeta \right)^2}{2} } \\
    &= \frac{1}{\sqrt{2\pi}} \int_{-\infty}^\infty \,d\zeta \Exp{- (1+m^2) \frac{\zeta^2}{2} - m^2 \phi \zeta - m^2 \frac{\phi^2}{2} } \\
    &= \frac{1}{\sqrt{2\pi}} \int_{-\infty}^\infty \,d\zeta \Exp{- \frac{1 + m^2}{2} \left( \zeta + \frac{m^2}{1 + m^2} \phi \right)^2
        + \frac{m^4}{1+m^2} \frac{\phi^2}{2} - m^2 \frac{\phi^2}{2} } \\
    &= \Exp{- \frac{m^2}{1+m^2} \frac{\phi^2}{2}} 
        \frac{1}{\sqrt{2\pi}} \int_{-\infty}^\infty \,d\zeta \Exp{- \frac{1 + m^2}{2} \left( \zeta + \frac{m^2}{1 + m^2} \phi \right)^2} \\
    &= \Exp{- \frac{m^2}{1 + m^2} \frac{\phi^2}{2} + \log \sqrt{\frac{1}{1+m^2}} }.
\end{align}
The constant term $\log \sqrt{\frac{1}{1+m^2}}$ will not affect the rest of our analysis and we will simply ignore it.
By the same derivation we find by induction that we may express
\begin{equation}
    \pot_k(\phi) = \mm_k^2 \frac{\phi^2}{2} + C
\end{equation}
for some constants $C = C_{m,k}$ which we will ignore, where the ``renormalized mass'' $\mm_k > 0$ satisfies
\begin{equation}
\label{eq:mrec}
    \mm_{k+1}^2 = \frac{b \mm_k^2}{1 + \mm_k^2}
    \qquad \text{with} \qquad
    \mm_0^2 = m^2.
\end{equation}
Now the map $x \mapsto \frac{b x}{1 + x}$ on the interval $(0,\infty)$ has a unique attracting fixed point at $x = b-1$
(see Figure \ref{fig:rational_function}), and it may easily be shown that $\mm_k^2 \to b-1$ as $k$ increases.
So, for large $k$, the renormalized potential $\pot_k(\rho)$ is
approximately $(b-1) \frac{\rho^2}{2}$ up to a constant shift, and so
the renormalized walk distribution $\walk_k^\phi(d\rho)$ has density approximately
proportional to
\begin{equation}
\label{eq:ou_source}
    \Exp{- \frac{(\rho - \phi)^2}{2} - (b-1) \frac{\rho^2}{2}}
    \propto \Exp{ - b \frac{\rho^2}{2} + \rho \phi}
    \propto \Exp{- \frac{b}{2} \left( \rho - \frac{\phi}{b} \right)^2};
\end{equation}
here our proportionality constants can depend on anything other than $\rho$.
In other words, the distribution of the \emph{increment} $\rho - \phi$
is a Gaussian with mean approximately $- \frac{b-1}{b} \phi$ and variance
approximately $\frac{1}{b}$.

This is exactly the behavior of an Ornstein--Uhlenbeck (OU) walk as alluded to
in Section \ref{sec:intro_setupresults}.
The parameters of the OU walk depend only on the branching
rate $b$, and when $b$ increases the walk feels a stronger pull back towards
the origin with smaller fluctuations.
As we will see in Lemma \ref{lem:oud_stationary}, the stationary distribution
of this OU walk is $\cN_{b/(b^2-1)}^0$, a centered normal distribution 
with variance $\frac{b}{b^2-1}$.
This dependence on $b$ makes sense heuristically, as with higher branching rate
each step of the walk has more descendants whose behavior must be controlled.

To find the correlation length, then, it suffices to see how large $k$ must be
for this approximation to be valid (recall the discussion below the statement of
Theorem \ref{thm:main_brw} for an explanation of this relationship).
The closeness of $\walk_{n-k}^\phi$ to the OU step distribution is directly related
to the closeness of $\mm_k^2$ to $b-1$.
Recall that $\mm_0^2 = m^2$; when $m^2$ is closer to zero, the recurrence
\eqref{eq:mrec} will take a longer time to get close to $b-1$,
since $0$ is also a fixed point of $x \mapsto \frac{bx}{1+x}$, although it is a
repelling fixed point (see Figure \ref{fig:rational_function}).
The derivative at $0$ is $b$, and so while $\mm_k^2$ is small, each iteration 
effectively multiplies it by $b$;
thus it will take $\Theta \left( \log_b \frac{1}{m^2} \right)$ iterations to reach an $\Omega(1)$
value, at which point it will only take $O(1)$ more steps to get close to $b-1$.
Thus we recover a correlation length of $O\left( \log \frac{1}{m} \right)$ as expected.

\begin{figure}
    \centering
    \includegraphics[width=0.25\textwidth]{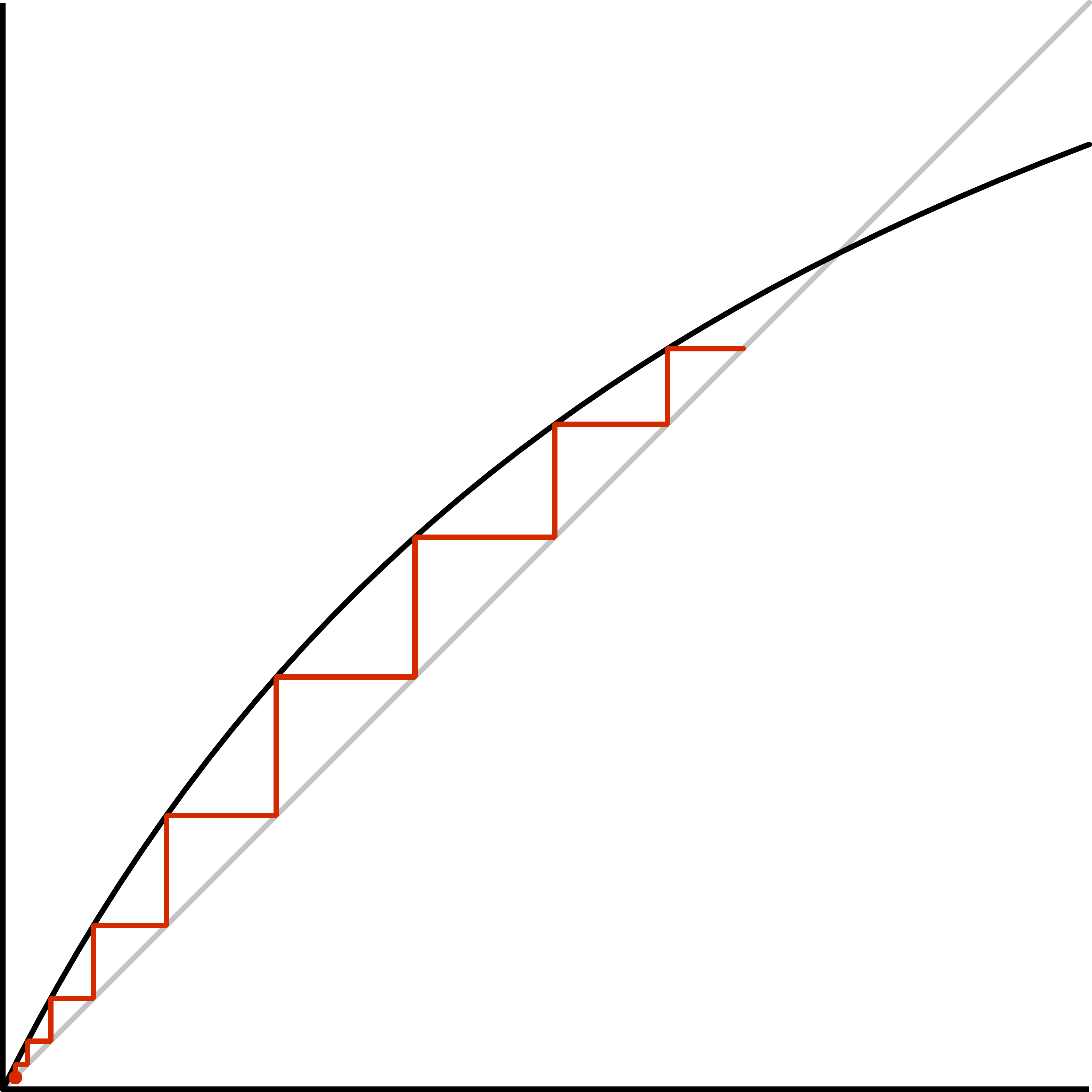}
    \hspace{0.2\textwidth}
    \includegraphics[width=0.25\textwidth]{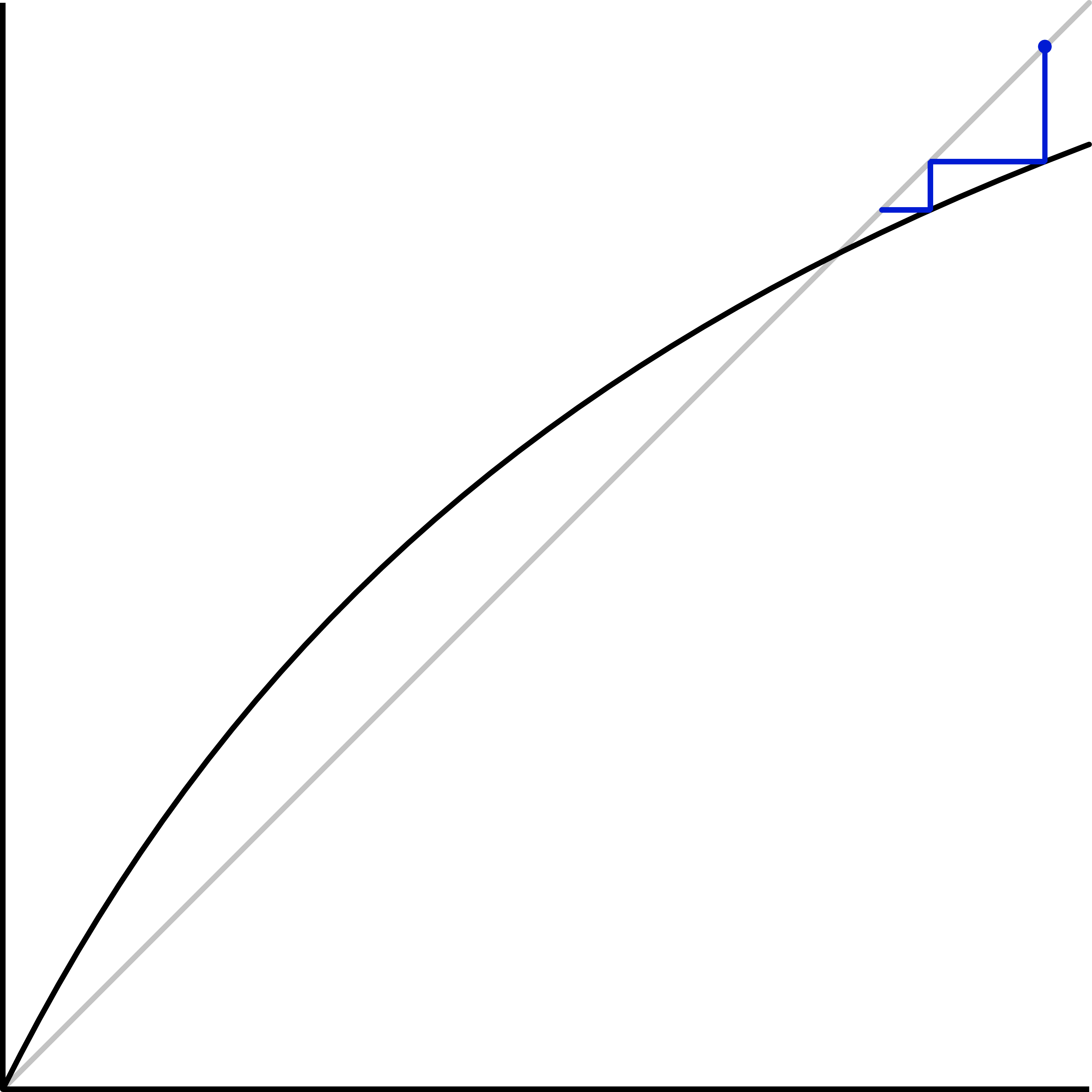}
    \caption{
        Plots of the function $x \mapsto \frac{bx}{1+x}$ with iterations starting
        below (left) and above (right) the fixed point $b-1$.
        Starting at $\mm_0^2 = m^2 \approx 0$ (red point in bottom left) as in Section \ref{sec:oop_quadratic},
        it will take $\Theta\left(\log_b\frac{1}{m^2}\right)$ iterations to
        get close to the fixed point.
        Starting at $\bb_1-1 = b$ (blue point in top right) as in Section \ref{sec:oop_flow},
        it will take $\Theta(1)$ iterations to get close to the fixed point.
    }
    \label{fig:rational_function}
\end{figure}
\subsection{The renormalization group flow}
\label{sec:oop_flow}

Let us now return to the general setup of a convex potential $\ham$.
Now it is not possible in general to obtain explicit formulas for the
renormalized potentials $\pot_k$ as in the quadratic case studied
in Section \ref{sec:oop_quadratic}.
Instead, we will separate $\pot_k$ into two terms.
One will behave in a similar manner to the quadratic potentials
described in Section \ref{sec:oop_quadratic}, and the other
will be a remainder term which we will bound.
In language which is common in other works using the renormalization group,
the quadratic part
is the \emph{relevant} part and the remainder term is the
\emph{irrelevant} part.
Despite this terminology, most of our analysis will be focused on the remainder
term, which displays interesting behavior which will be commented
on further in Section \ref{sec:oop_shrinking} below.

This behavior is only observed when $\ham$ is at least as strongly
confining as the quadratic potential.
We formalize this with the following definition.

\begin{definition}
\label{def:confining}
A potential function $\ham : \R \to \R \cup \{\infty\}$ is said to be
\emph{confining} if $\ham$ is convex and symmetric (i.e.\ $\ham(-x) = \ham(x)$),
and if for all $y \geq 0$ the function $x \mapsto \ham(y+x) - \ham(y-x)$ is
convex on the interval
\begin{equation}
    J_y = \left\{ x \geq 0 : \ham(y+x) < \infty \text{ or } \ham(y-x) < \infty \right\};
\end{equation}
note that this definition also requires that $J_y$ is indeed an interval
for all $y \geq 0$.
\end{definition}

The somewhat strange-looking condition is a technical one used
only at one point of the proof, but
it is easily checked for a variety of potentials of interest,
and in particular we will show later in Section \ref{sec:init} that the
infinite square well (small ball conditioning) potential of Theorem \ref{thm:main_brw}, the sinh--Gordon
potential of Theorem \ref{thm:main_sinhgordon}, and the power potential of Theorem \ref{thm:main_phialpha}
are all confining by this definition.
However, for instance, the absolute-value potential $\ham(x) = |x|$
is not confining by this definition.

We may now turn to the main structural proposition outlining
the two parts of the renormalized potentials as discussed above, including a
precise description of the remainder term's second derivative.
This second derivative turns out to be the crucial object
of study controlling the behavior of the renormalized walk distributions
$\walk_k^\phi$ of Definition \ref{def:wk}.
The following proposition will be proved in Section \ref{sec:flow}.

\begin{proposition}
\label{prop:flow}
If $\ham : \R \to \R \cup \{\infty\}$ is confining as in
Definition \ref{def:confining}, then for all $k \in \N$ we may write
\begin{equation}
\label{eq:structure}
    \pot_k(\phi) = (\bb_k-1) \frac{\phi^2}{2} - \rem_k(\phi),
\end{equation}
where $\bb_k = \frac{b^{k+1}-1}{b^k-1}$, and the remainder terms
$\rem_k : \R \to \R$ satisfy the following properties:
\begin{enumerate}
    \item
    \label{item:rem_initial}
    The first remainder term is given by
    \begin{equation}
    \label{eq:r1}
        \rem_1(\phi) = b \log \cN_1^0 \left[ e^{\phi \zeta - \ham(\zeta)} \right],
    \end{equation}
    using our convention for the expectation over $\zeta \sim \cN_1^0$, a standard Gaussian.
    \item
    \label{item:rem_symconleqbk}
    Each remainder term $\rem_k$ is a symmetric function
    with $0 \leq \rem_k''(\phi) \leq \bb_k - 1$ for all $\phi \in \R$.
    \item
    \label{item:rem_rec}
    The second derivatives satisfy the following recursive upper bound:
    \begin{equation}
    \label{eq:rem_rec_bl}
        \rem_{k+1}''(\phi)
        \leq \frac{b}{\bb_k} \cdot \walk_k^\phi \left[
            \frac{\rem_k''(\rho)}{\bb_k - \rem_k''(\rho)}
            \right],
    \end{equation}
    using our convention for the expectation over $\rho \sim \walk_k^\phi$
    as in Definition \ref{def:wk}.
    \item
    \label{item:rem_unimodal}
    The second derivatives $\rem_k''$ are unimodal, i.e.\ $\rem_k''(\phi)$ is
    decreasing in $|\phi|$.
\end{enumerate}
\end{proposition}

\begin{remark}
\label{rmk:confining}
In fact, the assumption that $\ham$ is confining is only needed
for item \ref{item:rem_unimodal} in this proposition.
The other items only use the assumption that $\ham$ is symmetric and convex,
and some mileage may be gotten from the recursive upper bound
\eqref{eq:rem_rec_bl} in particular in these situations, although we do not
study them in the present article.
\end{remark}

\begin{remark}
\label{rmk:bbk}
The constants $\bb_k-1$ appearing in the statement of Proposition
\ref{prop:flow} are \emph{not} the same as $\mm_k^2$ appearing
in Section \ref{sec:oop_quadratic}.
They do satisfy the same recurrence \eqref{eq:mrec}, but they are initialized
differently; namely, we do not define $\bb_0$ and instead start with
$\bb_1 = \frac{b^2-1}{b-1} = b+1$.
In particular, $\bb_1 - 1 = b$ begins \emph{above} the limiting value of $b-1$,
and does not need to spend any time escaping the well of the fixed point
$0$ of the function $x \mapsto \frac{bx}{1+x}$ (see Figure \ref{fig:rational_function}).
Thus, the evolution of $\bb_k$ is not the bottleneck in our analysis in general.
We could of course have chosen to separate out a remainder term in
Section \ref{sec:oop_quadratic}, which would yield the same formula $\bb_k$,
since indeed Proposition \ref{prop:flow} applies to a quadratic potential.
However, this would not simplify the analysis as the remainder term would
itself just be another quadratic term.
\end{remark}

As Remark \ref{rmk:bbk} hints, it is now the remainder terms which play the primary role.
Moreover, it will be crucial for us that under
potentials $\ham$ with stronger than quadratic confinement, 
the remainder term will grow slower than quadratically at infinity.
For instance, just looking at the formula \eqref{eq:r1} for the first
remainder term, if we impose the hard conditioning potential \eqref{eq:hamr}
then for $\phi \gg r$ the expectation in \eqref{eq:r1} is dominated by
the contributions from $\zeta \approx r$, leading to \emph{linear} growth
of $\rem_1(\phi)$ as $\phi \to \infty$.

The slower-than-quadratic growth of $\rem_k(\phi)$ implies that the second
derivative $\rem_k''(\phi)$ (which is nonnegative by item \ref{item:rem_symconleqbk}
of Proposition \ref{prop:flow}) tends to $0$ as $|\phi| \to \infty$.
This is the crucial behavior which we exploit as the starting point of our
bounds on $\rem_k''$, which for instance allows us to obtain the optimal $O(r)$
bound on the correlation length in the model of BRW conditioned on
$\| \Phi \|_\infty \leq r$, rather than the suboptimal bound of $O(r^2)$.
Let us now see in more detail a sketch of how this argument will proceed.
\subsection{How the remainder term shrinks}
\label{sec:oop_shrinking}

As mentioned above, we will control the remainder term $\rem_k$ via its second
derivative $\rem_k''$;
since it is symmetric and the constant term can be ignored for our purposes,
this second derivative actually captures all of the relevant information
about $\rem_k$.
For this reason, the recursive upper bound \eqref{eq:rem_rec_bl}
is of fundamental importance to our analysis.

A simple consequence of this recurrence is an iterative uniform bound;
in other words, a uniform bound on $\rem_K''$ yields another uniform bound on
$\rem_{K+k}''$ for all $k \geq K$ as follows:
\begin{equation}
\label{eq:exp}
    \text{if } \qquad \sup_{\phi \in \R} \rem_K''(\phi) \leq M,
    \qquad
    \text{then } \qquad \sup_{\phi \in \R} \rem_{K+k}''(\phi)
    \leq \frac{C}{b^k (\bb_K - 1 - M)}.
\end{equation}
This fact will be proved as Lemma \ref{lem:exp} in Section \ref{sec:shrink_exponential} below.
The bound on $\rem_{K+k}''$ above decays exponentially as $k$ increases, but if the initial bound $M$
is very close to $\bb_K - 1$, then the result will be ineffective unless $k$ is very large.

For example, let us consider the case of the square-well potential \eqref{eq:hamr},
and take $K=1$.
Then, using items \ref{item:rem_initial} and \ref{item:rem_unimodal} of
Proposition \ref{prop:flow} and the fact that $b = \bb_1 - 1$, we
will show in Section \ref{sec:flow_derivatives} below (see \eqref{eq:2dr_cum} specifically)
that
\begin{equation}
\label{eq:variance_formulation_1}
    \sup_{\phi \in \R} \rem_1''(\phi) = \rem_1''(0) = (\bb_1 - 1) \Var\left[ \zeta \,\middle|\, |\zeta| \leq r \right],
\end{equation}
where $\zeta \sim \cN_1^0$ is a standard Gaussian variable.
As $\P[|\zeta| > r] = e^{-\Omega(r^2)}$ and
the variance of the unconditioned variable is $1$, we find that
\begin{equation}
    \rem_1''(0) = \bb_1 - 1 - e^{-\Omega(r^2)}.
\end{equation}
Thus, the uniform bound \eqref{eq:exp} with $K=1$ will \emph{not} be
effective until $k \gtrsim r^2$.
This does lead to an upper bound of $O(r^2)$ on the correlation
length of the conditioned model, but as mentioned in Section \ref{sec:intro_confined},
this is not optimal.

To obtain the optimal $O(r)$ correlation length in this model,
we need to go beyond a uniform bound.
The reasoning in the above paragraph extends to $|\phi| \lesssim r$,
i.e.\ $\rem_1''(\phi)$ is very close to $\bb_1 - 1$ for such $\phi$.
However, as mentioned at the end of the previous subsection,
$\rem_1''(\phi) \to 0$ as $|\phi| \to \infty$, and using
\eqref{eq:variance_formulation_1} we can see that it begins to decrease
as soon as $|\phi| \gtrsim r$.
Thus we find that $\rem_1''$ has a plateau of width $\Theta(r)$
at height $\approx \bb_1 - 1$, and quickly tends to zero away from this plateau;
similarly, $\rem_k''$ has a plateau at height $\approx \bb_k - 1$.
Our main analysis will be to show that the \emph{width} of this plateau
shrinks linearly as $k$ increases.
Thus, in the setting of the hard conditioning potential, the plateau
will disappear after $O(r)$ steps, at which point the uniform bound
\eqref{eq:exp} is effective.
This linear shrinking of the plateau which starts with width $O(r)$
is what leads to the optimal $O(r)$ bound on the correlation length which
we obtain.
Figure \ref{fig:plateau} illustrates this behavior.

\begin{figure}
    \centering
    \includegraphics[width=0.95\textwidth]{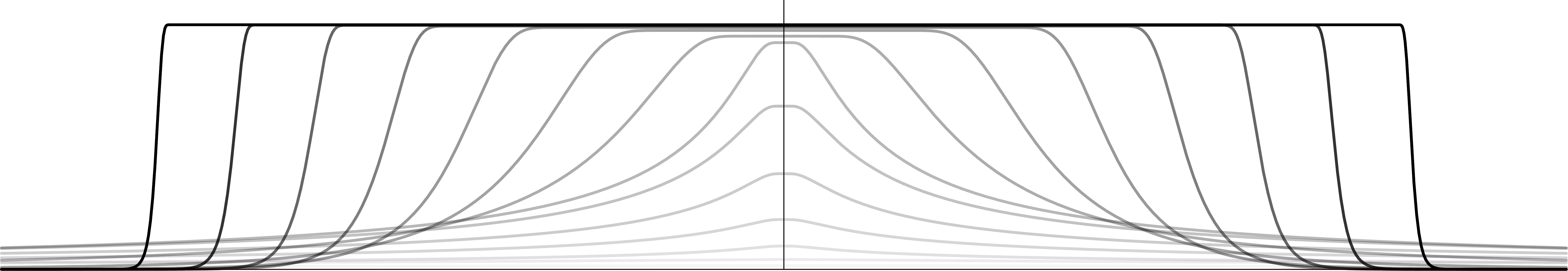}
    \caption{
        Illustrative depiction of the shrinking behavior of $\rem_k''$,
        lighter curves representing higher values of $k$.
        There is a plateau which first shrinks horizontally by $\Omega(1)$
        at each step (while also shrinking by a very small amount vertically).
        Eventually, the uniform bound on the vertical height begins to
        decrease exponentially.
    }
    \label{fig:plateau}
\end{figure}

As this plateau-shrinking behavior will drive most of our analysis,
we codify the initial behavior which allows it to work
in the following definition.

\begin{definition}
\label{def:plateau}
For any potential $\ham : \R \to \R \cup \{\infty\}$,
$T \in \N$, and $\eps > 0$, we say that $\plat(\ham,T,\eps)$
holds if the potential $\ham$ is confining in the sense of Definition \ref{def:confining},
and we have $\rem_1''(T) \leq b - \eps$,
where $\rem_1$ is the first remainder term defined in \eqref{eq:r1}.
\end{definition}

To give some intuition for the shrinking behavior described above, let us recall again
the recursive upper bound \eqref{eq:rem_rec_bl} from Proposition \ref{prop:flow}.
Since $\rem_k''(\phi) \approx \bb_k - 1$ for $\phi$ in the plateau,
the renormalized potential $\pot_k(\phi)$ is roughly constant in the plateau
by \eqref{eq:structure}.
So when $\phi$ is well enough inside of the plateau, $\walk_k^\phi$
is close to a $\cN_1^\phi$, a Gaussian with mean $\phi$ and variance $1$.
This behavior continues to hold approximately,
even for $\phi$ close to the edge of the plateau, and so in computing the expectation
over $\walk_k^\phi$ in the bound
\begin{equation}
    \rem_{k+1}''(\phi) \leq \frac{b}{\bb_k} \cdot \walk_k^\phi \left[ \frac{\rem_k''(\rho)}{\bb_k - \rem_k''(\rho)} \right],
\end{equation}
the variable $\rho \sim \walk_k^\phi$ will have some chance to overshoot the edge of the plateau,
leading to the expectation picking up some lower values.

This shrinks the edge of the subsequent plateau back towards the origin.
However, since there may only be a small chance for $\rho \sim \walk_k^\phi$ to overshoot the edge of the $k$th plateau,
the height which qualifies for the $(k+1)$st plateau must actually be slightly increased.
More precisely, the gap between $\bb_k - 1$ and the height at the edge of the plateau must shrink by some factor $\delta > 0$.
These ideas lead to the following proposition, which will be proved in Section \ref{sec:shrink}.

\begin{proposition}
\label{prop:shrinking}
There are some constants $C, \delta > 0$ depending only on $b$ such that
if $\plat(\ham,T,\eps)$ holds with some choice of $T \in \N$ and $\eps > 0$ satisfying $\eps T^2 \leq 1$,
then for all $k \geq 0$ we have
\begin{equation}
    \sup_{\phi \in \R} \rem_{2T+2+k}''(\phi) \leq \frac{C}{b^k \delta^T \eps},
\end{equation}
where $\rem_k$ is the remainder term of Proposition \ref{prop:flow}.
\end{proposition}

\begin{remark}
\label{rmk:lowerbounds}
We expect this plateau analysis idea to be useful for a variety of other applications.
For instance, as mentioned in Remark \ref{rmk:lowerbounds_intro}, we expect that a lower bound on the covariance
may be obtained in the model conditioned on a small ball event.
Such a bound may for instance be proved by showing that the width of the plateau does not shrink by too much at each step,
so that it really does take $\Omega(r)$ steps to shrink.
This seems plausible as when $\phi$ is well within the plateau, the average on the right-hand side
of \eqref{eq:rem_rec_bl} will only pick up contributions from high values of $\rem_k''$.
We leave such analysis to future work.
\end{remark}
\subsection{Coupling with a branching Ornstein--Uhlenbeck walk}
\label{sec:oop_oucc}

Now let us see what the bound in Proposition \ref{prop:shrinking} can say about the distribution of
the tree-indexed Markov chain $\left( \phi_\tau \right)_{\tau \in \treebn}$ of Definition \ref{def:treemc},
whose endpoints $\Phi = \left( \phi_\tau \right)_{\tau \in [b]^n}$ are distributed according to $\mubnh$.
As alluded to in Sections \ref{sec:intro_setupresults} and \ref{sec:oop_quadratic}, all but the last few steps of this chain should
behave like a branching Ornstein--Uhlenbeck (OU) walk.
For the OU walk we consider, given that the current position is $\psi$, the distribution
of the next position will be Gaussian with mean $\frac{\psi}{b}$ and variance $\frac{1}{b}$; let us introduce
some convenient shorthand notation for this.

\begin{definition}
\label{def:oud}
For $a > 1$ and $\psi \in \R$, let $\oud_a^\psi = \cN_{a^{-1}}^{a^{-1} \psi}$,
where we remind the reader that $\cN_{\sigma^2}^\mu$ always denotes a Gaussian distribution with mean
$\mu$ and variance $\sigma^2$.
\end{definition}

We will usually use this notation with $a = b$, but in the proofs it will also be helpful to consider
other values of $a$.
Now a branching OU walk is a tree-indexed Markov chain $\left( \psi_\tau \right)_{\tau \in \treebn}$,
where
\begin{enumerate}
    \item $\psi_\emptyset = 0$, and 
    \item given $\psi_\tau$ for some $\tau \in \treebn$, each $\psi_{\tau\sigma}$ for $\sigma \in [b]$ 
    is an independent sample from $\oud_b^{\psi_\tau}$.
\end{enumerate}
Let us now introduce a coupling between this and the tree-indexed Markov chain $\left( \phi_\tau \right)_{\tau \in \treebn}$
of Definition \ref{def:treemc}.
This construction will also include a set of auxiliary indicator variables
$\left( \chi_\tau \right)_{\tau \in \treebn}$ such that $\chi_\tau = 1$ if and only if
$\phi_{\tau'} = \psi_{\tau'}$ for all $\tau' \preceq \tau$.

\begin{definition}
\label{def:joint}
First set both $\phi_\emptyset = \psi_\emptyset = 0$ and $\chi_\emptyset = 1$.
Then inductively, if we have defined $(\phi_\tau, \psi_\tau, \chi_\tau)$ for all $\tau \in [b]^{k-1}$
for some $k \in [n]$, we define $(\phi_{\tau\sigma}, \psi_{\tau\sigma}, \chi_{\tau\sigma})$
independently for all $\tau \in [b]^{k-1}$ and $\sigma \in [b]$ by:
\begin{enumerate}
    \item If $\chi_\tau = 0$, then set $\chi_{\tau\sigma} = 0$ and sample
    $\phi_{\tau\sigma} \sim \walk_{n-k}^{\phi_\tau}$ and $\psi_{\tau\sigma} \sim \oud_b^{\psi_\tau}$
    independently.
    \item If $\chi_\tau = 1$, which means by induction that $\phi_\tau = \psi_\tau = \phi$ for some
    $\phi \in \R$, then sample $(\phi_{\tau\sigma}, \psi_{\tau\sigma})$ from the coupling
    between $\walk_{n-k}^\phi$ and $\oud_b^\phi$ which maximizes $\P[\phi_{\tau\sigma} = \psi_{\tau\sigma}]$.
    If the samples do indeed satisfy $\phi_{\tau\sigma} = \psi_{\tau\sigma}$, then set $\chi_{\tau\sigma} = 1$,
    otherwise set $\chi_{\tau\sigma} = 0$.
\end{enumerate}
\end{definition}

The next proposition, which will be proved in Section \ref{sec:ou}, states that with high probability
the ancestors of a single endpoint will agree under this coupling, until near the end of the walk.

\begin{proposition}
\label{prop:coupling}
There are some constants $C, c > 0$ depending only on $b$ such that
if $\plat(\ham,T,\eps)$ holds as in Definition \ref{def:plateau},
then the coupling of Definition \ref{def:joint} satisfies
\begin{equation}
    \P \left[ \phi_{\tau'} = \psi_{\tau'} \text{ for all } \tau' \preceq \tau \text{ with } |\tau'| \leq n - k \right]
    \geq 1 - C \Exp{-c k + C \left( T + \log \frac{1}{\eps} \right)}
\end{equation}
for all $k \in [n]$ and $\tau \in [b]^n$.
\end{proposition}

\begin{remark}
\label{rmk:epsTsquared}
The condition $\eps T^2 \leq 1$, which was a technical assumption in Proposition \ref{prop:shrinking},
is not stated here.
This is because we may always simply reduce $\eps$ until it is $\leq e^{-T} \leq T^{-2}$
without affecting the bound (up to some universal constant), since $\plat(\ham,T,\eps')$ holds
for any $\eps' \in (0,\eps)$ if $\plat(\ham,T,\eps)$ holds.
\end{remark}

\begin{remark}
Of course, it may also be of interest to know that the walks agree not just for the ancestors of a single endpoint but
indeed for a large portion of the entire tree.
Such a result could also be proved using the bound of Proposition \ref{prop:shrinking} to analyze the branching
process $\left\{ \tau : \chi_\tau = 1 \right\}$; for instance, an analogous type of analysis was done by
\cite{BH23limit,BH24Phase} to understand the \emph{maximum} of a BRW conditioned to take integer values at the end,
as mentioned in Section \ref{sec:intro_renormalization}.
However, since our main goal is to understand the covariance between only two endpoints, the above result will
suffice for our purposes.
\end{remark}

Another important consequence of Proposition \ref{prop:shrinking} is the following Gaussian tail bound.

\begin{proposition}
\label{prop:tailbound}
There are some constants $C, c > 0$ depending only on $b$ such that
if $\plat(\ham,T,\eps)$ holds as in Definition \ref{def:plateau},
then we have the following tail bound for any $\tau \in [b]^n$:
\begin{equation}
    \mubnh \left[ |\phi_\tau| \geq s \right] \leq C \Exp{- \frac{c s^2}{T + \log \frac{1}{\eps}}}.
\end{equation}
\end{proposition}

Propositions \ref{prop:coupling} and \ref{prop:tailbound} allow us to derive our main result in its general form,
which, as just remarked, is a bound on the covariance under $\mubnh$ between $\phi_{\tau_1}$ and $\phi_{\tau_2}$
for $\tau_1, \tau_2 \in [b]^n$.
Note that by our symmetry assumption on $\ham$, we have $\mubnh[\phi_\tau] = 0$ for all $\tau \in [b]^n$,
and so the covariance is simply $\mubnh \left[ \phi_{\tau_1} \phi_{\tau_2} \right]$.

\begin{theorem}
\label{thm:main_general} 
There are some constants $C, c > 0$ depending only on $b$ such that
if $\plat(\ham,T,\eps)$ holds as in Definition \ref{def:plateau},
then for all $\tau_1, \tau_2 \in [b]^n$ we have
\begin{equation}
    \left| \mubnh \left[ \phi_{\tau_1} \phi_{\tau_2} \right] \right|
    \leq C \Exp{- c \d(\tau_1, \tau_2) + C \left( T + \log \frac{1}{\eps} \right)},
\end{equation}
where the distance $\d(\tau_1,\tau_2)$ was defined above the statement of Theorem \ref{thm:main_brw}.
\end{theorem}

With this theorem in hand, in order to prove Theorems \ref{thm:main_brw}, \ref{thm:main_sinhgordon}, and \ref{thm:main_phialpha},
it just suffices to check the plateau condition for appropriate choices of $T \in \N$ and $\eps > 0$;
this will be done in Section \ref{sec:init}.
\section{Structure of the renormalization group flow}
\label{sec:flow}

In this section we prove Proposition \ref{prop:flow}.
The necessary ingredients will be laid out in the subsequent subsections,
and the following proof simply refers to the relevant statements which may
be checked quickly.

Following our conventions laid out at the beginning of Section \ref{sec:oop},
we will use $\rho$ to denote the variable of integration when taking expectations
with respect to a measure on $\R$.
For instance, $\cN_1^\phi[F(\rho)]$ denotes the expectation of $F(\rho)$
when $\rho \sim \cN_1^\phi$ is a Gaussian with mean $\phi$ and variance $1$,
and $\walk_k^\phi[F(\rho)]$ denotes the expectation when $\rho$ is distributed
according to the renormalized walk distribution
$\walk_k^\phi$ of Definition \ref{def:wk}.

\begin{proof}[Proof of Proposition \ref{prop:flow}]
For each $k \in \N$, we simply define
\begin{equation}
\label{eq:rem_def}
    \rem_k(\phi) \coloneqq (\bb_k-1) \frac{\phi^2}{2} - \pot_k(\phi),
\end{equation}
where $\pot_k$ is defined recursively in Definition \ref{def:vk},
and $\bb_k = \frac{b^{k+1}-1}{b^k-1}$.
So by definition we have \eqref{eq:structure},
and it only remains to show that this definition of $\rem_k$ satisfies
the four items of Proposition \ref{prop:flow}.
Item \ref{item:rem_initial} will be shown in \eqref{eq:rem_initial_proof}.
Items \ref{item:rem_symconleqbk} and \ref{item:rem_rec} will be shown
in Lemma \ref{lem:conv_rec}, except for the symmetry in item \ref{item:rem_symconleqbk}
which holds by induction: $\ham$ is symmetric, and Definition \ref{def:vk} retains
that symmetry for each $\pot_k$.
Finally, item \ref{item:rem_unimodal} will be shown in Lemma \ref{lem:unimodality}.
\end{proof}

\subsection{The first renormalized potential}
\label{sec:flow_init}

We begin by examining $\pot_1(\phi)$, the first renormalized potential.
By Definition \ref{def:vk}, with $\rho \sim \walk_1^\phi$, we have
\begin{equation}
    e^{- \frac{1}{b} \pot_1(\phi)}
    = \cN_1^\phi \left[ e^{- \ham(\rho)} \right].
\end{equation}
Now since we have
\begin{equation}
\label{eq:normal_rnd}
    \frac{\cN_1^\phi(d\rho)}{\cN_1^0(d\rho)}
    = \frac{e^{-\frac{(\rho-\phi)^2}{2}}}
    {e^{-\frac{\rho^2}{2}}}
    = e^{\phi \rho - \frac{\phi^2}{2}},
\end{equation}
we see, with $\rho \sim \cN_1^0$ now, that
\begin{equation}
    e^{-\frac{1}{b} \pot_1(\phi)}
    = e^{- \frac{\phi^2}{2}} \cdot \cN_1^0 \left[ e^{\phi \rho - \ham(\rho)} \right],
\end{equation}
and thus that
\begin{equation}
    \pot_1(\phi) = b \frac{\phi^2}{2} - b \log \cN_1^0 \left[ e^{\phi \rho - \ham(\rho)} \right].
\end{equation}
Since $\bb_1 = \frac{b^2-1}{b-1} = b+1$ so that $\bb_1 - 1 = b$, we thus find
by the definition \eqref{eq:rem_def} that
\begin{equation}
\label{eq:rem_initial_proof}
    \rem_1(\phi) = b \log \cN_1^0 \left[ e^{\phi \rho - \ham(\rho)} \right].
\end{equation}
This is exactly the form described in item \ref{item:rem_initial}
of Proposition \ref{prop:flow}.
\subsection{Derivatives of renormalized potentials via cumulants}
\label{sec:flow_derivatives}

Let us now exhibit a few recursive formulas for the
derivatives of the renormalized potentials in terms of the
walk measures $\walk_k^\phi$ of Definition \ref{def:wk}.
These formulas will allow us to prove the remaining items of
Proposition \ref{prop:flow}.
To begin with, since
\begin{equation}
    \frac{\walk_k^\phi(d\rho)}{\cN_1^\phi(d\rho)} \propto e^{-\pot_k(\rho)}
\end{equation}
for all $\phi$, the recursive Definition \ref{def:vk} of the renormalized potentials
as well as \eqref{eq:normal_rnd} yield that
\begin{equation}
    e^{-\frac{1}{b} \pot_{k+1}(\phi)}
    = \cN_1^\phi \left[ e^{- \pot_k(\rho)} \right] 
    = e^{-\frac{\phi^2}{2}} \cdot \cN_1^0 \left[ e^{\phi \rho - \pot_k(\rho)} \right] 
    \propto e^{-\frac{\phi^2}{2}} \cdot \walk_k^0 \left[ e^{\phi \rho} \right],
\end{equation}
where the proportionality constant does not depend on $\phi$.
Therefore we have
\begin{equation}
\label{eq:vklogw}
    \pot_{k+1}(\phi)
    = b \frac{\phi^2}{2} - b \log \walk_k^0 \left[ e^{\phi \rho} \right]
    + C,
\end{equation}
where $C$ is some constant which will be unimportant for our analysis.

Now to access the derivatives of $\pot_{k+1}(\phi)$ we use the fact that
$\log \walk_k^0 \left[ e^{\phi \rho} \right]$ is a cumulant generating
function, and so its derivatives may be expressed in terms of cumulants.
However, we will be taking a derivative at any $\phi \in \R$, not just $\phi=0$,
and so it will be useful to know that the following equality holds
for any measurable function $F : \R \to \R$:
\begin{equation}
    \frac{\walk_k^0 \left[ e^{\phi \rho} F(\rho) \right]}
    {\walk_k^0 \left[ e^{\phi \rho} \right]}
    = \frac{\cN_1^0 \left[ e^{\phi \rho - \pot_k(\rho)} F(\rho) \right]}
    {\cN_1^0 \left[ e^{\phi \rho - \pot_k(\rho)} \right]}
    = \frac{\cN_1^\phi \left[ e^{-\pot_k(\rho)} F(\rho) \right]}
    {\cN_1^\phi \left[ e^{-\pot_k(\rho)} \right]}
    = \walk_k^\phi \left[ F(\rho) \right].
\end{equation}
Now we can express the $\phi$-derivatives of
$\log \walk_k^0 \left[ e^{\phi \rho} \right]$
in terms of cumulants of $\walk_k^\phi$ as follows:
\begin{align}
\label{eq:cum1}
    \frac{d}{d\phi} \log \walk_k^0 \left[ e^{\phi \rho} \right]
    &= \frac{\walk_k^0 \left[ e^{\phi \rho} \rho \right]}
    {\walk_k^0 \left[ e^{\phi \rho} \right]}
    = \walk_k^\phi \left[ \rho \right], \\
\label{eq:cum2}
    \frac{d^2}{d\phi^2} \log \walk_k^0 \left[ e^{\phi \rho} \right]
    &= \frac{\walk_k^0 \left[ e^{\phi \rho} \rho^2 \right]}
    {\walk_k^0 \left[ e^{\phi \rho} \right]}
    - \left( \frac{\walk_k^0 \left[ e^{\phi \rho} \rho \right]}
    {\walk_k^0 \left[ e^{\phi \rho} \right]} \right)^2
    = \walk_k^\phi \left[ (\rho - \walk_k^\phi[\rho])^2 \right], \\
\label{eq:cum3}
    \frac{d^3}{d\phi^3} \log \walk_k^0 \left[ e^{\phi \rho} \right]
    &= \frac{\walk_k^0 \left[ e^{\phi \rho} \rho^3 \right]}
    {\walk_k^0 \left[ e^{\phi \rho} \right]}
    - 3 \frac{\walk_k^0 \left[ e^{\phi \rho} \rho^2 \right]
    \walk_k^0 \left[ e^{\phi \rho} \rho \right]}
    {\walk_k^0 \left[e^{\phi \rho} \right]^2}
    + 2 \left( \frac{\walk_k^0 \left[ e^{\phi \rho} \rho \right]}
    {\walk_k^0 \left[ e^{\phi \rho} \right]} \right)^3
    = \walk_k^\phi \left[ (\rho - \walk_k^\phi[\rho])^3 \right].
\end{align}
We may now access the derivatives of the remainder terms $\rem_k$
as defined in \eqref{eq:rem_def}.
We will not need to consider the first derivative in what follows.
For the second derivative, by \eqref{eq:vklogw} and \eqref{eq:cum2} we have
\begin{align}
    \rem_{k+1}''(\phi) &= (\bb_{k+1} - 1) - \pot_{k+1}''(\phi) \\
    &= \frac{b^{k+2}-1}{b^{k+1}-1} - 1 - b + b \cdot \frac{d^2}{d\phi^2} \log \walk_k^0 \left[ e^{\phi \rho} \right] \\
    &= \frac{b - b^{k+1}}{b^{k+1}-1} + b \cdot \walk_k^\phi \left[ (\rho - \walk_k^\phi[\rho])^2 \right].
\label{eq:2dr_cum}
\end{align}
Note that for $k \geq 1$ we have $\frac{b - b^{k+1}}{b^{k+1} - 1} = - \frac{b}{\bb_k}$,
and for $k = 0$ the fraction is $0$.
For the third derivative, by \eqref{eq:vklogw} and \eqref{eq:cum3} we have
\begin{equation}
    \rem_{k+1}'''(\phi)
    = - \pot_{k+1}'''(\phi)
    = b \cdot \walk_k^\phi \left[ (\rho - \walk_k^\phi[\rho])^3 \right].
\label{eq:3dr_cum}
\end{equation}
\subsection{Recursive bounds for the remainder term}
\label{sec:flow_blcr}

We will now use \eqref{eq:2dr_cum} and the Brascamp--Lieb inequality
\cite[Theorem 4.1]{BL76Extensions} (and \cite[Theorem 12]{HC18Dimension} for a nonsmooth version)
to derive the recursive upper bound for $\rem_k''$ in item \ref{item:rem_rec}
of Proposition \ref{prop:flow}.
A special case of this inequality, which is all we will need,
states that if a real-valued random
variable $\rho$ has density $e^{-f(\rho)}$ for some strictly convex function
$f : \R \to \R$, then we have
\begin{equation}
\label{eq:bl}
    \Var[\rho] \leq \E \left[ \frac{1}{f''(\rho)} \right].
\end{equation}
We will also use a special case of the Cram\'er--Rao inequality \cite[Section 32.3]{C99Mathematical}
(see also \cite{R45Information}), which is a partial reversal of \eqref{eq:bl}:
\begin{equation}
\label{eq:cr}
    \Var[\rho] \geq \frac{1}{\E[f''(\rho)]}.
\end{equation}
This will be useful for item \ref{item:rem_symconleqbk} of Proposition \ref{prop:flow},
in particular the nonnegativity of $\rem_k''$.

\begin{lemma}
\label{lem:conv_rec}
If $\ham$ is convex, then for all $k \in \N$ and $\phi \in \R$, we have
$0 \leq \rem_k''(\phi) \leq \bb_k-1$ and
\begin{equation}
\label{eq:rec_ub}
    \rem_{k+1}''(\phi) \leq
    \frac{b}{\bb_k} \cdot \walk_k^\phi \left[ \frac{\rem_k''(\rho)}{\bb_k - \rem_k''(\rho)} \right].
\end{equation}
\end{lemma}

\begin{proof}[Proof of Lemma \ref{lem:conv_rec}]
We proceed by induction.
First, by Definition \ref{def:wk},
the density $g$ of $\walk_0^\phi$ is proportional to
$e^{- \frac{(\rho-\phi)^2}{2} - \ham(\rho)}$,
and so $- \log g$ may be written as a smooth function with second derivative $\geq 1$,
plus the possibly nonsmooth convex function $\ham$.
Thus the nonsmooth Brascamp--Lieb inequality of \cite[Theorem 12]{HC18Dimension} (see also \cite{H04Convex}) yields
\begin{equation}
    0 \leq \walk_0^\phi \left[ (\rho - \walk_0^\phi[\rho])^2 \right]
    \leq 1.
\end{equation}
So by \eqref{eq:2dr_cum} with $k = 0$, we have
\begin{equation}
\label{eq:conv_basecase}
    0 \leq \rem_1''(\phi) \leq b = \bb_1 - 1.
\end{equation}
This completes the base case.
Now, assuming that $\rem_k''(\rho) \leq \bb_k - 1$ for all $\rho \in \R$,
we have $\pot_k''(\rho) \geq 0$ by \eqref{eq:rem_def},
and so the density of $\walk_k^\phi$, which is proportional to
$\rho \mapsto e^{-\frac{(\rho-\phi)^2}{2} - \pot_k(\rho)}$, is strictly log-concave.
So the Brascamp--Lieb inequality \eqref{eq:bl} yields
\begin{equation}
    \walk_k^\phi \left[ (\rho - \walk_k^\phi[\rho])^2 \right]
    \leq \walk_k^\phi \left[ \frac{1}{1 + \pot_k''(\rho)} \right]
    = \walk_k^\phi \left[ \frac{1}{\bb_k - \rem_k''(\rho)} \right].
\end{equation}
Therefore by \eqref{eq:2dr_cum} again, for $k \geq 1$ we have
\begin{align}
    \rem_{k+1}''(\phi)
    &\leq - \frac{b}{\bb_k}
    + b \cdot \walk_k^\phi \left[ \frac{1}{\bb_k - \rem_k''(\rho)} \right] \\
    &= \frac{b}{\bb_k} \cdot \walk_k^\phi \left[ \frac{\bb_k}{\bb_k - \rem_k''(\rho)}
    - \frac{\bb_k - \rem_k''(\rho)}{\bb_k - \rem_k''(\rho)} \right] \\
    &= \frac{b}{\bb_k} \cdot \walk_k^\phi \left[ \frac{\rem_k''(\rho)}{\bb_k - \rem_k''(\rho)} \right].
\end{align}
This proves \eqref{eq:rec_ub}.
Now since $x \mapsto \frac{x}{\bb_k - x}$ is increasing on $(-\infty,\bb_k)$ and
$\rem_k''(\rho) \leq \bb_k - 1$ for all $\rho \in \R$, we have
\begin{align}
    \rem_{k+1}''(\phi)
    &\leq \frac{b}{\bb_k} \frac{\bb_k-1}{\bb_k - (\bb_k - 1)} \\
    &= \frac{b \bb_k - b + \bb_k}{\bb_k} - 1 \\
    &= \frac{\frac{b^{k+2} - b}{b^k-1} - \frac{b^{k+1} - b}{b^k-1} + \frac{b^{k+1} - 1}{b^k-1}}
    {\frac{b^{k+1}-1}{b^k - 1}} - 1 \\
    &= \bb_{k+1} - 1. \label{eq:bk_rec}
\end{align}
This completes the inductive step for both upper bounds in the statement of the lemma.

As for the nonnegativity, we have already seen the base case in \eqref{eq:conv_basecase}.
Now, if $0 \leq \rem_k''(\rho) \leq \bb_k - 1$ for all $\rho \in \R$ then \eqref{eq:2dr_cum} and the
Cram\'er--Rao inequality \eqref{eq:cr} yield
\begin{equation}
    \rem_{k+1}''(\phi)
    \geq - \frac{b}{\bb_k} + \frac{b}{\walk_k^\phi \left[ 1 + \pot_k''(\rho) \right]}
    = - \frac{b}{\bb_k} + \frac{b}{\walk_k^\phi \left[ \bb_k - \rem_k''(\rho) \right]}
    \geq 0,
\end{equation}
which finishes the proof.
\end{proof}

\subsection{Unimodality of the second derivative}
\label{sec:flow_unimodality}

Now we turn to proving item \ref{item:rem_unimodal} of Proposition
\ref{prop:flow}, which states that the second derivative $\rem_k''$
is unimodal.
To prove this, we will show that the third derivative satisfies $\rem_k'''(\phi) \leq 0$
for $\phi \geq 0$ and $\rem_k'''(\phi) \geq 0$ for $\phi \leq 0$.
By symmetry and the fact that $\rem_k'''(\phi) = - \pot_k'''(\phi)$,
it suffices to show that $\pot_k'''(\phi) \geq 0$ for $\phi \geq 0$.

\begin{lemma}
\label{lem:unimodality}
If $\ham$ is confining as in Definition \ref{def:confining},
then for all $k \in \N$ and all
$\phi \geq 0$, we have $\pot_k'''(\phi) \geq 0$.
\end{lemma}

We will proceed by induction using the cumulant form \eqref{eq:3dr_cum}
of the third derivative.
To prove the negativity of the third cumulant of $\walk_k^\phi$, we use the
following ``single sign-change trick'' lemma.
The applicability of this lemma in the present context was brought to the attention of the
author by a publicly available large language model in May 2026.
One possible source for the lemma itself is as an easy special case of
\cite[Lemma B]{KN63Generalized}, but since it is simple we provide
the short proof for completeness.

\begin{lemma}
\label{lem:sst} 
Let $f(x) : [0,\infty) \to \R$ such that
$\int_0^\infty x f(x) \,dx = 0$, and such that $f(x)$ changes sign exactly once from
negative to positive.
More precisely, we require that there exists some $x_0 \geq 0$ such that
$f(x) \leq 0$ for $x < x_0$ and $f(x) \geq 0$ for $x > x_0$.
Then we have $\int_0^\infty x^3 f(x) \,dx \geq 0$.
\end{lemma}

\begin{proof}[Proof of Lemma \ref{lem:sst}]
We claim that $x^3 f(x) \geq x_0^2 x f(x)$ for all $x \in [0,\infty) \setminus \{x_0\}$.
Indeed, for $x < x_0$ we have $f(x) \leq 0$ and $x^2 < x_0^2$,
and for $x > x_0$ we have $f(x) \geq 0$ and $x^2 > x_0^2$.
Thus
\begin{equation}
    \int_0^\infty x^3 f(x) \,dx \geq x_0^2 \int_0^\infty x f(x) \,dx = 0
\end{equation}
as required.
\end{proof}

\begin{proof}[Proof of Lemma \ref{lem:unimodality}]
Fix some $k \in \N \cup \{0\}$ and $\phi \in \R$, and
let $g$ denote the density of the centered random variable
$\eta = \rho - \walk_k^\phi[\rho]$, where $\rho \sim \walk_k^\phi$, i.e.\
\begin{equation}
    g(\eta) \propto \Exp{- \frac{(\eta + \walk_k^\phi[\rho] - \phi)^2 }{2}
    - \pot_k (\eta + \walk_k^\phi[\rho])}.
\end{equation}
Then by \eqref{eq:3dr_cum} we have
\begin{equation}
    \pot_{k+1}'''(\phi) = - \int_{-\infty}^\infty \eta^3 g(\eta) \,d\eta
    = \int_0^\infty \eta^3 \left( g(-\eta) - g(\eta) \right) \,d\eta.
\end{equation}
Now letting $f = g(-\eta) - g(\eta)$, since the random variable
$\eta$ is centered we have $\int_0^\infty \eta f(\eta) \,d\eta = 0$.
So by Lemma \ref{lem:sst}, it suffices to show that there is some $\eta_0 \geq 0$
such that $f(\eta) \leq 0$ for $0 \leq \eta < \eta_0$ and $f(\eta) \geq 0$
for $\eta > \eta_0$.
For this, note that $f(\eta)$ has the same sign as
\begin{align}
    \log g(-\eta) - \log g(\eta)
    &= - \frac{((\walk_k^\phi[\rho] - \phi) - \eta)^2}{2}
    + \frac{((\walk_k^\phi[\rho] - \phi) + \eta)^2}{2}
    - \pot_k(\walk_k^\phi[\rho] - \eta)
    + \pot_k(\walk_k^\phi[\rho] + \eta) \\
    &= 2 (\walk_k^\phi[\rho] - \phi) \eta 
    - \pot_k(\walk_k^\phi[\rho] - \eta)
    + \pot_k(\walk_k^\phi[\rho] + \eta).
\label{eq:tobeconvex}
\end{align}
Note that even in the case $k=0$ where we may have $\pot_0(\phi) = \ham(\phi) = \infty$
for some $\phi$, the above makes sense for all $\eta \in J_{\walk_0^\phi[\rho]}$,
where $J_y$ is as in Definition \ref{def:confining}; for this we use the fact that $\walk_k^\phi[\rho] \geq 0$
for all $k$, since the density $h$ of $\walk_k^\phi$ is right-skewed, i.e.\ $h(\rho) \geq h(-\rho)$ for $\rho \geq 0$.
By its definition, outside of this interval we have $f(\eta) = 0$ in the case $k=0$.
So for the base case, Definition \ref{def:confining} shows that
\eqref{eq:tobeconvex} is convex on $J_{\walk_0^\phi[\rho]}$, which implies the single sign-change property
for $f$ since $f(0) = 0$ and $\int_0^\infty \eta f(\eta) = 0$;
thus we obtain $\pot_1'''(\phi) \geq 0$ for $\phi \geq 0$
by Lemma \ref{lem:sst}.

Now for the inductive case, we again aim to show \eqref{eq:tobeconvex} is convex in $\eta \geq 0$.
Its second derivative is
\begin{equation}
    \pot_k''(|\walk_k^\phi[\rho] + \eta|) - \pot_k''(|\walk_k^\phi[\rho] - \eta|)
\end{equation}
and we have $|\walk_k^\phi[\rho] + \eta| \geq |\walk_k^\phi[\rho] - \eta|$
for $\eta \geq 0$ (again using the fact that $\walk_k^\phi[\rho] \geq 0$),
so this is nonnegative if $\pot_k'''(\phi) \geq 0$ for $\phi \geq 0$.
Under this inductive assumption, Lemma \ref{lem:sst} yields
that $\pot_{k+1}'''(\phi) \geq 0$ for $\phi \geq 0$ as well, finishing the proof.
\end{proof}
\section{Decay of the remainder}
\label{sec:shrink}

In this section we prove Proposition \ref{prop:shrinking}.
As described in Section \ref{sec:oop_shrinking}, the second derivative
$\rem_k''$ of the remainder term looks like a plateau as illustrated
in Figure \ref{fig:plateau}, which shrinks horizontally as $k$ increases.
This horizontal shrinking will be proven in Section \ref{sec:shrink_linear}.
After the plateau has shrunk, a recursive uniform bound will be effective,
yielding uniform exponential decay; as this is somewhat easier, it will 
be discussed first in Section \ref{sec:shrink_exponential}.
Both behaviors will be based on the recursive upper bound \eqref{eq:rec_ub}
of Proposition \ref{prop:flow}, but with more careful analysis required
for the linear behavior.
Finally, in Section \ref{sec:shrink_combination}, we will combine the
linear and exponential behavior to prove Proposition \ref{prop:shrinking}.

\subsection{Eventual exponential decay}
\label{sec:shrink_exponential}

The following lemma giving a recursive uniform bound on $\rem_k''$
was previously alluded to in \eqref{eq:exp}.

\begin{lemma}
\label{lem:exp}
If there are some $K \in \N$ and $M < b-1$ for which
\begin{equation}
    \sup_{\phi \in \R} \rem_K''(\phi) \leq M,
\end{equation}
then for all $k \geq K$ we have
\begin{equation}
    \sup_{\phi \in \R} \rem_{K+k}''(\phi) \leq \frac{1}{b^{k-2}(\bb_K - 1 - M)}.
\end{equation}
\end{lemma}

\begin{proof}[Proof of Lemma \ref{lem:exp}]
First note that by item \ref{item:rem_rec} of Proposition \ref{prop:flow},
if
\begin{equation}
    \sup_{\phi \in \R} \rem_k''(\phi) \leq M'
\end{equation}
for some $M' < b < \bb_k$, then since $x \mapsto \frac{x}{\bb_k - x}$ is increasing
on $(-\infty,\bb_k)$, we have
\begin{equation}
\label{eq:exp_rec}
    \sup_{\phi \in \R} \rem_{k+1}''(\phi) \leq \frac{b}{\bb_k}
    \frac{M'}{\bb_k - M'} \leq \frac{M'}{b - M'},
\end{equation}
using again the fact that $\bb_k > b$.
Now for $k \geq 0$ let us define
\begin{align}
\label{eq:mkdef1}
    M_k &= \frac{M}{b^k - M \cdot \sum_{j=0}^{k-1} b^j} \\
    &= \frac{M}{b^k - M \cdot \frac{b^k-1}{b-1}} \\
    &= \frac{1}{b^k} \frac{(b-1) M}{b - 1 - M \cdot (1 - b^{-k})} \\
    &= \frac{1}{b^k} \frac{(\bb_k - 1) M}{\bb_k - 1 - M},
\label{eq:mkdef2}
\end{align}
using the fact that $\frac{b-1}{1 - b^{-k}} = \frac{b^{k+1} - b^k}{b^k - 1} = \bb_k - 1$.
We now claim that $M_k$ satisfies the recurrence
\begin{equation}
\label{eq:mkrec}
    M_{k+1} = \frac{M_k}{b - M_k},
    \qquad \text{with}
    \qquad M_0 = M.
\end{equation}
Indeed, the base case is given and we simply check using \eqref{eq:mkdef1} that
\begin{align}
    \frac{M_k}{b - M_k}
    &= \frac{\frac{M}{b^k - M \cdot \sum_{j=0}^{k-1} b^j}}
    {b - \frac{M}{b^k - M \cdot \sum_{j=0}^{k-1} b^j}} \\
    &= \frac{M}
    {b^{k+1} - b M \cdot \sum_{j=0}^{k-1} b^j - M} \\
    &= \frac{M}
    {b^{k+1} - M \cdot \sum_{j=0}^{k} b^j} \\
    &= M_{k+1}.
\end{align}
By \eqref{eq:exp_rec}, this shows that
\begin{equation}
    \sup_{\phi \in \R} \rem_{K+k}''(\phi) \leq M_k
\end{equation}
for $k \geq 0$.
Now, for all $k \geq 0$, the expression
\eqref{eq:mkdef2} and the fact that $M < b$, so that $M (\bb_k - 1) \leq b^2$, yield
\begin{equation}
    M_k \leq \frac{1}{b^k} \frac{b^2}{\bb_k - 1 - M}.
\end{equation}
Finally, using the fact that $\bb_k$ is decreasing in $k$, we may replace $\bb_k$ with
$\bb_K$ for all $k \geq K$ and we find that
\begin{equation}
    M_k \leq \frac{1}{b^{k-2}} \frac{1}{\bb_K - 1 - M},
\end{equation}
finishing the proof.
\end{proof}
\subsection{Monotonicity of spread in log-concavity}
\label{sec:shrink_tailcomparison}

In this subsection we present a general lemma which will allow us to control
tail probabilities via comparison with simpler measures which are more or less
log-concave than the measure we are after.
This will be useful in Sections \ref{sec:shrink_linear} and \ref{sec:shrink_combination}
below, as well as Sections \ref{sec:ou} and \ref{sec:init} later.
The idea for this lemma was provided by a publicly available large language model in August 2026;
as in the case of Lemma \ref{lem:sst} in Section \ref{sec:flow_unimodality}, one possible source
of related results seems to be the work of Karlin and Novikoff on convex inequalities \cite{KN63Generalized}.
Following our conventions, for $\cL$ a probability distribution on $\R$, we use $\cL[|\rho| \geq T]$ to denote
the probability that $\rho \sim \cL$ satisfies $|\rho| \geq T$.

\begin{lemma}
\label{lem:tailcomparison}
Suppose that $\cL_1$ and $\cL_2$ are two probability distributions on $\R$ with densities
\begin{equation}
    f_1(\rho) \propto e^{-\qua(\rho) + \ext_1(\rho)}
    \qquad \text{and} \qquad
    f_2(\rho) \propto e^{-\qua(\rho) + \ext_2(\rho)}
\end{equation}
respectively, for some continuous functions $\qua, \ext_1, \ext_2 : \R \to \R$.
If $\ext_1 - \ext_2$ is convex and symmetric, then for all $T \geq 0$ we have
$\cL_1 \left[ |\rho| \geq T \right] \geq \cL_2 \left[ |\rho| \geq T \right]$.
In particular, $\cL_1$ stochastically dominates $\cL_2$ in absolute value
(meaning $\rho_1 \sim \cL_1$ and $\rho_2 \sim \cL_2$ may be coupled so that $|\rho_1| \geq |\rho_2|$).
\end{lemma}

\begin{proof}[Proof of Lemma \ref{lem:tailcomparison}]
First, we have
\begin{equation}
    \cL_1 \left[ |\rho| \geq T \right] - \cL_2 \left[ |\rho| \geq T \right]
    = \cL_2 \left[ |\rho| < T \right] - \cL_1 \left[ |\rho| < T \right]
    = \int_{-T}^T \left( f_2(\rho) - f_1(\rho) \right) \,d\rho,
\end{equation}
and we aim to show that the right-hand side is nonnegative.
For this, note that
\begin{equation}
\label{eq:intderiv}
    \frac{d}{dt} \int_{-t}^t \left( f_2(\rho) - f_1(\rho) \right) \,d\rho = (f_2(t) - f_1(t)) + (f_2(-t) - f_1(-t)),
\end{equation}
and that the limit of the integral is $0$ as $t \to \infty$ since $f_1$ and $f_2$ are probability densities,
and also as $t \to 0$ since $f_1$ and $f_2$ are continuous.
Thus to show that the integral is nonnegative for all $t$ (in particular for $t=T$), it would suffice to show that the
derivative \eqref{eq:intderiv} changes sign exactly once from positive to negative, so the integral increases from $0$
and then decreases back to $0$ as $t$ increases from $0$ to $\infty$.

To prove this, let us define $g : \R \to \R$ as the logarithm of the ratio of the densities, i.e.\
\begin{equation}
    e^{g(t)} \coloneqq \frac{f_2(t)}{f_1(t)} \propto \Exp{\ext_2(t) - \ext_1(t)}.
\end{equation}
By our assumption, the function $g$ is concave and symmetric.
If $g$ is identically zero then the proof is finished, so let us suppose this is not the case.
Under this assumption, $g$ may not be nonnegative on all of $\R$ or nonpositive on all of $\R$,
since $f_1$ and $f_2$ are probability densities
and either of these conditions would imply that they do not have the same integral.
So $g$ takes both positive and negative values, and the concavity and symmetry of $g$ implies that there is a unique $t_* > 0$
for which $g(t) > 0$ when $|t| < t_*$ and $g(t) \leq 0$ when $|t| \geq t_*$.
So $f_2(t) - f_1(t) > 0$ for $|t| < t_*$ and $f_2(t) - f_1(t) \leq 0$ for $|t| \geq t_*$,
finishing the proof of the inequality between tail probabilities.

The stochastic domination in absolute value follows immediately from the inequality between tails
by standard facts, see e.g.\  \cite[Theorem 4.3]{R24Modern}.
\end{proof}
\subsection{Initial linear horizontal shrinking}
\label{sec:shrink_linear}

In this section we prove that the plateau of the remainder term shrinks linearly,
as illustrated in Figure \ref{fig:plateau}.
To make this precise, for any $\eps > 0$ and $k \in \N$, let us define the threshold
\begin{equation}
\label{eq:def_threshold}
    T_{k,\eps} \coloneqq \sup \left\{ \phi > 0 : \rem_k''(\phi) > \bb_k - 1 - \eps \right\}.
\end{equation}
The main result of this section is that under the renormalization group flow
this threshold decreases by a constant amount, which may be taken
to be $1$, after an appropriate decrease of the tolerance $\eps$ is allowed.

\begin{lemma}
\label{lem:linear}
There is some constant $\delta = \delta_b \in (0,1)$ such that for all $k \in \N$, if $\eps > 0$ satisfies $T_{k,\eps} \geq 1$
and $\eps T_{k,\eps}^2 \leq 1$, then we have $T_{k+1,\delta\eps} \leq T_{k,\eps} - 1$.
\end{lemma}

The main technical portion of the above lemma is in proving that the random variable $\rho \sim \walk_k^{T_{k,\eps}-1}$
has some uniformly positive chance to exceed $T_{k,\eps}$, as mentioned below Definition \ref{def:plateau}.
Heuristically, this is because $T_{k,\eps}-1$ is within the plateau, so the variable $\rho \sim \walk_k^{T_{k,\eps}-1}$
still has nonnegligible rightwards fluctuations from $T_{k,\eps}-1$.
The following lemma makes this precise.

\begin{lemma}
\label{lem:biggerthanT}
There is some constant $\delta = \delta_b > 0$ such that for all $k \in \N$, if $\eps > 0$ satisfies $T_{k,\eps} \geq 1$
and $\eps T_{k,\eps}^2 \leq 1$, then we have
\begin{equation}
    \walk_k^{T_{k,\eps}-1} \left[ \rho \geq T_{k,\eps} \right] \geq \delta.
\end{equation}
\end{lemma}

Let us now see how Lemma \ref{lem:biggerthanT} proves Lemma \ref{lem:linear}, after which we will prove
Lemma \ref{lem:biggerthanT} itself.

\begin{proof}[Proof of Lemma \ref{lem:linear}]
By the definition \eqref{eq:def_threshold} of the thresholds,
our aim is to show that there is some constant $\delta = \delta_b \in (0,1)$ for which
\begin{equation}
    \rem_{k+1}''(T_{k,\eps}-1) \leq \bb_{k+1} - 1 - \delta \eps.
\end{equation}
Recall the recursive upper bound from item \ref{item:rem_rec} of Proposition \ref{prop:flow},
which states that
\begin{equation}
    \rem_{k+1}''(T_{k,\eps} - 1) \leq \frac{b}{\bb_k} \cdot \walk_k^{T_{k,\eps}-1} \left[ \frac{\rem_k''(\rho)}{\bb_k - \rem_k''(\rho)} \right].
\end{equation}
Now recall that $\rem_k''(\rho) \leq \bb_k - 1$ for all $\rho$ by item \ref{item:rem_symconleqbk} of Proposition \ref{prop:flow}
and $0 \leq \rem_k''(\rho) \leq \bb_k - 1 - \eps$ for all $\rho \geq T_{k,\eps}$ by the definition \eqref{eq:def_threshold} as well as
the unimodality of $\rem_k''$ as stated in item \ref{item:rem_unimodal} of Proposition \ref{prop:flow}.
So, since $x \mapsto \frac{x}{\bb_k - x}$ is increasing on $[0,\bb_k-1]$, Lemma \ref{lem:biggerthanT} implies that
\begin{align}
    \rem_{k+1}''(T_{k,\eps}-1) &\leq \frac{b}{\bb_k} \left( (1-\delta) \frac{\bb_k - 1}{\bb_k - (\bb_k - 1)} + \delta \frac{\bb_k - 1 - \eps}{\bb_k - (\bb_k - 1 - \eps)} \right) \\
    &\leq \frac{b}{\bb_k} \left( (1-\delta) (\bb_k - 1) + \delta (\bb_k - 1 - \eps) \right) \\
    &= \frac{b (\bb_k - 1)}{\bb_k} - \frac{b \delta}{\bb_k} \eps \\
    &\leq \bb_{k+1} - 1 - \frac{b \delta}{b+1} \eps,
\label{eq:derivation1}
\end{align}
using at the last step the fact that $\bb_k \leq b+1$ as well as the fact that $\frac{b(\bb_k-1)}{\bb_k} = \bb_{k+1} - 1$ as
previously derived in \eqref{eq:bk_rec}.
Thus, after adjusting the value of the constant $\delta = \delta_b > 0$, the proof is finished;
note that we of course may ensure that $\delta < 1$ simply by reducing the value if this does not hold already.
\end{proof}

Now we turn to the proof of Lemma \ref{lem:biggerthanT}.
For this, it will be useful to replace the remainder $\rem_k$ with a function which represents the worst-case
scenario under the assumptions that we have.
Namely, the replacement $\trem_{k,\eps}$ will have $\trem_{k,\eps}''(\rho) = \bb_k - 1 - \eps$ for $|\rho| < T_{k,\eps}$
and $\trem_{k,\eps}''(\rho) = 0$ for $|\rho| > T_{k,\eps}$.
This is the worst case because reducing the value of the second derivative of the remainder term
pushes the renormalized potential closer to the Ornstein--Uhlenbeck mean-reverting potential,
thus decreasing the probability of upwards fluctuations.
The following lemma shows that we obtain a lower bound on the desired probability via a distribution with this adjusted remainder term.

\begin{lemma}
\label{lem:worstcase}
Suppose that $k \in \N$ and $\eps > 0$ are such that $T_{k,\eps} \geq 1$.
Define
\begin{equation}
\label{eq:trem_def}
    \trem_{k,\eps}(\rho) \coloneqq \begin{cases}
        \frac{1}{2} \left( \bb_k - 1 - \eps \right) \rho^2 & \text{for } |\rho| \leq T_{k,\eps}, \\
        \frac{1}{2} \left( \bb_k - 1 - \eps \right) T_{k,\eps}^2 + \left( \bb_k - 1 - \eps \right) T_{k,\eps} \left( |\rho| - T_{k,\eps} \right) & \text{for } |\rho| > T_{k,\eps},
    \end{cases}
\end{equation}
and define the measure $\twalk_{k,\eps}$ to have the following Radon--Nikodym derivative with respect to $\walk_k^{T_{k,\eps}-1}$:
\begin{equation}
\label{eq:twalk_def}
    \frac{\twalk_{k,\eps}(d\rho)}{\walk_k^{T_{k,\eps}-1}(d\rho)}
    \propto \Exp{\trem_{k,\eps}(\rho) - \rem_k(\rho)}.
\end{equation}
Then we have
\begin{equation}
    \walk_k^{T_{k,\eps}-1} \left[ \rho \geq T_{k,\eps} \right] \geq \frac{1}{2} \twalk_{k,\eps} \left[ \rho \geq T_{k,\eps} \right].
\end{equation}
\end{lemma}

\begin{proof}[Proof of Lemma \ref{lem:worstcase}]
First we will apply Lemma \ref{lem:tailcomparison} with $\cL_1 = \walk_k^{T_{k,\eps}-1}$ and $\cL_2 = \twalk_{k,\eps}$,
for which we may take $\ext_1 = \rem_k$ and $\ext_2 = \trem_{k,\eps}$.
By the definition \eqref{eq:def_threshold} of the threshold $T_{k,\eps}$ and the definition \eqref{eq:trem_def} of $\trem_{k,\eps}$,
we observe that the derivative of $\rem_k - \trem_{k,\eps}$ is increasing, so that $\ext_1 - \ext_2$ is convex.
This difference is also symmetric as both $\rem_k$ and $\trem_{k,\eps}$ are; $\trem_{k,\eps}$ by definition
and $\rem_k$ by item \ref{item:rem_symconleqbk} of Proposition \ref{prop:flow}.
So Lemma \ref{lem:tailcomparison} implies that
\begin{equation}
    \walk_k^{T_{k,\eps}-1} \left[ |\rho| \geq T_{k,\eps} \right] \geq \twalk_{k,\eps} \left[ |\rho| \geq T_{k,\eps} \right].
\end{equation}
Now to conclude we note that since $T_{k,\eps} \geq 1$, the distribution $\walk_k^{T_{k,\eps} - 1}$ is right-skewed, i.e.\ its density
$f$ satisfies $f(\rho) \geq f(-\rho)$ for all $\rho \geq 0$.
Therefore
\begin{equation}
    \walk_k^{T_{k,\eps} - 1} \left[ \rho \geq T_{k,\eps} \right]
    \geq 
    \frac{1}{2} \walk_k^{T_{k,\eps} - 1} \left[ |\rho| \geq T_{k,\eps} \right]
    \geq 
    \frac{1}{2} \twalk_{k,\eps} \left[ |\rho| \geq T_{k,\eps} \right]
    \geq
    \frac{1}{2} \twalk_{k,\eps} \left[ \rho \geq T_{k,\eps} \right],
\end{equation}
finishing the proof.
\end{proof}

Lemma \ref{lem:worstcase} reduces Lemma \ref{lem:biggerthanT}
to lower bounding $\twalk_{k,\eps} \left[ \rho \geq T_{k,\eps} \right]$;
we now turn to this.

\begin{proof}[Proof of Lemma \ref{lem:biggerthanT}]
For ease of notation, let us set $T = T_{k,\eps}$ throughout this proof.
Now first note that by the definition \eqref{eq:twalk_def} of $\twalk_{k,\eps}$
as well as Definition \ref{def:wk} of $\walk_k^{T-1}$ and the representation
\eqref{eq:structure} of the renormalized potential $\pot_k$,
the density $f_{k,\eps}$ of $\twalk_{k,\eps}$ satisfies
\begin{align}
    f_{k,\eps}(\rho) &\propto \Exp{- \frac{(\rho-(T-1))^2}{2} - (\bb_k - 1)\frac{\rho^2}{2} + \trem_{k,\eps}(\rho)} \\
    &\propto \Exp{ - \frac{\bb_k}{2} \left( \rho - \frac{T-1}{\bb_k} \right)^2 + \trem_{k,\eps}(\rho)},
\end{align}
by a similar derivation as in \eqref{eq:ou_source},
where the proportionality constant may depend on $T$, but not $\rho$.
Since $T \geq 1$ and $\trem_k$ is symmetric, this satisfies $f_{k,\eps}(\rho) \geq f_{k,\eps}(-\rho)$,
for all $\rho \geq 0$, meaning that, since $T \geq 0$, we have $\twalk_{k,\eps}\left[\rho \leq - T\right] \leq \twalk_{k,\eps}\left[ \rho \geq T \right]$.
So we find that
\begin{equation}
\label{eq:twalk_integrals}
    \twalk_{k,\eps} \left[ \rho \geq T \right]
    = \frac{I_{[T,\infty)}}{I_{(-\infty,-T]} + I_{[-T,T]} + I_{[T,\infty)}}
    \geq \frac{I_{[T,\infty)}}{I_{[-T,T]} + 2 I_{[T,\infty)}},
\end{equation}
where we have defined
\begin{equation}
    I_J = \int_J \Exp{- \frac{\bb_k}{2} \left( \rho - \frac{T-1}{\bb_k} \right)^2 + \trem_{k,\eps}(\rho)} \,d\rho
\end{equation}
for any interval $J$.

Thus to finish the proof it suffices to show that there is some constant $C$ for which $I_{[-T,T]} \leq C I_{[T,\infty)}$.
To show this, let us recall the definition of $\trem_{k,\eps}$ from \eqref{eq:trem_def} (noting again that we
have identified $T = T_{k,\eps}$), and observe that
\begin{align}
    I_{[-T,T]} &= \int_{-T}^T \Exp{-\frac{\bb_k}{2} \left( \rho - \frac{T-1}{\bb_k} \right)^2 + \frac{\bb_k - 1 - \eps}{2} \rho^2 } \,d\rho, \\
    \text{and} \qquad I_{[T,\infty)} &= \int_T^\infty \Exp{- \frac{\bb_k}{2} \left( \rho - \frac{T-1}{\bb_k} \right)^2 - \frac{\bb_k - 1 - \eps}{2} T^2  + ( \bb_k - 1 - \eps ) T \rho } \,d\rho.
\end{align}
Now let us expand the expressions inside of these exponentials.
For $I_{[-T,T]}$, we have
\begin{align}
    &- \frac{\bb_k}{2} \left( \rho - \frac{T-1}{\bb_k} \right)^2 + \frac{\bb_k - 1 - \eps}{2} \rho^2 \\
    &\qquad \qquad = - \frac{1 + \eps}{2} \rho^2 + \rho (T-1) - \frac{(T-1)^2}{2 \bb_k} \\
    &\qquad \qquad = - \frac{1 + \eps}{2} \left( \rho - \frac{T-1}{1+\eps} \right)^2 + \frac{(T-1)^2}{2(1+\eps)} - \frac{(T-1)^2}{2 \bb_k} \\
    &\qquad \qquad = - \frac{1 + \eps}{2} \left( \rho - \frac{T-1}{1+\eps} \right)^2 + \frac{\bb_k-1}{2 \bb_k} (T-1)^2 + O(\eps T^2).
\label{eq:ITT_simplified}
\end{align}
And for $I_{[T,\infty)}$ we have
\begin{align}
    &- \frac{\bb_k}{2} \left( \rho - \frac{T-1}{\bb_k} \right)^2 - \frac{\bb_k - 1 - \eps}{2} T^2 + (\bb_k - 1 - \eps) T \rho \\
    &\qquad \qquad = - \frac{\bb_k}{2} \rho^2 + \rho (T-1) - \frac{(T-1)^2}{2\bb_k} - \frac{\bb_k - 1}{2} T^2 + (\bb_k - 1 - \eps) T \rho + O(\eps T^2) \\
    &\qquad \qquad = - \frac{\bb_k}{2} \rho^2 + \rho (\bb_k T - 1 - \eps T ) - \frac{(T-1)^2}{2\bb_k} - \frac{\bb_k -1}{2} T^2 + O(\eps T^2) \\
    &\qquad \qquad = - \frac{\bb_k}{2} \left( \rho - \left( T - \frac{1 + \eps T}{\bb_k} \right) \right)^2
        + \frac{\bb_k}{2} \left( T - \frac{1}{\bb_k} \right)^2 - \frac{(T-1)^2}{2 \bb_k} - \frac{\bb_k - 1}{2} T^2 + O(\eps T^2) \\
    &\qquad \qquad = - \frac{\bb_k}{2} \left( \rho - \left( T - \frac{1 + \eps T}{\bb_k} \right) \right)^2
        - \frac{(T-1)^2}{2 \bb_k} + \frac{1}{2} T^2 - T + \frac{1}{2\bb_k} + O(\eps T^2) \\
    &\qquad \qquad = - \frac{\bb_k}{2} \left( \rho - \left( T - \frac{1 + \eps T}{\bb_k} \right) \right)^2
        - \frac{(T-1)^2}{2\bb_k} + \frac{(T-1)^2}{2} - \frac{1}{2} + \frac{1}{2\bb_k} + O(\eps T^2) \\
    &\qquad \qquad = - \frac{\bb_k}{2} \left( \rho - \left( T - \frac{1 + \eps T}{\bb_k} \right) \right)^2
        + \frac{\bb_k - 1}{2\bb_k} (T-1)^2 - \frac{\bb_k - 1}{2\bb_k} + O(\eps T^2).
\label{eq:ITinf_simplified}
\end{align}
Note that the $\rho$-independent term in \eqref{eq:ITinf_simplified} differs from that in \eqref{eq:ITT_simplified} only by the constant $\frac{\bb_k - 1}{2 \bb_k}$.
So using these formulas we find that
\begin{equation}
    \frac{I_{[-T,T]}}{I_{[T,\infty)}}
    = \Exp{\frac{\bb_k - 1}{2 \bb_k} + O(\eps T^2)} \cdot
    \frac{\int_{-T}^T \Exp{- \frac{1+\eps}{2} \left( \rho - \frac{T-1}{1+\eps} \right)^2} \,d\rho}
    {\int_T^\infty \Exp{- \frac{\bb_k}{2} \left( \rho - \left( T - \frac{1+\eps T}{\bb_k} \right) \right)^2} \,d\rho}.
\end{equation}
Now the fraction in the right-hand side is a ratio of unnormalized probabilities of events for Gaussian random variables.
With this interpretation, if $\zeta \sim \cN_1^0$ is a standard normal random variable, we have
\begin{align}
    \frac{I_{[-T,T]}}{I_{[T,\infty)}}
    &= \Exp{\frac{\bb_k-1}{2 \bb_k} + O(\eps T^2)} \cdot \sqrt{\frac{\bb_k}{1+\eps}} \cdot
    \frac{\cN_1^0 \left[ \left| \frac{\zeta + T - 1}{1 + \eps} \right| \leq T \right]}
    {\cN_1^0 \left[ \frac{\zeta}{\bb_k} + T - \frac{1-\eps T}{\bb_k} > T \right]} \\
    &\leq \frac{\Exp{\frac{1}{2} + C \eps T^2} \cdot \sqrt{b+1}}{\cN_1^0 \left[ \zeta > 1 + \eps T \right]},
\end{align}
using the fact that $\bb_k \leq b+1$ for all $k$ as well as the fact that the probability in the numerator is $\leq 1$.
We have also upper bounded the $O(\eps T^2)$ term with $C \eps T^2$ for some constant $C = C_b$.
Finally, since $\eps T^2 \leq 1$ and $T \geq 1$, we have $\eps T \leq 1$ as well, so the probability in the denominator
is at least $\P \left[ \zeta > 2 \right]$, which is an absolute constant.
So we find that the ratio $I_{[-T,T]} / I_{[T,\infty)}$ is bounded above by some $b$-dependent constant,
which by \eqref{eq:twalk_integrals} shows that
\begin{equation}
    \twalk_{k,\eps}\left[ \rho \geq T \right] \geq \delta
\end{equation}
for some constant $\delta = \delta_b$, and thus the proof is finished by Lemma \ref{lem:worstcase}.
\end{proof}
\subsection{Combining the behaviors to obtain a bound}
\label{sec:shrink_combination}

In this subsection we combine Lemmas \ref{lem:linear} and \ref{lem:exp}
to prove Proposition \ref{prop:shrinking}, which states that there are some constants $C = C_b > 0$
and $\delta = \delta_b > 0$ such that if $\eps > 0$ and $T \in \N$ are such that $\eps T^2 \leq 1$
and $\rem_1''(T) \leq b - \eps$, then for all $k \geq 0$ we have
\begin{equation}
    \sup_{\phi \in \R} \rem_{2T+2+k}''(\phi) \leq \frac{C}{b^k \delta^T \eps}.
\end{equation}
We begin with a fact which results from iterating Lemma \ref{lem:linear} to its natural conclusion.

\begin{lemma}
\label{lem:linear_iterated}
If $\eps > 0$ and $T \in \N$ are such that $\eps T^2 \leq 1$ and $\rem_1''(T) \leq b - \eps$, then
\begin{equation}
\label{eq:iteration_goal}
    \sup_{\phi \in \R} \rem_{T+1}''(\phi) \leq \bb_{T+1} - 1 - \delta^{T+1} \eps.
\end{equation}
\end{lemma}

\begin{proof}[Proof of Lemma \ref{lem:linear_iterated}]
By the definition \eqref{eq:def_threshold} of the threshold $T_{1,\eps}$ and the fact that $\bb_1 - 1 = b$,
we have $T_{1,\eps} \leq T$, and so $\eps T_{1,\eps}^2 \leq 1$.
If we also have $T_{1,\eps} \geq 1$ then we may apply Lemma \ref{lem:linear}
and we find that $T_{2,\delta\eps} \leq T_{1,\eps} - 1$.
Since $\delta \in (0,1)$, we have $\delta \eps \leq \eps$ and so
\begin{equation}
    \delta \eps T_{2,\delta \eps}^2 \leq \eps (T_{1,\eps} - 1)^2 \leq \eps T_{1,\eps}^2 \leq 1,
\end{equation}
using the fact that $T_{2,\delta} \geq 0$ by definition and $T_{1,\eps} \geq 1$ by assumption.
So if $T_{2,\delta \eps} \geq 1$ then we may repeat this argument; eventually there will be some
first integer $K \geq 1$ for which $T_{K,\delta^{K-1} \eps} < 1$.
Since $T_{k+1,\delta^k \eps} \leq T_{k, \delta^{k-1} \eps} - 1$ for all $k < K$, we must have
$K \leq \lceil T_{1,\eps} \rceil \leq T$, since $T$ is an integer.

Now by items \ref{item:rem_rec} and \ref{item:rem_unimodal} of Proposition \ref{prop:flow}, we have
\begin{equation}
\label{eq:supub}
    \sup_{\phi \in \R} \rem_{K+1}''(\phi)
    = \rem_{K+1}''(0)
    \leq \frac{b}{\bb_K} \cdot \walk_K^0 \left[ \frac{\rem_K''(\rho)}{\bb_K - \rem_K''(\rho)} \right],
\end{equation}
and we claim that there is some constant $\eta = \eta_b > 0$ for which
\begin{equation}
\label{eq:spread_later}
    \walk_K^0 \left[ |\rho| \geq 1 \right] \geq \eta.
\end{equation}
Assuming \eqref{eq:spread_later} for now, from \eqref{eq:supub} and the definition of
$T_{K, \delta^{K-1} \eps}$, since $x \mapsto \frac{x}{\bb_K - x}$ is increasing on $[0,\bb_K - 1]$
and $\rem_K''(\rho) \in [0,\bb_K - 1]$ for all $\rho \in \R$, we have
\begin{align}
    \sup_{\phi \in \R} \rem_{K+1}''(\phi) &\leq
    \frac{b}{\bb_K} \left( (1-\eta) \frac{\bb_K - 1}{\bb_K - (\bb_K - 1)} +
    \eta \frac{\bb_K - 1 - \delta^{K-1} \eps}{\bb_K - (\bb_K - 1 - \delta^{K-1} \eps)} \right) \\
    &\leq \bb_{K+1} - 1 - \frac{b \eta}{b + 1} \delta^{K-1} \eps
\end{align}
using the same derivation as for \eqref{eq:derivation1}.
We also have
\begin{align}
    \sup_{\phi \in \R} \rem_{K+2}''(\phi)
    &\leq \frac{b}{\bb_{K+1}}
    \frac{\bb_{K+1} - 1 - \frac{b \eta}{b+1} \delta^{K+1} \eps}
    {\bb_{K+1} - (\bb_{K+1} - 1 - \frac{b \eta}{b+1} \delta^{K-1} \eps)} \\
    &\leq \bb_{K+2} - 1 - \frac{b}{b+1} \frac{b \eta}{b+1} \delta^{K-1} \eps,
\end{align}
and, iterating this,
\begin{equation}
    \sup_{\phi \in \R} \rem_{K+k}''(\phi)
    \leq \bb_{K+k} - 1 - \left( \frac{b}{b+1} \right)^k \eta \delta^{K-1} \eps.
\end{equation}
Since $K \leq T$, adjusting $\delta$ to be the minimum of $\eta$, $\frac{b}{b+1}$, and the previous value of $\delta$,
we obtain \eqref{eq:iteration_goal}.

It just remains to show \eqref{eq:spread_later}.
For this, we will apply Lemma \ref{lem:tailcomparison}.
By Definition \ref{def:wk} of $\walk_K^0$ as well as the representation \eqref{eq:structure}
of the renormalized potential $\pot_k$, the density $f_K$ of $\walk_K^0$ satisfies
\begin{equation}
    f_K(\rho) \propto \Exp{- \frac{\rho^2}{2} - \frac{\bb_K - 1}{2} \rho^2 + \rem_K(\rho)}
    = \Exp{- \frac{\bb_K}{2} \rho^2 + \rem_K(\rho)}.
\end{equation}
So, since $\rem_K$ is convex by item \ref{item:rem_symconleqbk} of Proposition \ref{prop:flow},
Lemma \ref{lem:tailcomparison} applies with $\cL_1 = \walk_K^0$ and $\cL_2 = \cN_{\bb_K^{-1}}^0$,
a centered Gaussian distribution with variance $\bb_K^{-1}$.
Since $\bb_K \leq b+1$, this yields
\begin{equation}
    \walk_K^0 \left[ |\rho| \geq 1 \right] \geq \cN_{\bb_K^{-1}}^0 \left[ |\rho| \geq 1 \right]
    \geq \eta
\end{equation}
for some constant $\eta = \eta_b > 0$, finishing the proof.
\end{proof}

Now we apply the exponential decay upper bound of Lemma \ref{lem:exp} after 
the plateau has decayed linearly.

\begin{proof}[Proof of Proposition \ref{prop:shrinking}]
By Lemma \ref{lem:linear_iterated}, we have
\begin{equation}
    \sup_{\phi \in \R} \rem_{T+1}''(\phi) \leq \bb_{T+1} - 1 - \delta^{T+1} \eps,
\end{equation}
and so we may apply Lemma \ref{lem:exp}, which yields
\begin{equation}
    \sup_{\phi \in \R} \rem_{T+1+k}''(\phi) \leq \frac{1}{b^{k-2} \delta^{T+1} \eps}
\end{equation}
for all $k \geq T+1$.
Rewriting $k = T+1+j$ with $j \geq 0$, we have
\begin{equation}
    \sup_{\phi \in \R} \rem_{2(T+1) + j}''(\phi) \leq \frac{1}{b^j b^{T+1-2} \delta^{T+1} \eps},
\end{equation}
and so the proof is finished after changing the value of $\delta$ to $b \delta$ and setting $C = \frac{b}{\delta}$.
\end{proof}
\section{Ornstein--Uhlenbeck couplings}
\label{sec:ou}

In this section we prove Proposition \ref{prop:coupling},
showing that the coupling of Definition \ref{def:joint} between the tree-indexed
Markov chain $\left( \phi_\tau \right)_{\tau \in \treebn}$ of Definition \ref{def:treemc} and 
a branching Ornstein--Uhlenbeck (OU) walk $\left( \psi_\tau \right)_{\tau \in \treebn}$ satisfies
\begin{equation}
    \P \left[ \phi_{\tau'} = \psi_{\tau'} \text{ for all } \tau' \preceq \tau \text{ with }
    |\tau'| \leq n - k \right] \geq
    1 - C e^{- ck + C (T + \log \frac{1}{\eps})}
\end{equation}
for all $k \in [n]$ and $\tau \in [b]^n$, under the assumption of $\plat(\ham,T,\eps)$
as in Definition \ref{def:plateau}.
We will also prove Proposition \ref{prop:tailbound}, giving the following Gaussian tail bound
under the same assumption:
\begin{equation}
    \mubnh \left[ |\phi_\tau| \geq s \right] \leq C \Exp{- \frac{c s^2}{T + \log \frac{1}{\eps}}}.
\end{equation}
This tail bound will arise from a stochastic domination coupling with a modified OU
walk whose last $C (T + \log \frac{1}{\eps})$ increments are standard Gaussians instead.
Finally, we will combine these two propositions to obtain Theorem \ref{thm:main_general},
the general version of our main result, giving exponential decay of correlations under $\mubnh$
with correlation length $O\left( T + \log \frac{1}{\eps} \right)$.

We begin by proving Proposition \ref{prop:tailbound} in Section \ref{sec:ou_tailbound}.
Then in Section \ref{sec:ou_steps} we will bound the total variation distance between individual
steps of the tree-indexed Markov chains, which we will then use to prove Proposition \ref{prop:coupling}
in Section \ref{sec:ou_coupling}.
Finally, we will prove Theorem \ref{thm:main_general} in Section \ref{sec:ou_covariance}.

Throughout this section, for any $\tau \in [b]^n$ and any $k \in \{0,\dotsc,n\}$, we will use
$\tau^k$ to denote the initial segment of length $k$ in $\tau$, i.e.\ if $\tau = (\tau_1, \dotsc, \tau_n)$
then $\tau^k = (\tau_1, \dotsc, \tau_k)$, and in particular $\tau^0 = \emptyset$.
Additionally, the reader should recall that $\cN_{\sigma^2}^\mu$ denotes a normal distribution with mean $\mu$
and variance $\sigma^2$, and recall from Definition \ref{def:oud} that
\begin{equation}
    \oud_a^\phi = \cN_{a^{-1}}^{a^{-1} \phi}.
\end{equation}

\subsection{Gaussian tail bound}
\label{sec:ou_tailbound}

Before we proceed, it will be helpful presently and throughout the remainder of this section
to know that the stationary distribution of the OU walk $\left( \zeta_k \right)$
with increments given by $\zeta_k \sim \oud_a^{\zeta_{k-1}}$ is $\cN_{a/(a^2-1)}^0$ for any $a > 1$,
and furthermore that this walk is reversible with respect to this distribution.

\begin{lemma}
\label{lem:oud_stationary}
For $a > 1$, suppose that $\zeta_1 \sim \cN_{a/(a^2-1)}^0$ and that, conditionally on $\zeta_1$,
we have $\zeta_2 \sim \oud_a^{\zeta_1}$.
Then the marginal distribution of $\zeta_2$ is also $\cN_{a/(a^2-1)}^0$.
Furthermore, conditionally on $\zeta_2$, we have $\zeta_1 \sim \oud_a^{\zeta_2}$.
\end{lemma}

\begin{proof}[Proof of Lemma \ref{lem:oud_stationary}]
We have
\begin{equation}
    \zeta_2 \stackrel{d}{=} \frac{\zeta_1}{a} + \frac{\zeta}{\sqrt{a}},
\end{equation}
where $\zeta \sim \cN_1^0$ is independent from $\zeta_1$.
Thus $\zeta_2$ is a Gaussian with mean $0$ and variance
\begin{equation}
    \frac{1}{a^2} \frac{a}{a^2-1} + \frac{1}{a}
    = \frac{a}{a^2(a^2-1)} + \frac{a(a^2-1)}{a^2(a^2-1)}
    = \frac{a}{a^2-1},
\end{equation}
proving the stationarity.
Now to prove reversibility, note that $\zeta_1, \zeta_2$ are jointly Gaussian
and satisfy
\begin{equation}
    \Cov[\zeta_1, \zeta_2] = \E \left[ \zeta_1 \left( \frac{\zeta_1}{a} + \frac{\zeta}{\sqrt{a}} \right) \right]
    = \frac{1}{a}.
\end{equation}
So, conditional on $\zeta_2$ we have
\begin{equation}
    \zeta_1 \stackrel{d}{=} \frac{\zeta_2}{a} + x \zeta'
\end{equation}
for some $x \geq 0$, where $\zeta' \sim \cN_1^0$ is independent from $\zeta_2$.
Since we know the variance of $\zeta_1$ and $\zeta_2$, we find that $x$ must satisfy
\begin{equation}
    \frac{a}{a^2 - 1} = \frac{1}{a^2} \frac{a}{a^2-1} + x^2,
\end{equation}
and solving this yields $x = \frac{1}{\sqrt{a}}$.
\end{proof}

As mentioned above, to prove the tail bound of Proposition \ref{prop:tailbound}
we will construct a stochastic domination coupling by a modified OU
walk $\left( \xi_k \right)_{k=0}^n$ defined as follows for some $C, \eps > 0$ and $T \in \N$.
First, $\xi_0 \sim \cN_{(b-\frac{1}{2})/\left((b-\frac{1}{2})^2 - 1 \right)}^0$ and inductively for $k = 1, \dotsc, n$ in order we sample
\begin{equation}
    \xi_k \sim \begin{cases}
        \oud_{b-\frac{1}{2}}^{\xi_k-1} & \text{if } k \leq n - \left\lceil C \left( T + \log \frac{1}{\eps} \right) \right\rceil, \\
        \cN_1^{\xi_{k-1}} & \text{if } k > n - \left\lceil C \left( T + \log \frac{1}{\eps} \right) \right\rceil.
    \end{cases}
\end{equation}
The choice of $b - \frac{1}{2}$ above is incidental; any constant in the open interval $(1,b)$ would suffice.

\begin{lemma}
\label{lem:sd_coupling}
There is some $C = C_b > 0$ such that if $\plat(\ham,T,\eps)$ holds
with $T \in \N$ and $\eps > 0$ which satisfy $\eps T^2 \leq 1$, then
for any $\tau \in [b]^n$ the random variable $|\phi_\tau|$ is stochastically
dominated by $|\xi_n|$, where $\left( \phi_\tau \right)_{\tau \in [b]^n} \sim \mubnh$
and $\left( \xi_k \right)_{k=0}^n$ is as just above.
\end{lemma}

The following fact about Gaussians will be helpful to prove both Lemma \ref{lem:sd_coupling}
and Proposition \ref{prop:coupling} in Section \ref{sec:ou_coupling} below.
As the proof shows, it in fact holds for any symmetric continuous distribution on $\R$,
but we will only need it for $\cN_1^0$.

\begin{lemma}
\label{lem:normal_domination}
If $|a_1| \geq |a_2|$, then $\zeta_1, \zeta_2 \sim \cN_1^0$ may be coupled
so that $|a_1 + \zeta_1| \geq |a_2 + \zeta_2|$ almost surely.
\end{lemma}

\begin{proof}[Proof of Lemma \ref{lem:normal_domination}]
The densities $f_1$ and $f_2$ of $a_1 + \zeta_1$ and $a_2 + \zeta_2$ are plotted in Figure \ref{fig:gaussians}.
Let us sample a uniformly random point $(x,y)$ from the area under the curve $f_1$; if it also lies under the curve
$f_2$, then we set $a_1 + \zeta_1 = x = a_2 + \zeta_2$.
If not, then we set $a_1 + \zeta_1 = x$ and $a_2 + \zeta_2 = a_1+a_2 - x$; see Figure \ref{fig:gaussians}.
Then $\zeta_1$ and $\zeta_2$ are both distributed according to $\cN_1^0$.
In the first case we have $|a_1 + \zeta_1| = |a_2 + \zeta_2|$.
In the second case the signs of $x$, $a_1$, and $a_1 + a_2$ all agree,
using the fact that $|a_1| \geq |a_2|$.
So we have
\begin{equation}
    \left| a_1 + a_2 - x \right| \leq \min \left\{ \left| a_1 + a_2 \right|, |x| \right\}
    \leq |x|,
\end{equation}
which finishes the proof.
\end{proof}

\begin{figure}
    \centering
    \includegraphics[width=0.95\textwidth]{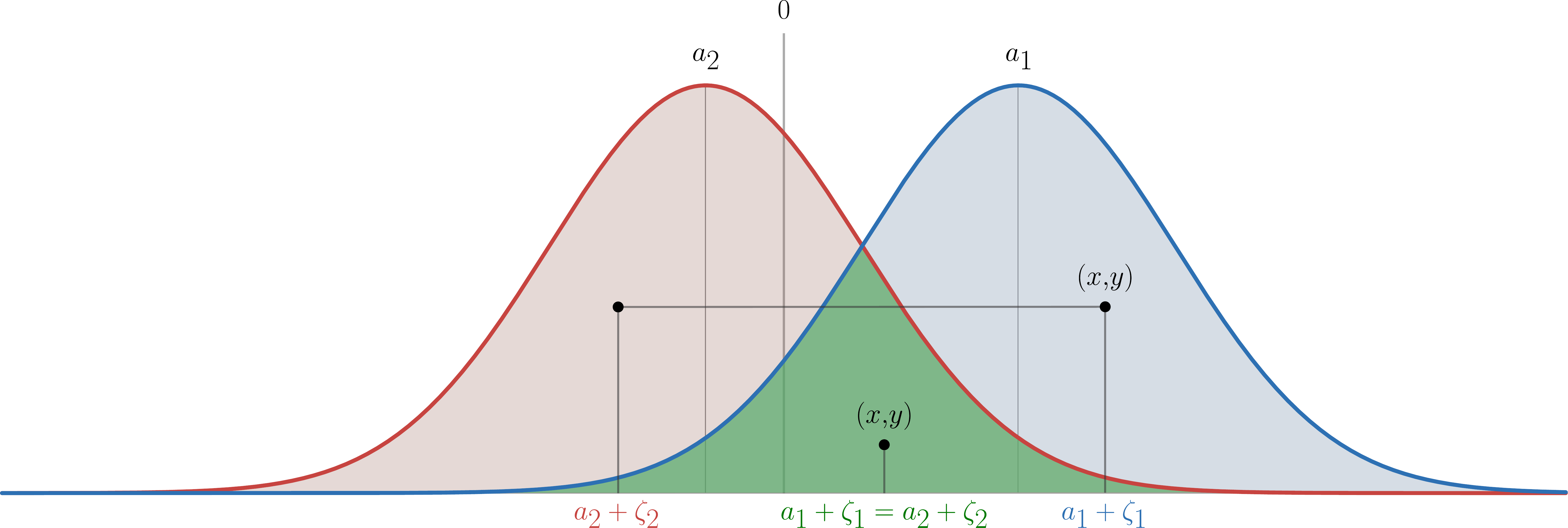}
    \caption{
    The coupling described in the proof of Lemma \ref{lem:normal_domination}.
    A point $(x,y)$ is sampled under the density of $\cN_1^{a_1}$ (blue and green area).
    If it lies under both densities (green area), then we set $a_1 + \zeta_1 = x = a_2 + \zeta_2$.
    Otherwise, we reflect the point to the red area, and we still obtain $|a_1 + \zeta_1| \geq |a_2 + \zeta_2|$.
    The same coupling also shows that the total variation distance between
    $\cN_1^{a_1}$ and $\cN_1^{a_2}$ is at most $1 - 2 \cdot \cN_1^0 \left[ \zeta \geq \frac{|a_1-a_2|}{2} \right]$,
    which will be used in the proof of Lemma \ref{lem:ou_coalescence}.
    }
    \label{fig:gaussians}
\end{figure}

\begin{proof}[Proof of Lemma \ref{lem:sd_coupling}]
We will build the stochastic domination coupling by induction.
By Proposition \ref{prop:shrinking}, the definition of $\bb_\ell$,
and the fact that $T \in \N$ so that $T \geq 1$, there is some $C = C_b > 0$ such that
for all $\ell \in \{0,\dotsc,n\}$ with
$\ell \geq \left\lceil C \left( T + \log \frac{1}{\eps} \right) \right\rceil$
we have $\inf_{\phi \in \R} \pot_\ell''(\phi) \geq b - \frac{3}{2}$.
So we find that
\begin{equation}
    - \left( b - \frac{3}{2} \right) \frac{\rho^2}{2}
    + \pot_\ell(\rho)
\end{equation}
is convex and symmetric for such $\ell$, (using item \ref{item:rem_symconleqbk} of Proposition \ref{prop:flow}
for the symmetry), and thus by Definition \ref{def:wk} and Lemma \ref{lem:tailcomparison}
the renormalized walk measure $\walk_\ell^\phi$ is stochastically dominated in absolute
value by the measure $\oud_{b-\frac{1}{2}}^{\phi}$, which has density
\begin{equation}
    g(\rho) \propto \Exp{- \frac{b - \frac{1}{2}}{2} \left( \rho - \frac{\phi}{b - \frac{1}{2}} \right)^2}
    \propto \Exp{- \frac{(\rho-\phi)^2}{2} - \frac{b - \frac{3}{2}}{2} \rho^2 }.
\end{equation}
If on the other hand $\ell < \left\lceil C \left( T + \log \frac{1}{\eps} \right) \right\rceil$,
we instead have stochastic domination in absolute value simply by $\cN_1^\phi$,
also by Lemma \ref{lem:tailcomparison}, since $\pot_\ell$ is convex and symmetric for all $\ell$ by item \ref{item:rem_symconleqbk}
of Proposition \ref{prop:flow}.

So we begin constructing our stochastic domination coupling by sampling
\begin{equation}
    \xi_0 \sim \cN_{(b-\frac{1}{2})/\left( (b-\frac{1}{2})^2 - 1 \right)}^0 ;
\end{equation}
since $\phi_{\tau^0} = \phi_\emptyset = 0$
we indeed have $|\xi_0| \geq |\phi_{\tau^0}|$.
Then inductively suppose that $|\xi_{k-1}| \geq |\phi_{\tau^{k-1}}|$.

First, if $k \leq n - \left\lceil C \left( T + \log \frac{1}{\eps} \right) \right\rceil$,
then $\xi_k \sim \oud_{b-\frac{1}{2}}^{\xi_{k-1}}$, and so by Lemma \ref{lem:normal_domination}
and the induction hypothesis, $|\xi_k|$ stochastically dominates $|\zeta_k|$,
where $\zeta_k \sim \oud_{b-\frac{1}{2}}^{\phi_{\tau^{k-1}}}$.
And by the above discussion, $|\zeta_k|$ stochastically dominates $|\phi_{\tau^k}|$,
so $|\xi_k|$ may be coupled with $|\phi_{\tau^k}|$ so that $|\xi_k| \geq |\phi_{\tau^k}|$.

On the other hand, if $k > n - \left\lceil C \left( T + \log \frac{1}{\eps} \right) \right\rceil$,
then $\xi_k \sim \cN_1^{\xi_{k-1}}$ and so again by Lemma \ref{lem:normal_domination}
and the induction hypothesis, $|\xi_k|$ stochastically dominates $|\zeta_k|$,
where $\zeta_k \sim \cN_1^{\phi_{\tau^{k-1}}}$, and again this $|\zeta_k|$ stochastically
dominates $|\phi_{\tau^k}|$, completing the inductive step.

So we end up with a coupling such that $|\xi_n| \geq |\phi_{\tau^n}| = |\phi_\tau|$.
\end{proof}

With this coupling in hand, we may now finish the proof of the Gaussian tail bound.

\begin{proof}[Proof of Proposition \ref{prop:tailbound}]
As mentioned in Remark \ref{rmk:epsTsquared}, we may 
shrink $\eps$ until we have $\eps \leq e^{-T}$ since this does not change the bound
we will derive (up to constants), and this allows the assumption $\eps T^2 \leq 1$
in Lemma \ref{lem:sd_coupling} to be satisfied, even if it was not satisfied
before this adjustment.

Since $|\xi_n|$ as defined above Lemma \ref{lem:sd_coupling} stochastically dominates
$|\phi_\tau|$ by that lemma, it suffices to prove a tail bound for $|\xi_n|$ instead.
But $\xi_n$ is a Gaussian and we can easily find its variance.
First, by Lemma \ref{lem:oud_stationary} and the definition of $\left( \xi_k \right)$,
we have
\begin{equation}
    \xi_{n - \left\lceil C \left( T + \log \frac{1}{\eps} \right) \right\rceil}
    \sim \cN_{(b-\frac{1}{2})/\left( (b-\frac{1}{2})^2 - 1 \right)}^0,
\end{equation}
and the remaining $\left\lceil C \left( T + \log \frac{1}{\eps} \right) \right\rceil$
increments are independent standard Gaussians.
Thus the variance of $\xi_n$ is exactly
\begin{equation}
    \frac{b-\frac{1}{2}}{\left( b - \frac{1}{2} \right)^2 - 1}
    + \left\lceil C \left( T + \log \frac{1}{\eps} \right) \right\rceil
    \leq C \left( T + \log \frac{1}{\eps} \right),
\end{equation}
after adjusting the constant, since $T \in \N$ so $T \geq 1$.
This finishes the proof.
\end{proof}
\subsection{Total variation distances between walk steps}
\label{sec:ou_steps}

In this section we prepare to prove Proposition \ref{prop:coupling} by deriving bounds on the total variation
distances between steps of the walks $\left( \phi_\tau \right)_{\tau \in \treebn}$ and
$\left( \psi_\tau \right)_{\tau \in \treebn}$ in the coupling of Definition \ref{def:joint}.
In the sequel, we will use $\dtv$ to denote the total variation distance between probability distributions.

First, in Lemma \ref{lem:steptv_bk} below, we compare $\walk_k^\phi$ not with $\oud_b^\phi$
but with a different Gaussian distribution, namely $\oud_{\bb_k}^\phi$,
under a smallness assumption on $\rem_k''$,
where $\rem_k$ is the remainder term of Proposition \ref{prop:flow}.
Later in Lemma \ref{lem:steptv_b} we will compare $\oud_{\bb_k}^\phi$ with  $\oud_b^\phi$,
obtaining a bound which is effective when $k$ is large.

\begin{lemma}
\label{lem:steptv_bk}
For all $\eps > 0$, $k \in \N$,
if we have $\sup_{\phi \in \R} \rem_k''(\phi) \leq \eps \leq \bb_k - 1$, then
for all $\phi \in \R$ we have
\begin{equation}
    \dtv \left( \walk_k^\phi, \oud_{\bb_k}^\phi \right) \leq e^{\eps(1+\phi^2)} - 1.
\end{equation}
\end{lemma}

Definition \ref{def:wk} only gives a formula for the distribution $\walk_k^\phi$ up to a normalizing constant,
so the first order of business is to control this constant in terms of $\rem_k''$.

\begin{lemma}
\label{lem:wk_pf}
Suppose that $\sup_{\phi \in \R} \rem_k''(\phi) \leq \eps \leq \bb_k - 1$.
Then for any $\phi \in \R$, we have
\begin{equation}
    \sqrt{\frac{2\pi}{\bb_k}} \leq
    \int_{-\infty}^\infty \Exp{-\frac{\bb_k}{2} \left( \rho - \frac{\phi}{\bb_k} \right)^2 + \rem_k(\rho) - \rem_k(0)} \,d\rho
    \leq \sqrt{\frac{2\pi}{\bb_k-\eps}} \Exp{\frac{\eps \phi^2}{2 \bb_k}}.
\end{equation}
\end{lemma}

\begin{proof}[Proof of Lemma \ref{lem:wk_pf}]
First, since $\rem_k$ is symmetric and $\rem_k''(\rho) \geq 0$ for all $\rho \in \R$ by item \ref{item:rem_symconleqbk}
of Proposition \ref{prop:flow}, we have $\rem_k(\rho) - \rem_k(0) \geq 0$ for all $\rho \in \R$, and so the integral in question is
at least
\begin{equation}
    \int_{-\infty}^\infty \Exp{-\frac{\bb_k}{2} \left( \rho - \frac{\phi}{\bb_k} \right)^2} \,d\rho = \sqrt{\frac{2\pi}{\bb_k}}.
\end{equation}
For the upper bound, again using symmetry and the assumed bound on $\rem_k''$, we have
$\rem_k(\rho) - \rem_k(0) \leq \eps \frac{\rho^2}{2}$ for all $\rho \in \R$.
Thus the integral is at most
\begin{align}
    \int_{-\infty}^\infty \Exp{-\frac{\bb_k}{2} \left( \rho - \frac{\phi}{\bb_k} \right)^2 + \eps \frac{\rho^2}{2}} \,d\rho
    &= \int_{-\infty}^\infty \Exp{-\frac{\bb_k - \eps}{2} \rho^2 + \phi \rho - \frac{\phi^2}{2 \bb_k} } \,d\rho \\
    &= \int_{-\infty}^\infty \Exp{-\frac{\bb_k - \eps}{2} \left( \rho - \frac{\phi}{\bb_k - \eps} \right)^2 + \frac{\phi^2}{2(\bb_k - \eps)} - \frac{\phi^2}{2\bb_k}} \,d\rho \\
    &= \sqrt{\frac{2\pi}{\bb_k-\eps}} \Exp{\frac{\eps \phi^2}{2 \bb_k (\bb_k-\eps)}},
    \label{eq:derivation2}
\end{align}
and the final upper bound is obtained from the fact that $\bb_k - \eps \geq 1$.
\end{proof}

\begin{proof}[Proof of Lemma \ref{lem:steptv_bk}]
Let $f_1$ and $f_2$ denote the densities of $\walk_k^\phi$ and $\oud_{\bb_k}^\phi = \cN_{\bb_k^{-1}}^{\bb_k^{-1} \phi}$ respectively.
Namely, 
\begin{equation}
    f_2(\rho) = \sqrt{\frac{\bb_k}{2\pi}} \Exp{- \frac{\bb_k}{2} \left( \rho - \frac{\phi}{\bb_k} \right)^2},
\end{equation}
and by Lemma \ref{lem:wk_pf} there is some constant $A = A_\phi > 0$ satisfying
$\sqrt{\frac{\bb_k-\eps}{\bb_k}} \Exp{-\frac{\eps \phi^2}{2\bb_k}} \leq A \leq 1$
such that
\begin{equation}
    f_1(\rho) = A \cdot \sqrt{\frac{\bb_k}{2\pi}} \Exp{- \frac{\bb_k}{2} \left( \rho - \frac{\phi}{\bb_k} \right)^2 + \rem_k(\rho) - \rem_k(0)}.
\end{equation}
Therefore we have
\begin{align}
    \dtv \left( \walk_k^\phi, \oud_{\bb_k}^\phi \right)
    &= \frac{1}{2} \int_{-\infty}^\infty \left| f_1(\rho) - f_2(\rho) \right| \,d\rho \\
    &= \frac{1}{2} \sqrt{\frac{\bb_k}{2\pi}} \int_{-\infty}^\infty e^{- \frac{\bb_k}{2} \left( \rho - \frac{\phi}{\bb_k} \right)^2}
    \left| 1 - A e^{\rem_k(\rho) - \rem_k(0)} \right| \,d\rho.
\end{align}
Now note that by upper bounding both factors (which are positive) we have
\begin{equation}
    A e^{\rem_k(\rho) - \rem_k(0)} \leq e^{\eps \frac{\rho^2}{2}},
\end{equation}
and lower bounding both factors, using $e^{-x} \geq 1 - x$ and $\sqrt{1 - x} \geq 1 - x$ for $x \in [0,1]$, we have
\begin{align}
    A e^{\rem_k(\rho) - \rem_k(0)}
    &\geq \sqrt{\frac{\bb_k - \eps}{\bb_k}} \Exp{- \frac{\eps \phi^2}{2\bb_k}} \\
    &\geq \left( 1 - \frac{\eps}{\bb_k} \right) \left( 1 - \frac{\eps \phi^2}{2\bb_k} \right)
    \geq 1 - \eps \frac{1 + \phi^2}{\bb_k}.
\end{align}
So, recalling that $\oud_{\bb_k}^\phi = \cN_{\bb_k^{-1}}^{\bb_k^{-1} \phi}$, we have
\begin{equation}
    \dtv \left( \walk_k^\phi, \oud_{\bb_k}^\phi \right)
    \leq \frac{1}{2} \cdot \oud_{\bb_k}^\phi \left[ \max\left\{ e^{\eps \frac{\rho^2}{2}} - 1, \eps \frac{1 + \phi^2}{\bb_k} \right\} \right],
\end{equation}
where the right-hand side denotes an expectation with respect to $\rho \sim \oud_{\bb_k}^\phi$.
Since $\max\{x,y\} \leq x + y$ for $x,y \geq 0$, we find
\begin{equation}
    \dtv \left( \walk_k^\phi, \oud_{\bb_k}^\phi \right)
    \leq \eps \frac{1 + \phi^2}{2\bb_k} + \frac{1}{2} \cdot \oud_{\bb_k}^\phi \left[ e^{\eps \rho^2} - 1 \right].
\label{eq:coup_inter1}
\end{equation}
Now the right-most term may be computed exactly:
\begin{align}
    \oud_{\bb_k}^\phi \left[ e^{\eps \frac{\rho^2}{2}} \right]
    &= \sqrt{\frac{\bb_k}{2\pi}} \int_{-\infty}^\infty \Exp{- \frac{\bb_k}{2} \left( \rho - \frac{\phi}{\bb_k} \right)^2 + \eps \frac{\rho^2}{2}} \,d\rho \\
    &= \sqrt{\frac{\bb_k}{\bb_k-\eps}} \Exp{\frac{\eps \phi^2}{2 \bb_k (\bb_k - \eps)}},
\end{align}
by the same calculation as in \eqref{eq:derivation2}.
Since again $\bb_k - \eps \geq 1$ and since $\sqrt{1 + x} \leq 1 + \frac{x}{2}$ for all $x \geq 0$, we find
\begin{equation}
    \oud_{\bb_k}^\phi \left[ e^{\eps \rho^2} \right]
    \leq \left( 1 + \frac{\eps}{2(\bb_k-\eps)} \right) \Exp{\frac{\eps \phi^2}{2 \bb_k}}
    \leq \left( 1 + \frac{\eps}{2} \right) \Exp{\frac{\eps \phi^2}{2 \bb_k}}
    \leq \Exp{\eps (1+\phi^2)},
\end{equation}
simply bounding $\bb_k > b \geq 1$.
Using the same loose bound on the first term of \eqref{eq:coup_inter1}
and using the fact that $x \leq e^x - 1$, we ultimately find that
\begin{equation}
\label{eq:coup_bound1}
    \dtv \left( \walk_k^\phi, \oud_{\bb_k}^\phi \right)
    \leq e^{\eps(1+\phi^2)} - 1,
\end{equation}
as desired.
\end{proof}

Now we turn to the comparison of the two Gaussian distributions $\oud_{\bb_k}^\phi$ and $\oud_b^\phi$, recalling this 
notation from Definition \ref{def:oud}.
Of course, there are a variety of standard bounds on the total variation distance between two Gaussians
\cite{B87Estimates,DMR18Total}, but we present the calculation here for completeness.

\begin{lemma}
\label{lem:steptv_b} 
For all $k \in \N$ and all $\phi \in \R$, we have
\begin{equation}
    \dtv \left( \oud_b^\phi, \oud_{\bb_k}^\phi \right)
    \leq \frac{1}{2} \sqrt{ \sqrt{\frac{\bb_k}{b}} \Exp{\frac{\bb_k - b}{2 b \bb_k} \phi^2} - 1 }.
\end{equation}
\end{lemma}

\begin{proof}[Proof of Lemma \ref{lem:steptv_b}]
We have 
\begin{align}
    2 \dtv \left( \oud_b^\phi, \oud_{\bb_k}^\phi \right)
    &= \int_{-\infty}^\infty \left|
        \sqrt{\frac{b}{2\pi}} e^{- \frac{b}{2} \left( \rho - \frac{\phi}{b} \right)^2}
        - \sqrt{\frac{\bb_k}{2\pi}} e^{- \frac{\bb_k}{2} \left( \rho - \frac{\phi}{\bb_k} \right)^2}
    \right| \,d\rho \\
    &= \sqrt{\frac{b}{2\pi}} \int_{-\infty}^\infty e^{-\frac{b}{2} \left( \rho - \frac{\phi}{b} \right)^2}
    \left|
        1 - \sqrt{\frac{\bb_k}{b}} e^{- \frac{\bb_k}{2} \left( \rho - \frac{\phi}{\bb_k} \right)^2 + \frac{b}{2} \left( \rho - \frac{\phi}{b} \right)^2}
    \right| \,d\rho \\
    &= \oud_b^\phi \left[ \left|
        1 - \sqrt{\frac{\bb_k}{b}} \Exp{- \frac{\bb_k}{2} \rho^2 + \rho \phi - \frac{\phi^2}{2 \bb_k}
        + \frac{b}{2} \rho^2 - \rho \phi + \frac{\phi^2}{2 b} }
    \right| \right] \\
    &= \oud_b^\phi \left[ \left|
        1 - \sqrt{\frac{\bb_k}{b}} \Exp{- \frac{\bb_k - b}{2} \rho^2 + \frac{\bb_k - b}{2 b \bb_k} \phi^2 }
    \right| \right],
\end{align}
where the latter denotes the expectation under $\rho \sim \oud_b^\phi$.
Now by Jensen's inequality we find that
\begin{align}
    4 \dtv \left( \oud_b^\phi, \oud_{\bb_k}^\phi \right)^2
    &\leq \oud_b^\phi
        \left[ \left( 1 - \sqrt{\frac{\bb_k}{b}} \Exp{- \frac{\bb_k-b}{2} \rho^2 + \frac{\bb_k-b}{2 b \bb_k} \phi^2} \right)^2 \right] \\
    &= 1 - 2 \sqrt{\frac{\bb_k}{b}} \Exp{\frac{\bb_k-b}{2 b\bb_k} \phi^2}
        \cdot \oud_b^\phi \left[ \Exp{- \frac{\bb_k-b}{2} \rho^2} \right] \\
    &\qquad \quad \; + \frac{\bb_k}{b} \Exp{\frac{\bb_k-b}{b \bb_k} \phi^2}
        \cdot \oud_b^\phi \left[ \Exp{- \frac{2\bb_k-2b}{2} \rho^2 } \right] \\
    &\leq 1 - \left( 2 \sqrt{\frac{\bb_k}{b}} e^{\frac{\bb_k-b}{2 b\bb_k} \phi^2}
        - \frac{\bb_k}{b} e^{\frac{\bb_k-b}{b \bb_k} \phi^2} \right)
        \cdot \oud_b^\phi \left[ \Exp{- \frac{\bb_k-b}{2} \rho^2} \right] \label{eq:normex}
\end{align}
and we may now calculate the expectation in \eqref{eq:normex} explicitly as follows:
\begin{align}
    \oud_b^\phi \left[ \Exp{- \frac{\bb_k - b}{2} \rho^2} \right]
    &= \sqrt{\frac{b}{2\pi}} \int_{-\infty}^\infty
        \Exp{- \frac{b}{2} \left( \rho - \frac{\phi}{b} \right)^2 - \frac{\bb_k - b}{2} \rho^2} \,d\rho \\
    &= \sqrt{\frac{b}{2\pi}} \int_{-\infty}^\infty
        \Exp{- \frac{b}{2} \rho^2 + \rho \phi - \frac{\phi^2}{2b} - \frac{\bb_k - b}{2} \rho^2} \,d\rho \\
    &= \sqrt{\frac{b}{2\pi}} \int_{-\infty}^\infty
        \Exp{- \frac{\bb_k}{2} \left( \rho - \frac{\phi}{\bb_k} \right)^2 + \frac{\phi^2}{2\bb_k} - \frac{\phi^2}{2b}} \,d\rho \\
    &= \sqrt{\frac{b}{\bb_k}} \Exp{- \frac{\bb_k - b}{2 b \bb_k} \phi^2}.
\end{align}
So we ultimately find that
\begin{align}
    4 \dtv \left( \oud_b^\phi, \oud_{\bb_k}^\phi \right)^2
    &\leq 1 - 2 + \sqrt{\frac{\bb_k}{b}} \Exp{\frac{\bb_k - b}{2 b \bb_k} \phi^2},
\end{align}
as required.
\end{proof}
\subsection{Coupling all of the initial steps}
\label{sec:ou_coupling}

Now we will use the total variation distance bounds derived above
to prove Proposition \ref{prop:coupling}.
Throughout this section we will use $\left( \phi_\tau \right)_{\tau \in \treebn}$
to denote the tree-indexed Markov chain of Definition \ref{def:treemc}, and
we will use $\left( \psi_\tau \right)_{\tau \in \treebn}$ to denote the branching 
OU walk which has been coupled to $\left( \phi_\tau \right)_{\tau \in \treebn}$
as in Definition \ref{def:joint}.
Recall that Proposition \ref{prop:coupling} states that if $\plat(\ham,T,\eps)$ holds
then
\begin{equation}
    \P \left[ \phi_{\tau'} = \psi_{\tau'} \text{ for all } \tau' \preceq \tau \text{ with } |\tau'| \leq n - k \right]
    \geq 1 - C e^{- c k + C \left( T + \log \frac{1}{\eps} \right)}
\end{equation}
for all $k \in [n]$ and $\tau \in [b]^n$.

To begin with, since the bounds of Lemmas \ref{lem:steptv_bk} and \ref{lem:steptv_b}
depend on the size of the current $\phi$, we will need to control this position
in a uniform way compatible with the bound of Proposition \ref{prop:shrinking}.
More precisely, since we will have a bound
on the factor in front of the $\phi$ dependence (which is written as $\eps$ in Lemma \ref{lem:steptv_bk}
and as $\frac{\bb_k - b}{2b\bb_k}$ in Lemma \ref{lem:steptv_b}) which is shrinking
exponentially fast as we move towards the root of the tree, the size of $\phi$
may be allowed to \emph{grow} at most exponentially quickly in this direction
without destroying the bound, and the probability of its growth exceeding this exponential
rate is extremely small.

\begin{lemma}
\label{lem:ou_bound}
Let $\left( \psi_\tau \right)_{\tau \in \treebn}$ be the branching OU walk as
in Definition \ref{def:joint}.
For any $\alpha > 1$ there are some $C, c > 0$ such that
\begin{equation}
    \P \left[ \left| \psi_{\tau^k} \right| \leq C \alpha^{n-k} \text{ for all } k \leq n-\ell \right]
    \geq 1 - C e^{-c \alpha^{2\ell}}
\end{equation}
for all $\ell \in \{0,\dotsc,n\}$, and $\tau \in [b]^n$.
\end{lemma}

For this lemma as well as some later applications in Section \ref{sec:ou_covariance}, it will be helpful to consider instead a stationary
OU walk, which stochastically dominates the walk started at $0$ in absolute value, as the following
lemma shows.
This will allow us to use reversibility in the present application, and a Bernstein-type tail bound of \cite{P15Concentration}
in the latter section (in the proof of Lemma \ref{lem:ou_coalescence} therein).

\begin{lemma}
\label{lem:stationary_domination}
Let $\left( \zeta_k \right)_{k=0}^n$ be an OU walk started at $0$, and let
$\left( \theta_k \right)_{k=0}^n$ be a stationary OU walk.
In other words, $\zeta_0 = 0$ and $\theta_0 \sim \cN_{b/(b^2-1)}^0$
(recalling Lemma \ref{lem:oud_stationary}), and inductively for all $k \in [n]$ we have
$\zeta_k \sim \oud_b^{\zeta_{k-1}}$ and $\theta_k \sim \oud_b^{\theta_{k-1}}$.
There is a Markovian coupling between these walks with $|\zeta_k| \leq |\theta_k|$ for all $k \in \{0,\dotsc,n\}$.
\end{lemma}

\begin{proof}[Proof of Lemma \ref{lem:stationary_domination}]
We proceed by induction.
The base case holds by definition, since $|\zeta_0| = 0$.
Now suppose that there is some coupling between $\zeta_k$
and $\theta_k$ such that $|\zeta_k| \leq |\theta_k|$ almost surely.
Then applying Lemma \ref{lem:normal_domination} gives a coupling
between $\zeta_{k+1}$ and $\theta_{k+1}$ which has the same property, since
\begin{equation}
    \zeta_{k+1} \stackrel{d}{=} \frac{\zeta_k}{b} + \frac{\xi_1}{\sqrt{b}}
    \qquad \text{and} \qquad
    \theta_{k+1} \stackrel{d}{=} \frac{\theta_k}{b} + \frac{\xi_2}{\sqrt{b}}
\end{equation}
for $\xi_1, \xi_2 \sim \cN_1^0$, where $\stackrel{d}{=}$ denotes equality in distribution.
\end{proof}

\begin{proof}[Proof of Lemma \ref{lem:ou_bound}]
Let $\left( \theta_k \right)_{k=0}^n$ be a stationary OU walk as in Lemma \ref{lem:stationary_domination}.
Since $\left( \psi_{\tau^k} \right)_{k=0}^n$ is an OU walk started at $0$, by that lemma it suffices to show
instead that there are constants $C, c > 0$ for which
\begin{equation}
\label{eq:phidom_goal}
    \P \left[ \left| \theta_k \right| \leq C \alpha^{n-k} \text{ for all } k \leq n - \ell \right]
    \geq 1 - C e^{- c \alpha^{2\ell}}.
\end{equation}
Let us introduce $C > 0$ whose value will be determined later; $c > 0$ will also be determined later,
depending on this choice of $C$.

Since $\left( \theta_k \right)_{k=0}^n$ is a reversible Markov chain by Lemma \ref{lem:oud_stationary}, we may construct it instead by starting
with $\theta_{n-\ell} \sim \cN_{b/(b^2-1)}^0$ and then sampling $\theta_{k-1} \sim \oud_b^{\theta_k}$ iteratively for $k=n-\ell,\dotsc,1$.
Now for $k \in [n-\ell]$ let us define the events
\begin{equation}
    A_k = \left\{ \left|\theta_k\right| \leq C \alpha^{n-k} \right\}.
\end{equation}
Then if the event in \eqref{eq:phidom_goal} fails there must be some largest $k \leq n - \ell$ for which $A_k$ fails, so
\begin{equation}
\label{eq:sumac}
    \P \left[ \left| \theta_k \right| > C \alpha^{n-k} \text{ for some } k \in \{0,1,\dotsc,n-\ell\} \right]
    \leq \P \left[ A_{n-\ell}^\c \right] + \sum_{k=1}^{n-\ell} \P \left[ A_{k-1}^\c \middle| A_k \right],
\end{equation}
where $A_k^\c$ denotes the complement of $A_k$.
First we have
\begin{align}
    \P \left[ A_{n-\ell}^\c \right] &= \P \left[ \left| \theta_{n-\ell} \right| > C \alpha^\ell \right] \\
    &= \cN_1^0 \left[ |\zeta| > \sqrt{\frac{b^2-1}{b}} C \alpha^\ell \right],
\end{align}
where the latter denotes the probability under $\zeta \sim \cN_1^0$.
By a standard Gaussian tail bound \cite[Theorem 1.2.6]{D19Probability}, this is at most
\begin{equation}
    D \Exp{- c \frac{b^2-1}{b} C^2 \alpha^{2\ell}}
    = D \Exp{- c \alpha^{2\ell}},
\end{equation}
for some $D > 0$ and $c = c_{\alpha,b,C} > 0$ which changes from one side of the equation to the next;
here we use the fact that $\alpha > 1$ to eliminate the extraneous polynomial factors in the tail bound.
Next, we bound $\P \left[ A_{k-1}^\c \middle| A_k \right]$ for $k \in [n]$.
Under the assumption of $A_k$ we have $|\theta_k| \leq C \alpha^{n-k}$.
So, if $\zeta \sim \cN_1^0$ then by another application of the Gaussian tail bound we have
\begin{align}
    \P \left[ A_{k-1}^\c \middle| A_k \right]
    &= \cN_1^0 \left[ \left| \frac{\theta_k}{b} + \frac{\zeta}{\sqrt{b}} \right| > C \alpha^{n-(k-1)} \right] \\
    &\leq \cN_1^0 \left[ \left| \zeta \right| > \sqrt{b} C \alpha^{n-(k-1)} \left( 1 - \frac{1}{\alpha b} \right) \right] \\
    &\leq D \Exp{- c \alpha^{2(n-(k-1))}}
\end{align}
for some $c = c_{\alpha,b,C} > 0$ (using again the fact that $\alpha > 1$) which may not be the same constant as above,
but we may just take the minimum of the two to see that
\begin{equation}
    \P \left[ \left| \theta_k \right| > C \alpha^{n-k} \text{ for some } k \in \{0,1,\dotsc,n-\ell\} \right]
    \leq D \sum_{k=1}^{n-\ell+1} \Exp{- c \alpha^{2 (n-(k-1))}}.
\end{equation}
Finally, since the sequence being summed on the right-hand side decays faster than exponentially as $k$ decreases,
the sum is dominated by the last term (which is the largest).
In other words, there is some $D = D_\alpha > 0$ and $c = c_{\alpha,b,C}$ such that
\begin{equation}
    \P \left[ \left| \theta_k \right| > C \alpha^{n-k} \text{ for some } k \in \{0,1,\dotsc,n-\ell\} \right]
    \leq D \Exp{- c \alpha^{2\ell}}.
\end{equation}
This finishes the proof after adjusting the values of $C$ and $D$ upwards to be the same.
\end{proof}

With this bound in place, we may now apply the total variation bounds on the steps
in Lemmas \ref{lem:steptv_bk} and \ref{lem:steptv_b} to show that the walks $\left( \phi_{\tau^k} \right)_{k=0}^n$
and $\left( \psi_{\tau^k} \right)_{k=0}^n$ may be coupled to stay equal until near the end.

\begin{proof}[Proof of Proposition \ref{prop:coupling}]
As in the proof of Proposition \ref{prop:tailbound}, following Remark \ref{rmk:epsTsquared},
we may shrink the given value of $\eps$ until we have $\eps \leq e^{-T}$ so that
$\eps T^2 \leq 1$.
This allows Proposition \ref{prop:shrinking} to be applied without changing the bound
we will derive presently (up to constants).

Let us also fix some $\alpha > 1$ to be determined later, so that Lemma \ref{lem:ou_bound}
implies that for all $\ell$ we have
\begin{equation}
    \P \left[ \left| \psi_{\tau^k} \right| \leq C \alpha^{n-k}
    \text{ for all } k \leq n-\ell \right] \geq 1 - C \Exp{-c \alpha^{2\ell}}.
\end{equation}
Let us use $B_\ell$ to denote the event inside the probability on the left-hand side above,
and $B_\ell^\c$ to denote its complement.
Now recall the coupling of Definition \ref{def:joint}.
By the construction, our goal is to show
\begin{equation}
    \P \left[ \chi_{\tau^k} = 1 \text{ for all } k \leq n - \ell \right]
    \geq 1 - e^{- c \ell + C \left( T + \log \frac{1}{\eps} \right) }.
\end{equation}
If above event fails then there is some first $k \in [n-\ell]$ for which $\chi_{\tau^k} = 0$
(note that $\chi_{\tau^0} = 1$ always).
So we have
\begin{equation}
\label{eq:coupling_sum}
    \P \left[ \chi_{\tau^k} = 0 \text{ for some } k \leq n - \ell \right]
    \leq \P \left[ B_\ell^\c \right] +
    \sum_{k=1}^{n-\ell} \P \left[ \chi_{\tau^k} = 0 \,\middle|\, \chi_{\tau^{k-1}} = 1 \text{ and } B_\ell \text{ holds} \right].
\end{equation}
Now we will bound the probabilities in the sum on the right-hand side.

If $\chi_{\tau^{k-1}} = 1$ then $\phi_{\tau^{k-1}} = \psi_{\tau^{k-1}}$,
and under the assumption of $B_\ell$ these have absolute value at most $C \alpha^{n-k}$
(adjusting the constant $C$ which depends on $\alpha$ to replace $k-1$ by $k$).
Thus, since we have assumed $\plat(\ham,T,\eps)$,
by Proposition \ref{prop:shrinking} and Lemmas \ref{lem:steptv_bk} and \ref{lem:steptv_b}, since the steps in Definition \ref{def:joint}
are coupled via the coupling which achieves the total variation distance, after adjusting the constants we have
\begin{align}
    \P \left[ \chi_{\tau^k} = 0 \,\middle|\, B_\ell \text{ and } \chi_{\tau^{k-1}} = 1 \right]
    &\leq \dtv \left( \walk_{n-k}^{\psi_{\tau^{k-1}}}, \oud_{\bb_{n-k}}^{\psi_{\tau_{k-1}}} \right)
    + \dtv \left( \oud_{\bb_{n-k}}^{\psi_{\tau^{k-1}}}, \oud_b^{\psi_{\tau_{k-1}}} \right) \\
    &\leq \Exp{\frac{C \alpha^{2\left(n-k\right)}}{b^{n-k} \delta^T \eps}} - 1
    + \frac{1}{2} \sqrt{ \sqrt{\frac{\bb_{n-k}}{b}} \Exp{\frac{\bb_{n-k} - b}{2 b \bb_{n-k}} C \alpha^{2 \left( n-k \right)} } - 1}.
\end{align}
Now note that
\begin{equation}
    \bb_{n-k} - b = \frac{b^{n-k+1} - 1}{b^{n-k} - 1} - \frac{b^{n-k + 1} - b}{b^{n-k} - 1}
    = \frac{b - 1}{b^{n-k} - 1} \leq \frac{C}{b^{n-k}}
\end{equation}
and
\begin{equation}
    \frac{\bb_{n-k}}{b} = \frac{b^{n-k+1} - 1}{b^{n-k+1} - b} \leq 1 + \frac{b - 1}{b^{n-k + 1} - b} \leq \Exp{\frac{C}{b^{n-k}}}
\end{equation}
for some $C = C_b > 0$ as long as $n - k \geq 1$.
So after adjusting the constants again, we have
\begin{equation}
    \P \left[ \chi_{\tau^k} = 0 \,\middle|\, B_\ell \text{ and } \chi_{\tau^{k-1}} = 1 \right]
    \leq \Exp{C \frac{ \left( \alpha^2/b \right)^{n-k} }{\delta^T \eps} } - 1 + \frac{1}{2} \sqrt{ \Exp{C \left( \alpha^2 / b \right)^{n-k}} - 1 }.
\end{equation}
Now let us set $\alpha = b^{1/4}$, although any number in the open interval $(1, \sqrt{b})$ would suffice.
Then as long as $n - k \geq C \left(T + \log \frac{1}{\eps} \right)$ for some $C$, the expressions inside of both exponentials are
at most $1$, so we may use the inequality $e^x - 1 \leq 2 x$ for $0 \leq x \leq 1$, and the fact that $\delta^T \eps \leq 1$, to obtain
\begin{equation}
    \P \left[ \chi_{\tau^k} = 0 \,\middle|\, B_\ell \text{ and } \chi_{\tau^{k-1}} = 1 \right]
    \leq C \frac{\left(\alpha^2 / b \right)^{n-k}}{\delta^T \eps}
    \leq C e^{- c (n-k) + C \left( T + \log \frac{1}{\eps} \right)}
\end{equation}
for some constants $C, c > 0$ depending on $b$.

Now plugging this bound into the sum in \eqref{eq:coupling_sum} under the assumption that $\ell \geq C \left( T + \log \frac{1}{\eps} \right)$,
and applying the bound of Lemma \ref{lem:ou_bound} with $\alpha = b^{1/4}$, we find that
\begin{align}
    \P \left[ \chi_{\tau^k} = 0 \text{ for some } k \leq n - \ell \right]
    &\leq C e^{- c b^{\ell/2}} + C e^{C \left( T + \log \frac{1}{\eps} \right)} \sum_{k=0}^{n-\ell} e^{- c (n-k)} \\
    &\leq C e^{- c b^{\ell/2}} + C e^{C \left( T + \log \frac{1}{\eps} \right)} e^{- c \ell},
\end{align}
after adjusting the constants since the sum is geometric and thus dominated by its largest term, which is the $k=n-\ell$ term.
Since $b^{\ell/2} \geq \ell/2$ for all $\ell \geq 1$, and since the case of $\ell \leq C \left( T + \log \frac{1}{\eps} \right)$
holds trivially by adjusting the constants, this finishes the proof.
\end{proof}
\subsection{Covariance bound}
\label{sec:ou_covariance}

In this section we finally prove Theorem \ref{thm:main_general},
giving a bound on the covariance under $\mubnh$
between $\phi_{\tau_1}$ and $\phi_{\tau_2}$ for $\tau_1, \tau_2 \in [b]^n$.
To do this, we will compare $(\phi_{\tau_1}, \phi_{\tau_2})$ with a pair of
independent variables $(\xi_n^1, \xi_n^2)$ with the same marginals.
More precisely, we construct two independent time-inhomogeneous Markov chains $\left( \xi_k^1 \right)_{k=0}^n$
and $\left( \xi_k^2 \right)_{k=0}^n$ by setting $\xi_0^1 = \xi_0^2 = 0$
and then recursively sampling
\begin{equation}
    \xi_k^1 \sim \walk_{n-k}^{\xi_{k-1}^1}
    \qquad \text{and} \qquad
    \xi_k^2 \sim \walk_{n-k}^{\xi_{k-1}^2}
\end{equation}
independently for each $k = 1, \dotsc, n$ in order.
For the following statement,
we remind the reader that the distance $\d(\tau_1,\tau_2)$ was defined just above the statement of Theorem \ref{thm:main_brw}.

\begin{lemma}
\label{lem:coupling_independent}
There are some constants $C, c > 0$ depending on $b$ such that
if $\plat(\ham,T,\eps)$ holds then there is a coupling between
$\left( \phi_{\tau_1}, \phi_{\tau_2} \right)$ and $\left( \xi_n^1, \xi_n^2 \right)$ so that
\begin{equation}
    \P \left[ \phi_{\tau_1} = \xi_n^1 \text{ and } \phi_{\tau_2} = \xi_n^2 \right]
    \geq 1 - C e^{-c \d(\tau_1,\tau_2) + C (T + \log \frac{1}{\eps})}.
\end{equation}
\end{lemma}

To prove this we will apply Proposition \ref{prop:coupling} both to couple
$\left( \phi_\tau \right)_{\tau \in \treebn}$ with a branching OU walk
$\left( \psi_\tau \right)_{\tau \in \treebn}$, and additionally to couple the independent linear walks
$\left( \left( \xi_k^1 \right)_{k=0}^n, \left( \xi_k^2 \right)_{k=0}^n \right)$
with a pair of independent linear OU walks
$\left( \left( \zeta_k^1 \right)_{k=0}^n, \left( \zeta_k^2 \right)_{k=0}^n \right)$.
Then the following lemma shows that the OU versions of the branching
and linear walks may be coupled such that the two branches
in $\left( \psi_\tau \right)_{\tau \in \treebn}$ leading to $\tau_1$ and $\tau_2$ may be coupled
to coalesce with $\left( \zeta_k^1 \right)_{k=0}^n$ and $\left( \zeta_k^2 \right)_{k=0}^n$
shortly after they split.

\begin{lemma}
\label{lem:ou_coalescence}
Let $\left( \psi_\tau \right)_{\tau \in \treebn}$ be a branching OU walk
as in Definition \ref{def:joint}, and let $\left( \left( \zeta_k^1 \right)_{k=0}^n, \left( \zeta_k^2 \right)_{k=0}^n \right)$
be two independent linear OU walks started at $\zeta_0^1 = \zeta_0^2 = 0$ and
taking steps $\zeta_k^1 \sim \oud_b^{\zeta_{k-1}^1}$ and $\zeta_k^2 \sim \oud_b^{\zeta_{k-1}^2}$.
Then there are some constants $C, c > 0$ depending only on $b$,
and a Markovian coupling between $\left( \psi_\tau \right)_{\tau \in \treebn}$ and
$\left( \left( \zeta_k^1 \right)_{k=0}^n, \left( \zeta_k^2 \right)_{k=0}^n \right)$,
such that for all $m > 0$ we have
\begin{equation}
    \P \left[ \psi_{\tau_1^k} = \zeta_k^1 \text{ and } \psi_{\tau_2^k} = \zeta_k^2 \text{ for all } k \geq n - \d(\tau_1, \tau_2) + m \right]
    \geq 1 - C e^{- c m}.
\end{equation}
\end{lemma}

\begin{proof}[Proof of Lemma \ref{lem:ou_coalescence}]
First of all we may simply set $\psi_{\tau_1^k} = \zeta_k^1$ for all $k$.
Then $\psi_{\tau_2^k}$ splits off from $\psi_{\tau_1^k}$ only after $k > n - \d(\tau_1,\tau_2)$,
at which point it does an independent OU walk.
For each $k > n - \d(\tau_1, \tau_2)$, we will couple $\psi_{\tau_2^k} \sim \oud_b^{\psi_{\tau_2^{k-1}}}$
and $\zeta_k^2 \sim \oud_b^{\zeta_{k-1}^2}$ under the optimal coupling which achieves the total variation distance
between these two distributions, and it suffices
to show that under this coupling we have
\begin{equation}
    \P \left[ \psi_{\tau_2^k} = \zeta_k^2 \text{ for some } k \in \left\{ n- \d(\tau_1, \tau_2) + 1, n - \d(\tau_1, \tau_2) + m \right\} \right]
    \geq 1 - C e^{- c m}.
\end{equation}
Now let us bound the total variation distance between $\oud_b^\psi$ and $\oud_b^\zeta$ for any $\psi, \zeta \in \R$.
Consulting Figure \ref{fig:gaussians}, if $\psi \leq \zeta$ we see that, with $\rho \sim \cN_1^0$,
this total variation distance is at most one minus
\begin{equation}
    2 \cdot \cN_1^0 \left[ \frac{\psi}{b} + \frac{\rho}{\sqrt{b}} > \frac{\psi + \zeta}{2 b} \right]
    = 2 \cdot \cN_1^0 \left[ \rho > \frac{\zeta - \psi}{2 \sqrt{b}} \right],
\end{equation}
and similarly for $\zeta \leq \psi$.
Now a standard Gaussian tail lower bound \cite[Theorem 1.2.6]{D19Probability} and the fact that the total variation distance
is increasing in $|\psi - \zeta|$ shows that if $|\psi - \zeta| \leq 2 \alpha \sqrt{b}$ then we have
\begin{equation}
    \dtv \left( \oud_b^\psi, \oud_b^\zeta \right) \leq 1 - 
    2 \left( \alpha^{-1} - \alpha^{-3} \right) e^{- \frac{\alpha^2}{2}}.
\end{equation}
Thus for all $\alpha > 1$ there is some $\eta > 0$ such that if $|\psi - \zeta| \leq 2 \alpha \sqrt{b}$ then $\dtv(\oud_b^\psi, \oud_b^\zeta) \leq 1 - \eta$.
So it suffices to show that under the coupling described above, there are some $\alpha > 1$ and $\gamma, C, c > 0$ for which
\begin{equation}
\label{eq:bernstein_goal}
    \P \left[ | \psi_{\tau_2^k} |, | \zeta_k^2 | \leq \alpha \sqrt{b} \text{ for at least } \gamma m \text{ integers } k \in \{ n - \d(\tau_1, \tau_2) + 1, n - \d(\tau_1, \tau_2) + m \} \right]
    \geq 1 - C e^{-c m}.
\end{equation}
Indeed, if this is the case then with probability $1 - C e^{-c m}$ there will be at least $\gamma m$ independent trials to couple which each have success
probability at least $1 - \eta$, and the probability of all of them failing is at most $C e^{ - c \gamma m}$ for some $C, c > 0$, which would finish the proof.

To show \eqref{eq:bernstein_goal}, we will first use Lemma \ref{lem:stationary_domination} to see that both $\left( \psi_{\tau_2^k} \right)_{k=0}^n$ and $\left( \zeta_k^2 \right)_{k=0}^n$
are stochastically dominated in absolute value by \emph{stationary} OU walks $\left( \theta_k^1 \right)_{k=0}^n$ and $\left( \theta_k^2 \right)_{k=0}^n$
(which are not independent since they are coupled through $\left( \psi_{\tau_2^k} \right)_{k=0}^n$ and $\left( \zeta_k^2 \right)_{k=0}^n$).
Then we will apply the following Bernstein-type inequality of \cite[Theorem 3.3]{P15Concentration}, specifically equation (3.21)
of that result, which shows that
\begin{equation}
    \P \left[ \left| \frac{1}{m} \sum_{k=n-\d(\tau_1,\tau_2)+1}^{n-\d(\tau_1,\tau_2)+m} \ind{|\theta_k^1| \leq \alpha \sqrt{b}}
    - \P \left[ |\theta_0^1| \leq \alpha \sqrt{b} \right] \right| > \lambda \right]
    \leq 2 \Exp{ \frac{- m^2 \lambda^2 \cdot G_b}{4 m + 10 \lambda} },
\end{equation}
where $G_b$ is the spectral gap of the OU walk we are considering.
This is positive by \cite[Proposition 4.1.1]{BGL14Analysis}, so taking $\alpha$ large enough that
\begin{equation}
    \P \left[ |\theta_0^1| \leq \alpha \sqrt{b} \right] > \frac{3}{4}
\end{equation}
and then taking $\lambda \leq \frac{1}{8}$ we find that with probability at least $1 - C e^{- c m}$,
at least $\frac{5}{8} m$ of the integers $k \in \left\{ n - \d(\tau_1,\tau_2) + 1, n - \d(\tau_1,\tau_2) + m \right\}$
satisfy $|\theta_k^1| \leq \alpha\sqrt{b}$.
The same is true for $|\theta_k^2|$, so by the fact that $|\theta_k^1| \geq |\psi_{\tau_2^k}|$ and $|\theta_k^2| \geq |\zeta_k^2|$,
we obtain \eqref{eq:bernstein_goal} with $\gamma = \frac{1}{8}$ for $m$ large enough.
\end{proof}

Now we combine Proposition \ref{prop:coupling} and Lemma \ref{lem:ou_coalescence} to prove
a similar coalescence statement for the time-inhomogeneous walks $\left( \phi_\tau \right)_{\tau \in \treebn}$
and $\left( \left( \xi_k^1 \right)_{k=0}^n, \left( \xi_k^2 \right)_{k=0}^n \right)$, which take steps
according to the renormalized walk measures.

\begin{proof}[Proof of Lemma \ref{lem:coupling_independent}]
First, let $\left( \psi_\tau \right)_{\tau \in \treebn}$ be a branching OU
process starting at $0$ which is coupled to $\left( \phi_\tau \right)_{\tau \in \treebn}$ as in Definition \ref{def:joint}.
Also let $\left( \zeta_k^1 \right)_{k=0}^n$ and $\left( \zeta_k^2 \right)_{k=0}^n$
be independent linear OU walks started at $0$ which are coupled to $\left( \xi_k^1 \right)_{k=0}^n$
and $\left( \xi_k^2 \right)_{k=0}^n$ in the same way, by taking the coupling which maximizes the probability of remaining equal
at each step until the walks separate.
Then Proposition \ref{prop:coupling} (which also applies to the linear walks as they may each be seen as one branch in a tree)
states that for all $\ell \in [n]$ we have
\begin{equation}
\label{eq:coupling1}
    \P \left[
        \phi_{\tau_1^k} = \psi_{\tau_1^k},
        \phi_{\tau_2^k} = \psi_{\tau_2^k},
        \xi_k^1 = \zeta_k^1,
        \text{ and } \xi_k^2 = \zeta_k^2
        \text{ for all } k \leq n - \ell
    \right]
    \geq 1 - C e^{ - c \ell + C \left( T + \log \frac{1}{\eps} \right)}.
\end{equation}
Up to this point, we have not specified the coupling between
$\left( ( \phi_\tau ), ( \psi_\tau ) \right)$ and
$\left( ( \xi_k^1 ), ( \xi_k^2  ), ( \zeta_k^1 ), ( \zeta_k^2 ) \right)$;
the above is simply a union bound across multiple applications of Proposition \ref{prop:coupling}.
Now, however, we will use the Markovian coupling between $\left( \psi_\tau \right)$
and $\left( \left( \zeta_k^1 \right), \left( \zeta_k^2 \right) \right)$ given by Lemma \ref{lem:ou_coalescence},
which satisfies
\begin{equation}
\label{eq:coupling2}
    \P \left[ \psi_{\tau_1^k} = \zeta_k^1 \text{ and } \psi_{\tau_2^k} = \zeta_k^2 \text{ for all } k \geq n - \d(\tau_1, \tau_2) + m \right]
    \geq 1 - C e^{-cm}.
\end{equation}
Now if both events in \eqref{eq:coupling1} and \eqref{eq:coupling2} hold, and we have $n - \d(\tau_1, \tau_2) + m \leq n - \ell$, then
there is some $k$ for which
\begin{equation}
    \phi_{\tau_1^k} = \psi_{\tau_1^k} = \zeta_k^1 = \xi_k^1
    \qquad \text{and} \qquad 
    \phi_{\tau_2^k} = \psi_{\tau_2^k} = \zeta_k^2 = \xi_k^2.
\end{equation}
So, since all of the couplings are Markovian and the subsequent step distributions of $\phi_{\tau_1^k}$ and $\xi_k^1$
as well as $\phi_{\tau_2^k}$ and $\xi_k^2$ are the same, we may switch to coupling these to be equal once they coalesce,
meaning we will end up with $\phi_{\tau_1} = \xi_n^1$ and $\phi_{\tau_2} = \xi_n^2$.

Choosing $m = \ell = \left\lfloor \frac{\d(\tau_1, \tau_2)}{2} \right\rfloor$, we do indeed have
$n - \d(\tau_1,\tau_2) + m \leq n - \ell$, and the probability of the two events \eqref{eq:coupling1} and \eqref{eq:coupling2}
occurring is at least
\begin{equation}
    1 - C e^{-c \left\lfloor \frac{\d(\tau_1, \tau_2)}{2} \right\rfloor + C \left( T + \log \frac{1}{\eps} \right) }
    - C e^{- c \left\lfloor \frac{\d(\tau_1,\tau_2)}{2} \right\rfloor}
    \geq 1 - C e^{- c \d(\tau_1,\tau_2) + C \left( T + \log \frac{1}{\eps} \right)}
\end{equation}
after adjusting the constants.
This finishes the proof.
\end{proof}

Finally we combine the coupling of Lemma \ref{lem:coupling_independent} with the tail bound of Proposition \ref{prop:tailbound}
to conclude the covariance bound in our generalized main theorem.

\begin{proof}[Proof of Theorem \ref{thm:main_general}]
Let $A$ denote the event that $\phi_{\tau_1} = \xi_n^1$ and $\phi_{\tau_2} = \xi_n^2$
under the coupling in Lemma \ref{lem:coupling_independent}, and $A^\c$ denote its complement.
Then we have
\begin{align}
    \mubnh \left[ \phi_{\tau_1} \phi_{\tau_2} \right]
    &= \E \left[ \phi_{\tau_1} \phi_{\tau_2} \left( \ind*{A} + \ind*{A^\c} \right) \right] \\
    &= \E \left[ \xi_n^1 \xi_n^2 \ind*{A} \right] +  \E \left[ \phi_{\tau_1} \phi_{\tau_2} \ind*{A^\c} \right] \\
    &= \E \left[ \xi_n^1 \xi_n^2 \right] - \E \left[ \xi_n^1 \xi_n^2 \ind*{A^\c} \right] + \E \left[ \phi_{\tau_1} \phi_{\tau_2} \ind*{A^\c} \right] \\
    &= \E \left[ \left( \phi_{\tau_1} \phi_{\tau_2} - \xi_n^1 \xi_n^2 \right) \ind*{A^\c} \right],
\end{align}
using the independence and symmetry of $\xi_n^1$ and $\xi_n^2$.
Now applying Cauchy--Schwarz we find that
\begin{align}
    \mubnh \left[ \phi_{\tau_1} \phi_{\tau_2} \right]^2
    &\leq \E \left[ \left( \phi_{\tau_1} \phi_{\tau_2} - \xi_n^1 \xi_n^2 \right)^2 \right] \cdot \P \left[ A^\c \right] \\
    &\leq 2 \left( \E \left[ \phi_{\tau_1}^2 \phi_{\tau_2}^2 \right] + \E \left[ \phi_{\tau_1}^2 \right]^2 \right)
    \cdot C e^{- c \d(\tau_1, \tau_2) + C \left( T + \log \frac{1}{\eps} \right)},
\end{align}
using the inequality $(x+y)^2 \leq 2 (x^2 + y^2)$ as well as independence of $\xi_n^1$ and $\xi_n^2$ and the fact that
$\phi_{\tau_1}, \phi_{\tau_2}, \xi_n^1$, and $\xi_n^2$ all have the same marginal distributions.
Another application of Cauchy--Schwarz yields 
\begin{equation}
    \mubnh \left[ \phi_{\tau_1} \phi_{\tau_2} \right]^2
    \leq 4 C \mubnh \left[ \phi_{\tau_1}^4 \right] \cdot e^{- c \d(\tau_1, \tau_2) + C \left( T + \log \frac{1}{\eps} \right)}.
\end{equation}
Finally we apply the Gaussian tail bound of Proposition \ref{prop:tailbound} which implies that
\begin{equation}
    \mubnh \left[ \phi_{\tau_1}^4 \right] \leq D \left( T + \log \frac{1}{\eps} \right)^2
\end{equation}
for some constant $D = D_b > 0$,
which finishes the proof by putting the above inside the exponent, using the fact that $\log x \leq x$ for $x \geq 0$,
and then adjusting the constants.
\end{proof}
\section{Analysis of initial potentials}
\label{sec:init}

Now we will apply Theorem \ref{thm:main_general} to derive
Theorems \ref{thm:main_brw}, \ref{thm:main_sinhgordon}, and \ref{thm:main_phialpha}.
In all cases, we will just need to verify the condition $\plat(\ham,T,\eps)$
for appropriate choices of $T \in \N$ and $\eps > 0$.
This entails verifying that the potential $\ham$ is confining in the sense of Definition \ref{def:confining},
and showing that $\rem_1''(T) \leq b - \eps$, where $\rem_1$ is the first remainder term
defined in \eqref{eq:r1} in Proposition \ref{prop:flow}.
For this latter point, we will use the fact that
\begin{equation}
\label{eq:remder_walkot}
    \rem_1''(T) = b \cdot \walk_0^T \left[ \left( \rho - \walk_0^T[\rho] \right)^2 \right],
\end{equation}
which is the $k=0$ case of \eqref{eq:2dr_cum}.
Recall also from Definition \ref{def:wk} that $\walk_0^T$ has Radon-Nikodym derivative
$e^{-\ham(\rho)}$ with respect to $\cN_1^T$, a normal distribution with mean $T$ and variance $1$.

\subsection{Small ball events}
\label{sec:init_smallball}

We first consider Theorem \ref{thm:main_brw} which concerns the radius-$r$ small ball event,
i.e.\ the potential
\begin{equation}
    \hamr(\phi) = \begin{cases}
        0 & \text{if } |\phi| \leq r, \\
        \infty & \text{otherwise}
    \end{cases}
\end{equation}
for any $r > 0$.
We begin by showing that this potential is confining; note that in future sections,
this will be much easier, as the only difficulty arises with the fact that here we 
may have $\hamr(\phi) = \infty$.

\begin{lemma}
\label{lem:hamr_confining}
The potential $\hamr$ is confining in the sense of Definition \ref{def:confining}.
\end{lemma}

\begin{proof}[Proof of Lemma \ref{lem:hamr_confining}]
The symmetry and convexity hypotheses are immediate; it remains to show that for all $y \geq 0$, the set
\begin{equation}
    J_y = \left\{ x \geq 0 : \hamr(y + x) < \infty \text{ or } \hamr(y - x) < \infty \right\}
\end{equation}
is an interval, and that the function $x \mapsto \hamr(y + x) - \hamr(y - x)$ is convex on $J_y$.
For the first point, note that if $r \geq y \geq 0$ then we have
$J_y = [0, r-y] \cup [0,r+y] = [0,r+y]$, and if $r < y$ then we have
$J_y = [-r+y, r+y]$; in both cases, $J_y$ is an interval.
Now if $r \geq y \geq 0$, we have
\begin{equation}
    \hamr(y+x) - \hamr(y-x) =
    \begin{cases}
        0 &\text{if } 0 \leq x \leq r-y, \\
        \infty &\text{if } r-y < x \leq r+y,
    \end{cases}
\end{equation}
which is convex on $[0,r+y]$.
On the other hand, if $r < y$, then $\hamr(y+x) - \hamr(y-x) = \infty$ for all $x \in [-r+y,r+y]$,
which is also convex on this interval.
\end{proof}

Now we prove the plateau condition for $\hamr$.
Recalling \eqref{eq:remder_walkot}, this plateau arises from the fact that the variance
of a Gaussian with mean $T$ conditioned to lie in $[-r,r]$ will steeply drop off once $T \gtrsim r$,
since the Gaussian tail is lighter than exponential.

\begin{lemma}
\label{lem:hamr_plateau}
There is some constant $D = D_b > 0$ such that for any $r > 0$ we have $\plat\left( \hamr, \left\lceil 2r + D \right\rceil, \frac{b}{2} \right)$.
\end{lemma}

\begin{proof}[Proof of Lemma \ref{lem:hamr_plateau}]
By Lemma \ref{lem:hamr_confining}, it suffices
to bound the second derivative of the first remainder term, recalling \eqref{eq:remder_walkot}.
We claim that there are some absolute constants $C, c > 0$ such that for all $T \geq 2r+2$ we have
\begin{equation}
\label{eq:remder_smallball}
    \walk_0^T \left[ (\rho - \walk_0^T[\rho])^2 \right]
    \leq \frac{1}{4} + C e^{- c (T-r)}.
\end{equation}
This would finish the proof as we may set $D$ large enough so that for $T = \lceil 2r + D \rceil$
the right-hand side above is at most $\frac{1}{2}$, meaning $\rem_1''(T) \leq \frac{b}{2} = b - \frac{b}{2}$.

To prove \eqref{eq:remder_smallball},
let us fix some $a \in [0,1]$ to be determined later.
Letting $\zeta \sim \cN_1^T$, standard upper and lower bounds on Gaussian tails \cite[Theorem 1.2.6]{D19Probability} yield
\begin{equation}
    \frac{\cN_1^T \left[ \zeta \leq r - a \right]}
    {\cN_1^T \left[ \zeta \leq r \right]}
    \leq \frac{\frac{1}{T-r-a} e^{-\frac{(T-r-a)^2}{2}}}
    {\left( \frac{1}{T-r} - \frac{1}{(T-r)^3} \right) e^{- \frac{(T-r)^2}{2}}}
    \leq C e^{- a (T-r) }
\end{equation}
for some absolute constant $C$, using $a \leq 1$ and $T - r \geq r + 2$.
Now since $\walk_0^T[\rho]$ minimizes $x \mapsto \walk_0^T[(\rho-x)^2]$,
\begin{equation}
    \walk_0^T \left[ (\rho - \walk_0^T[\rho])^2 \right]
    \leq \walk_0^T \left[ (\rho - r)^2 \right]
    \leq a^2 + 4r^2 \walk_0^T\left[\rho \leq r - a \right]
    \leq a^2 + 4 C r^2 e^{-a (T-r)}.
\end{equation}
Setting $a = \frac{1}{2}$ and adjusting the constants (using $T - r \geq r + 2$)
finishes the proof of \eqref{eq:remder_smallball}.
\end{proof}

\begin{proof}[Proof of Theorem \ref{thm:main_brw}]
Theorem \ref{thm:main_general} and Lemma \ref{lem:hamr_plateau} yield
\begin{align}
    \mubnhr \left[ \phi_{\tau_1} \phi_{\tau_2} \right]
    &\leq C \Exp{- c \d(\tau_1,\tau_2) + C \left( \left\lceil 2r + D \right\rceil + \log \frac{2}{b} \right)} \\
    &\leq C \Exp{- c \d(\tau_1,\tau_2) + C r}
\end{align}
after adjusting the constants.
\end{proof}
\subsection{The sinh--Gordon model}
\label{sec:init_sinhgordon}

Next we turn to Theorem \ref{thm:main_sinhgordon}, which concerns the sinh--Gordon potential
\begin{equation}
    \hambl(\phi) = \lambda \cosh\left( \sqrt{\beta} \phi \right)
\end{equation}
for any $\lambda, \beta > 0$.
As mentioned in Remark \ref{rmk:strict_convexity}, this theorem does not use the full
power of our general Theorem \ref{thm:main_general}, as evidenced by the
fact that we may take $T=1$ in Lemma \ref{lem:hambl_plateau} below.
In other words, the uniform bound suffices for this model, and the plateau-type behavior
described in Section \ref{sec:oop_shrinking} is not relevant.

\begin{lemma}
\label{lem:hambl_confining}
The potential $\hambl$ is confining in the sense of Definition \ref{def:confining}.
\end{lemma}

\begin{proof}[Proof of Lemma \ref{lem:hambl_confining}]
The symmetry and convexity hypotheses are immediate,
and the interval $J_y = [0,\infty)$ for all $y \geq 0$ since $\hambl(\phi) < \infty$
for all $\phi \in \R$.
It remains to show that the map $x \mapsto \hambl(y+x) - \hambl(y-x)$ 
is convex on this interval, for which we take two derivatives:
\begin{equation}
    \hambl''(y+x) - \hambl''(y-x)
    = \beta \lambda \cosh\left( \sqrt{\beta} (y+x) \right)
    - \beta \lambda \cosh\left( \sqrt{\beta} (y-x) \right);
\end{equation}
now since $\cosh$ is a convex symmetric function and $|y+x| \geq |y-x|$ for $x,y\geq 0$,
the right-hand side above is nonnegative, showing that $\hambl$ satisfies Definition \ref{def:confining}.
\end{proof}

Now the plateau condition is a simple application of the Brascamp--Lieb inequality \eqref{eq:bl}.

\begin{lemma}
\label{lem:hambl_plateau}
For any $\lambda, \beta > 0$ we have $\plat\left(\hambl, 1, \frac{b}{2} \min\{ 1, \lambda \beta \} \right)$.
\end{lemma}

\begin{proof}[Proof of Lemma \ref{lem:hambl_plateau}]
By Lemma \ref{lem:hambl_confining}, it suffices to bound the second derivative of the remainder term; note that we have
\begin{equation}
    \rem_1''(1) \leq \rem_1''(0) = b \cdot \walk_0^0 \left[ (\rho - \walk_0^0[\rho])^2 \right],
\end{equation}
using item \ref{item:rem_unimodal} of Proposition \ref{prop:flow}.
Now the density of $\walk_0^0$ is
\begin{equation}
    g(\rho) \propto \Exp{- \frac{\rho^2}{2} - \lambda \cosh \left( \sqrt{\beta} \rho \right)},
\end{equation}
and so the Brascamp--Lieb inequality \eqref{eq:bl} implies that
\begin{equation}
    \walk_0^0 \left[ (\rho - \walk_0^0[\rho])^2 \right]
    \leq \walk_0^0 \left[ \frac{1}{1 + \lambda \beta \cosh \left( \sqrt{\beta} \rho \right)} \right]
    \leq \frac{1}{1 + \lambda \beta},
\end{equation}
since $\cosh(x) \geq 1$ for all $x \in \R$.
Now if $\lambda \beta \leq 1$ then the above is at most $1 - \frac{\lambda \beta}{2}$,
and otherwise it is at most $\frac{1}{2}$.
So we find that
\begin{equation}
    \rem_1''(1) \leq b \max \left\{ \frac{1}{2}, 1 - \frac{\lambda \beta}{2} \right\}
    = b - \min \left\{ \frac{b}{2}, \frac{b \lambda \beta}{2} \right\},
\end{equation}
which finishes the proof.
\end{proof}

\begin{proof}[Proof of Theorem \ref{thm:main_sinhgordon}]
Theorem \ref{thm:main_general} and Lemma \ref{lem:hambl_plateau} yield
\begin{align}
    \mubnhbl \left[ \phi_{\tau_1} \phi_{\tau_2} \right]
    &\leq C \Exp{- c \d(\tau_1, \tau_2) + C \left( 1 + \log \frac{1}{\frac{b}{2} \min \{1,\lambda\beta\}} \right)} \\
    &\leq C \Exp{- c \d(\tau_1,\tau_2) + C \log_+ \frac{1}{\lambda \beta}}
\end{align}
after adjusting the constants, recalling that $\log_+ x = \max\{0,\log x\}$.
\end{proof}

\subsection{Power potentials}
\label{sec:init_polynomial}

Finally we turn to Theorem \ref{thm:main_phialpha}, which concerns the power potentials
\begin{equation}
    \hamal(\phi) = \lambda |\phi|^\alpha
\end{equation}
for $\lambda > 0$ and $\alpha \geq 2$.
These potentials interpolate between the strict convexity of the sinh--Gordon model
and the square well of the conditioned model, as discussed in Section \ref{sec:intro_fieldtheories}.

\begin{lemma}
\label{lem:hamal_confining}
The potential $\hamal$ is confining in the sense of Definition \ref{def:confining}.
\end{lemma}

\begin{proof}[Proof of Lemma \ref{lem:hamal_confining}]
The symmetry and convexity hypotheses are immediate, and for all $y \geq 0$ the set $J_y$ is
the full interval $[0,\infty)$.
Now the second derivative of the function $x \mapsto \hamal(y + x) - \hamal(y - x)$ is
\begin{equation}
    \lambda \alpha(\alpha-1) \left( \left| y + x \right|^{\alpha-2} - \left| y - x \right|^{\alpha-2} \right),
\end{equation}
and since $\lambda > 0$ and $\alpha \geq 2$, the latter is nonnegative for all $x,y \geq 0$
since $|y+x| \geq |y-x|$.
\end{proof}

Now to verify the quantitative part of the plateau condition of Definition \ref{def:plateau}
for the power potential, we consider two separate cases.
First, if $\lambda \geq 1$, then although the potential is not strictly convex, an easy calculation
using the gamma function (an idea suggested by a publicly available large language model
in August 2026) shows that we may take $T = 1$ as was the case for the sinh--Gordon model.

\begin{lemma}
\label{lem:hamal_plateau_big}
For any $\alpha \geq 2$ and $\lambda \geq 1$, we have
$\plat\left( \hamal, 1, \frac{b}{2} \right)$.
\end{lemma}

\begin{proof}[Proof of Lemma \ref{lem:hamal_plateau_big}]
It suffices to bound $\rem_1''(0)$,
using the unimodality of item \ref{item:rem_unimodal} in Proposition \ref{prop:flow}.
By Lemma \ref{lem:tailcomparison}, and since $\walk_0^0[\rho] = 0$ by symmetry,
the variance under $\walk_0^0$ can only increase if we remove the \emph{quadratic} term, i.e.\ we have
\begin{equation}
\label{eq:gammastart}
    \walk_0^0 \left[ (\rho - \walk_0^0[\rho])^2 \right]
    \leq \frac{\int_{-\infty}^\infty \rho^2 e^{- \lambda |\rho|^\alpha \,d\rho}}
    {\int_{-\infty}^\infty e^{-\lambda |\rho|^\alpha} \,d\rho} \\
    = \frac{\int_0^\infty \rho^2 e^{- \lambda \rho^\alpha} \,d\rho}
    {\int_0^\infty e^{- \lambda \rho^\alpha} \,d\rho}.
\end{equation}
Now we claim that the latter can be expressed via a ratio of values of the gamma function
\begin{equation}
    \Gamma(z) = \int_0^\infty t^{z-1} e^{-t} \,dt.
\end{equation}
Indeed, after making the change of variables $t = \lambda \rho^\alpha$,
the ratio in the right-hand side of \eqref{eq:gammastart} becomes
\begin{equation}
    \frac{\int_0^\infty \left( \frac{t}{\lambda} \right)^{\frac{2}{\alpha}} e^{- t} \frac{dt}{\alpha \lambda \left( \frac{t}{\lambda} \right)^{\frac{\alpha-1}{\alpha}}}}
    {\int_0^\infty e^{-t} \frac{dt}{\alpha \lambda \left( \frac{t}{\lambda} \right)^{\frac{\alpha-1}{\alpha}}}}
    = \lambda^{- \frac{2}{\alpha}} \frac{\int_0^\infty t^{\frac{2}{\alpha} - \frac{\alpha - 1}{\alpha}} e^{-t} \,dt}
    {\int_0^\infty t^{- \frac{\alpha-1}{\alpha}} e^{-t} \,dt}
    = \lambda^{-\frac{2}{\alpha}} \frac{\Gamma\left( \frac{3}{\alpha} \right)}{\Gamma \left( \frac{1}{\alpha} \right)}.
\end{equation}
It may now be easily checked, numerically or otherwise, that the latter ratio of gamma function values is at most $\frac{1}{2}$
for all $\alpha \geq 2$.
So since $\lambda \geq 1$, we find that the variance under $\walk_0^0$ is at most $\frac{1}{2}$, and so $\rem_1''(T) \leq \frac{b}{2}$
for all $T \in \N$ in this case.
\end{proof}

Now for the case of small $\lambda$, we will need to choose a larger value of $T$,
and this will be done by carefully balancing the choice of $T \in \N$ against $\eps > 0$,
which was not necessary for $\hamr$ in Section \ref{sec:init_smallball}.
We will apply the Brascamp--Lieb inequality again, but to get an effective bound we will need
to increase $T$ to skew the distribution towards regions where the second derivative 
is larger.
To achieve this, we will bound the \emph{mode} of the distribution and apply
a Gaussian-type tail bound for log-concave distributions due to \cite{DLR09Maximum,DLR14Maximum}.

Since this can get somewhat technical, we do not give a uniform bound for all $\alpha \geq 2$,
but as mentioned in Remark \ref{rmk:alphabehavior},
we expect that a more careful analysis may be able to remove this restriction.

\begin{lemma}
\label{lem:hamal_plateau_small}
For any $\alpha_* \geq 2$ there is a constant $D_{\alpha_*}$ such that
for all $\alpha \in [2,\alpha_*]$ and $\lambda \in (0,1)$, we have
\begin{equation}
    \plat \left( \hamal,
    \left\lceil \lambda^{- \frac{\alpha-2}{\alpha^2}} + D_{\alpha_*} \right\rceil, 
    \frac{b}{4} \lambda^{\frac{4\alpha-4}{\alpha^2}} \right).
\end{equation}
\end{lemma}

\begin{proof}[Proof of Lemma \ref{lem:hamal_plateau_small}]
For any $T \in \N$, the density of $\walk_0^T$ is
\begin{equation}
    g(\rho) \propto \Exp{- \frac{(\rho-T)^2}{2} - \lambda |\rho|^\alpha}.
\end{equation}
So, by the Brascamp--Lieb inequality \eqref{eq:bl}, we have
\begin{equation}
    \walk_0^T \left[ (\rho - \walk_0^T[\rho])^2 \right]
    \leq \walk_0^T \left[ \frac{1}{1 + \lambda \alpha (\alpha-1) |\rho|^{\alpha-2}} \right]
    = 1 - \walk_0^T \left[ \frac{\lambda \alpha (\alpha-1) |\rho|^{\alpha-2}}{1 + \lambda \alpha (\alpha-1) |\rho|^{\alpha-2}} \right],
\end{equation}
and we will lower bound the latter expectation.
Now since $x \mapsto \frac{x}{1 + x}$ is an increasing function of $x \geq 0$, for any $\rho_0 > 0$ we have
\begin{equation}
\label{eq:wkinter}
    \walk_0^T \left[ \frac{\lambda \alpha (\alpha-1) |\rho|^{\alpha-2}}{1 + \lambda \alpha (\alpha-1) |\rho|^{\alpha-2}} \right]
    \geq \walk_0^T \left[ \rho \geq \rho_0 \right] \cdot \frac{\lambda \alpha (\alpha-1) \rho_0^{\alpha-2}}
    {1 + \lambda \alpha (\alpha-1) \rho_0^{\alpha-2}}.
\end{equation}
To find an appropriate choice of $\rho_0$, we will analyze the density $g$.
In particular, its maximizer is the minimizer of $\frac{(\rho - T)^2}{2} + \lambda |\rho|^\alpha$.
If $T \geq 0$ then the maximizer is also nonnegative since $g(\rho) \geq g(-\rho)$ for $\rho \geq 0$.
Thus, taking a derivative, the maximizer $\rho_*$ of $g$ must satisfy
\begin{equation}
\label{eq:rhostareq}
    \rho_* - T + \lambda \alpha \rho_*^{\alpha-1} = 0,
    \qquad \text{i.e.} \qquad
    \rho_* = T - \lambda \alpha \rho_*^{\alpha-1}
\end{equation}
Let us fix a constant $q$ to be determined later.
Now if $0 < \rho < \left( \lambda \alpha \right)^{- \frac{\alpha-2}{\alpha^2}} + q$,
then we have
\begin{align}
    T - \lambda \alpha \rho^{\alpha-1}
    &> T - \lambda \alpha \left( \left( \lambda \alpha \right)^{-\frac{\alpha-2}{\alpha^2}} + q \right)^{\alpha-1} \\
    &= T - \left( \lambda \alpha \right)^{1 - \frac{(\alpha-2)(\alpha-1)}{\alpha^2}}
        \left( 1 + q \left( \lambda \alpha \right)^{\frac{\alpha-2}{\alpha^2}} \right)^{\alpha-1} \\
    &\geq T - \left( \lambda \alpha \right)^{\frac{3\alpha-2}{\alpha^2}}
        \Exp{q (\alpha-1) \left( \lambda \alpha \right)^{\frac{\alpha-2}{\alpha^2}}} \\
    &\geq T - \alpha^{\frac{3\alpha-2}{\alpha^2}} \Exp{q (\alpha-1) \alpha^{\frac{\alpha-2}{\alpha^2}}},
\end{align}
using at the last step our assumption that $\lambda \leq 1$
and the fact that $\alpha \geq 2$.
Now let us set
\begin{equation}
    D_{\alpha_*} = \max \left\{ \alpha^{\frac{3\alpha-2}{\alpha^2}} \Exp{q (\alpha-1) \alpha^{\frac{\alpha-2}{\alpha(\alpha-1)}}} : \alpha \in [2, \alpha_*] \right\},
\end{equation}
which is finite for any fixed $\alpha_*$, so that if $\alpha \in [2,\alpha_*]$ and
\begin{equation}
\label{eq:Tassumption}
    T > \left( \lambda \alpha \right)^{-\frac{\alpha-2}{\alpha^2}} + q + D_{\alpha_*},
\end{equation}
then we must have $\rho_* \geq \left( \lambda \alpha \right)^{-\frac{\alpha-2}{\alpha^2}} + q$ by \eqref{eq:rhostareq}.
Now we set $\rho_0 = \left( \lambda \alpha \right)^{- \frac{\alpha-2}{\alpha^2}}$
which satisfies $\rho_0 \leq \rho_* - q$ under the assumption \eqref{eq:Tassumption}.
We will now adopt this assumption and apply the following tail bound for log-concave
distributions due to \cite[Lemma 7.1]{DLR14Maximum} (see also \cite{DLR09Maximum}):
\begin{equation}
\label{eq:probbound}
    \walk_0^T \left[ \rho \geq \rho_0 \right] 
    \geq \walk_0^T \left[ \rho \geq \rho_* - q \right]
    \geq 1 - e^{- g(\rho_*) q}.
\end{equation}
Since $(-\log g)'' (\rho) \geq 1$ for all $\rho$, and since $\rho_*$ is the mode of $g$
which is a probability density, we must have $g(\rho_*) \geq \frac{1}{\sqrt{2\pi}}$ via comparison
with the standard Gaussian density.
Thus we may choose $q$ to be an absolute constant large enough that the right-hand side above is $\geq \frac{1}{2}$.
Furthermore, for this choice of $\rho_0$ we have
\begin{align}
    \lambda \alpha (\alpha-1) \rho_0^{\alpha-2}
    &= (\alpha-1) \left( \lambda \alpha \right)^{1 - \frac{(\alpha-2)^2}{\alpha^2}} \\
    &= (\alpha-1) \left( \lambda \alpha \right)^{\frac{\alpha^2 - \alpha^2 + 4 \alpha - 4}{\alpha^2}} \\
    &= (\alpha-1) \left( \lambda \alpha \right)^{\frac{4 \alpha - 4}{\alpha^2}} \\
    &\geq \lambda^{\frac{4 \alpha - 4}{\alpha^2}},
\label{eq:ratio_bound}
\end{align}
using the fact that $\alpha \geq 2$ and that $\alpha^{\frac{4\alpha-4}{\alpha^2}} \geq 1$ for all $\alpha \geq 2$
(as may easily be checked numerically or otherwise).
So from \eqref{eq:wkinter} with \eqref{eq:probbound} and \eqref{eq:ratio_bound},
we find that under the assumption \eqref{eq:Tassumption} on $T$ we have
\begin{equation}
    \walk_0^T \left[ (\rho - \walk_0^T[\rho])^2 \right]
    \leq 1 - \frac{1}{2} \cdot \frac{\lambda^{\frac{4\alpha-4}{\alpha^2}}}{1 + \lambda^{\frac{4\alpha-4}{\alpha^2}}}
    \leq 1 - \frac{1}{4} \lambda^{\frac{4\alpha-4}{\alpha^2}},
\end{equation}
using again our assumption that $\lambda \leq 1$.
The assumption \eqref{eq:Tassumption} on $T$ may be satisfied by $T$ of the form $\left\lceil \lambda^{-\frac{\alpha-2}{\alpha^2}} + D_{\alpha_*} \right\rceil$,
after adjusting the constant $D_{\alpha_*}$ and using the fact that $\alpha^{-\frac{\alpha-2}{\alpha^2}} \leq 1$ for all $\alpha \geq 2$
(as again may be easily checked).
So
\begin{equation}
    \plat \left( \hamal,
    \left\lceil \lambda^{- \frac{\alpha-2}{\alpha^2}} + D_{\alpha_*} \right\rceil, 
    \frac{b}{4} \lambda^{\frac{4\alpha-4}{\alpha^2}} \right)
\end{equation}
holds as desired.
\end{proof}

\begin{proof}[Proof of Theorem \ref{thm:main_phialpha}]
Combining the cases of $\lambda \geq 1$ and $\lambda \in (0,1)$ from Lemmas \ref{lem:hamal_plateau_big}
and \ref{lem:hamal_plateau_small}, we find that
\begin{equation}
    \plat \left( \hamal, \left\lceil \lambda^{- \frac{\alpha-2}{\alpha^2}} + D_{\alpha_*} \right\rceil, \min \left\{ \frac{b}{4} \lambda^{\frac{4\alpha-4}{\alpha^2}}, \frac{b}{2} \right\} \right)
\end{equation}
holds.
Thus Theorem \ref{thm:main_general} yields
\begin{align}
    \mubnhal \left[ \phi_{\tau_1} \phi_{\tau_2} \right]
    &\leq C \Exp{- c \d(\tau_1,\tau_2) + C \left( \lambda^{-\frac{\alpha-2}{\alpha^2}} + D_{\alpha_*} + \log \frac{1}{\min \left\{ \frac{b}{4} \lambda^{\frac{4\alpha-4}{\alpha^2}}, \frac{b}{2} \right\}} \right)} \\
    &\leq C \Exp{- c \d(\tau_1,\tau_2) + C \left( \lambda^{-\frac{\alpha-2}{\alpha^2}} + \frac{1}{\alpha} \log_+ \frac{1}{\lambda} + D_{\alpha_*} \right)},
\end{align}
after adjusting the constants, since $\frac{4\alpha-4}{\alpha} \leq 4$ for all $\alpha \geq 2$,
and recalling that $\log_+ x = \max\{0,\log x\}$.
\end{proof}

\begin{remark}
\label{rmk:nonoptimal_power}
From the argument of Lemma \ref{lem:hamal_plateau_small}, we see that the power $\frac{\alpha-2}{\alpha^2}$ may be
replaced by any expression $f(\alpha)$ such that $\alpha^{-f(\alpha)}$ is bounded
above and $\alpha^{1 - (\alpha-2)f(\alpha)}$ is bounded below
for all $\alpha \geq 2$, at the cost of changing the constant $D_{\alpha_*}$
and the coefficient in front of the logarithm term.
We leave the problem of finding the true optimal bound here to future work.
\end{remark}

\bibliographystyle{plain}
\bibliography{references}

\end{document}